\documentclass[aap]{imsart}
\usepackage{epsfig,makeidx,amsmath,amsfonts,latexsym,subfigure, amsthm, lipsum}
\usepackage{times,pifont}
\usepackage{tikz}

\usepackage{graphicx}
\usepackage{caption} %%% optional but probably necessary
\usepackage[utf8]{inputenc} % allow utf-8 input
\usepackage[T1]{fontenc}    % use 8-bit T1 fonts
\usepackage{hyperref}       % hyperlinks
\usepackage{url}            % simple URL typesetting
\usepackage{booktabs}       % professional-quality tables
\usepackage{amsfonts}       % blackboard math symbols
\usepackage{nicefrac}       % compact symbols for 1/2, etc.
\usepackage{microtype}      % microtypography
\usepackage{lipsum}

\usepackage{tikz}
\usepackage{amsmath, amssymb}
\usetikzlibrary{shapes,decorations.pathreplacing}
\usetikzlibrary{arrows.meta, decorations.pathreplacing}

\usepackage{comment, amsmath,amssymb,amsthm,float, mathtools}

\RequirePackage{amsthm,amsmath,amsfonts,amssymb}

\startlocaldefs
\theoremstyle{plain}
\newtheorem{theorem}{Theorem}[section]
\newtheorem{proposition}[theorem]{Proposition}
\newtheorem{lemma}[theorem]{Lemma}
\newtheorem{corollary}[theorem]{Corollary}
\newtheorem{claim}[theorem]{Claim}
\theoremstyle{remark}
\newtheorem{remark}[theorem]{Remark}
\newtheorem{example}[theorem]{Example}
\newtheorem{definition}[theorem]{Definition}

\newcommand{\lab}{\label}

\def\smallpar{\smallbreak\@afterindentfalse\@afterheading\ignorespaces}

\def\quarter{{\textstyle{1\over4}}}
\def\ent{{ \rm \sc ENT}}

\def \f {{\mathcal H}}

\newcommand{\E}{{\mathbf E}}

\newcommand{\D}{{\mathcal D}}
\newcommand{\M}{{\mathcal A}}

\newcommand{\A}{{\mathcal A}}
\newcommand{\B}{{\mathcal B}}

\newcommand{\ENT}{{\rm ENT}}

\newcommand{\Z}{{\bf Z}}
\newcommand{\T}{{\mathcal T}}

\newcommand{\one}{{\mathbf 1}}

\newcommand{\bi}{\begin{itemize}}
\newcommand{\ei}{\end{itemize}}
\newcommand{\be}{\begin{enumerate}}
\newcommand{\ee}{\end{enumerate}}

\newcommand{\sgn}{{\operatorname{sgn}}}

\newcommand{\beq}{\begin{equation}}
\newcommand{\eeq}{\end{equation}}
\newcommand{\beqa}{\begin{eqnarray*}}
\newcommand{\eeqa}{\end{eqnarray*}}
\newcommand{\btm}{\begin{theorem}}
\newcommand{\etm}{\end{theorem}}
\newcommand{\bpf}{\begin{proof}}
\newcommand{\epf}{\end{proof}}
\newcommand{\bla}{\begin{lemma}}
\newcommand{\ela}{\end{lemma}}
\newcommand{\bdn}{\begin{definition}}
\newcommand{\edn}{\end{definition}}
\newcommand{\bpn}{\begin{proposition}}
\newcommand{\epn}{\end{proposition}}
\newcommand{\bcy}{\begin{corollary}}
\newcommand{\ecy}{\end{corollary}}

\newcommand\Abs[1]{\left \Vert #1 \right \|}
\newcommand\abs[1]{\left \Vert #1 \right \|}
\newcommand\bbE{\mathbb E}
\newcommand\bbP{\mathbb P}
\newcommand\bbT{\mathbb T}

\newcommand\id{\operatorname{id}}
\renewcommand{\vdots}{\vbox{\baselineskip=10pt \lineskiplimit=0pt \kern3pt \hbox{.}\hbox{.}\hbox{.}}}

\def \V {{\mathcal V}}

\def \u {{ \cal U}}

\def\ldotsplus{\mathinner{\ldotp\ldotp\ldotp\ldotp}}
\def\fourdots{\relax\ifmmode\ldotsplus\else$\m@th \ldotsplus\,$\fi}
\def\bigomega{{\Omega}}
\def\bigo{{\rm O}}
\def\id{{\rm id}}
\def \tc {{3-collision }}

\def \sexp#1{{\scalebox{.9}{$\scriptstyle \leq #1$}}}
\def \rsexp#1{{\scalebox{.7}{$\scriptstyle \leq #1$}}}

\def \etesc#1{{\tau_{w}^{#1}}}
\def\sfrac#1#2{{\textstyle{#1 \over #2}}}
\def\half{{\textstyle{1\over2}}}
\def\eighth{{\textstyle{1\over8}}}

\def \one {{\mathbf 1}}
\def \e {{\mathbf E}}

\def\var{{\rm  \bf Var}}

\def\P {{ \mathbb P}}
\def\p {{ \mathbb P}}
\def\given {{\,|\,}}

\def\law{{\cal L}}
\def\given {{\,|\,}}

\def \tail {{\mathcal T}}
\def \tailt {{\mathcal {\widetilde T}}}
\def \Et {{E}}
\def \bbE {{\e}} %{{\E}}
\def \bbP {{\P}}

\def \lambdam {{\lambda_k^A}}
\def \lambdahm {{\widehat{\lambda}_k^A}}
\def \plaw {{\mathcal P}}
\def \qlaw {{\mathcal Q}}
\def \ut {{\mathcal U_T}}

\def \boxi {{\mathcal B_i}}
\def \boxj {{\mathcal B_j}}
\def \boxk {{\mathcal B_k}}

\def \li {{\mathcal L_i}}
\def \lj {{\mathcal L_i}}
\def \lk {{\mathcal L_i}}
\def \tilet {{\mathbb T}}
\def \boxl {{B_\ell}}

\def \R {{\mathbf R}}
\def \ntol {{N_{\text{tol}}}}
\def \shat {{\widehat S}}
\def \effi {{\mathcal F}_i}
\def \effj {{\mathcal F}_j}
\def \effk {{\mathcal F}_k}
\def \tesc {{\tau_{w}}}
\def \po {{\overline{p}}}
\def \ellp {{\bar l}}
\def \pt {{\pi_{(t)}}}
\def \pim {{\pi_{(m)}}}
\def \pimin {{\pi^{-1}_{(m)}}}
\def \pmp {{\pi_{(m+1)}}}

\def \kss {{ k^*}}
\def \ptt {{\widetilde P}}

\def \tilrho {{\widetilde \rho}}
\def \rhot {{\tilrho}}

\def \btt {{\widehat B}}

\def \tc {{3-cycle }}
\def \tcoll {{3-collision }}
\def \tcolls {{3-collisions }}
\def \seta {{\mathcal A}}
\def \aset {{\mathcal A}}
\def \cset {{\mathcal C}}

\def \pis {{\mu}}

\def \btil {{\widetilde \beta}}
\def \tal {{\widetilde \alpha}}
\def \sym {{\rm S}}
    
\def \Sset {{\mathcal W}}    
\def \T {{\mathcal H}}    

\def \that {{\T^{\overline B}_{m+1}}}

\def \ttil {{\T^B_{m+1}}}

\def \c {{Z}}

\def \piu {{\pi_u}}

\def \Boy {{\D}}
\def \bcomp{{\overline B}}

\def \yY{{U}}
\def \a {{r}}
\def \b {{s}}
\def \h {{h_m}}
\def \hmp {{h_{m+1}}}
\def \Tail {{\tt T}}
\def \Crw {{\varpi}}
\def \C {{\mathbf C}}
\def \Cbe {{C_B}}
\def \Sigmam {{\mathbf \Sigma}}
\endlocaldefs
\def \K {{\mathcal M}}
\def \fK {{\mathcal N}}
\def \re {{\,||\,}}

\begin{document}

\begin{frontmatter}
  \title{Mixing time of the torus shuffle}
%\title{A sample article title with some additional note\thanksref{T1}}
%\runtitle{???}
%\thankstext{T1}{A sample of additional note to the title.}

\begin{aug}
%%%%%%%%%%%%%%%%%%%%%%%%%%%%%%%%%%%%%%%%%%%%%%%
%% Only one address is permitted per author. %%
%% Only division, organization and e-mail is %%
%% included in the address.                  %%
%% Additional information can be included in %%
\author{\fnms{Olena}~\snm{Blumberg}\ead[label=e1]{square25@gmail.com}},
\author[B]{\fnms{Ben}~\snm{Morris}\ead[label=e2]{morris@math.ucdavis.edu}}
\and
\author{\fnms{Alto}~\snm{Senda}\ead[label=e3]{aesenda@ucdavis.edu}}
%%%%%%%%%%%%%%%%%%%%%%%%%%%%%%%%%%%%%%%%%%%%%%
%% Addresses                                %%
%%%%%%%%%%%%%%%%%%%%%%%%%%%%%%%%%%%%%%%%%%%%%%
\address{\printead[presep={}]{e1}}
\address[B]{Department of Mathematics,
  University of California, Davis\printead[presep={,\ }]{e2}}
\address{\printead[presep={}]{e3}}

\end{aug}
\begin{abstract}
  In 2009, Morris studied a class of card shuffles that can be reduced
  to a version of the game three-card Monte on a deck of $n$ cards.
  In the present paper, we extend the technique and analyze shuffles that can be reduced to
  what we call $3$-Monte shuffles. The main feature of a 3-Monte shuffle
is that it uses what we call 3-collisions, 
which are
  permutations that are equally likely to be a certain fixed 3-cycle or the identity.

We prove a theorem that reduces bounding the mixing time of a 3-Monte
shuffle to verifying a condition that involves only triplets of cards.
Then we use it to analyze a model of card shuffling  we call the torus shuffle, which was introduced by
Diaconis in 1988.

In the torus shuffle, 
cards are arranged in an $n$ by $n$ grid. The transition rule is the following. In each step,
choose a row or column, uniformly at random,  and cyclically rotate it by one unit in a uniform
random direction. Diaconis conjectured that the mixing time is $\bigo(n^3 \log n)$.
We obtain a bound that is within a poly log factor of the conjecture.
\end{abstract}



\end{frontmatter}
%%%%%%%%%%%%%%%%%%%%%%%%%%%%%%%%%%%%%%%%%%%%%%
%% Please use \tableofcontents for articles %%
%% with 50 pages and more                   %%
%%%%%%%%%%%%%%%%%%%%%%%%%%%%%%%%%%%%%%%%%%%%%%
%\tableofcontents

%%%%%%%%%%%%%%%%%%%%%%%%%%%%%%%%%%%%%%%%%%%%%%
%%%% Main text entry area:
\section{Introduction}

In the open problems section of his book \cite{diaconis}, 
Diaconis introduced the following variant of the classic fifteen puzzle. Cards (hereafter called {\it tiles})
are arranged in
an $n \times n$ grid. In each step,
choose a row or column, uniformly at random, 
and cyclically rotate it in a uniform random direction. (Or, equivalently, choose a random tile and
move it up, down, left or right, moving all the other tiles in the same row or column, as appropriate,  along with it.) 
Figure \ref{moveright} below shows a typical move of the shuffle, with the second row being moved to the right.
\begin{figure}[t]
  \centering
\begin{tikzpicture}
%first sequence
\draw[step=0.5cm,color=gray, xshift=-0.5cm, yshift=-0.0cm] (-1,-1) grid (1,1);
\node at (-1.25,+0.75) {1};
\node at (-0.75,+0.75) {4};
\node at (-0.25,+0.75) {9};
\node at (+0.25,+0.75) {16};
\node at (-1.25,+0.25) {2};
\node at (-0.75,+0.25) {3};
\node at (-0.25,+0.25) {8};
\node at (+0.25,+0.25) {15};
\node at (-1.25,-0.25) {5};
\node at (-0.75,-0.25) {6};
\node at (-0.25,-0.25) {7};
\node at (+0.25,-0.25) {14};
\node at (-1.25,-0.75) {10};
\node at (-0.75,-0.75) {11};
\node at (-0.25,-0.75) {12};
\node at (+0.25,-0.75) {13};
\node at (-1.70,+0.25) {$\to$};
\node at (+0.70,+0.25) {$\to$};
\node at (+1.3, 0) {$\implies$};
\draw[step=0.5cm,color=gray] (2,-1) -- (2,1);
\draw[step=0.5cm,color=gray] (2,-1) grid (4,1);
\node at (+2.25,+0.75) {1};
\node at (+2.75,+0.75) {4};
\node at (+3.25,+0.75) {9};
\node at (+3.75,+0.75) {16};
\node at (+2.25,+0.25) {15};
\node at (+2.75,+0.25) {2};
\node at (+3.25,+0.25) {3};
\node at (+3.75,+0.25) {8};
\node at (+2.25,-0.25) {5};
\node at (+2.75,-0.25) {6};
\node at (+3.25,-0.25) {7};
\node at (+3.75,-0.25) {14};
\node at (+2.25,-0.75) {10};
\node at (+2.75,-0.75) {11};
\node at (+3.25,-0.75) {12};
\node at (+3.75,-0.75) {13};
\end{tikzpicture}
  \caption{A possible move in the torus shuffle in the $4 \times 4$ case. Here, the second row is moved to the right}
  \label{moveright}
\end{figure}
This Markov chain is now the basis of a game for iPhone called the {\it torus puzzle} and we shall call the chain the
{\it torus shuffle.}
Diaconis observed that $\bigo(n^3)$ moves are enough to randomize a single tile
and conjectured that $\bigo(n^3 \log n)$ moves are enough to randomize the entire puzzle.

    The torus shuffle has been hard to analyze. It is a random walk on a group,
and it is easy to see that four moves of the torus shuffle can be used to generate certain  
cycles of length 3 involving adjacent tiles (see Section \ref{entmaintheorem} for more details).
These 3-cycles can in turn be used to generate arbitrary pairs of transpositions. 
Thus, a first attempt to bound the mixing time would be to use the comparison techniques for random walks on groups
developed by Diaconis and Saloff-Coste in \cite{rwg} and compare the torus shuffle to
(two steps of) shuffling by random transpositions (see \cite{diashah}). However, since
order $n$ moves of the torus shuffle are
required to simulate a typical pair of transpositions, this would result in an upper bound of $\bigo(n^4 \log n)$,
a factor of $n$ higher than the conjecture. 

In the present paper, we 
give the first detailed analysis of the torus shuffle,
and prove that the mixing time is $\bigo(n^3 \log^3 n)$, which is a within a poly log factor of Diaconis's conjecture.
We accomplish this by first generalizing the entropy-based technique introduced in \cite{morris} to
handle shuffles that can be defined using {\it 3-collisions}, that is, random permutations whose distribution is
an equal mixture of some fixed 3-cycle and the identity. For such shuffles, we prove a theorem that reduces
bounding the mixing time to verifying a condition involving triplets of cards. We believe this theorem is
of independent interest. Indeed, it has already been used by the first two authors and Hans Oberschelp \cite{bmo} to prove an
upper bound (tight to within a polylogarithmic factor) for the so-called overlapping cycles shuffle that was introduced
by Jonasson \cite{jonasson}
and for which  Angel, Peres and Wilson \cite{apw} studied the motion of a single card.   

The remainder of this paper is organized as follows.
In Section 
\ref{bk} we give some necessary background on entropy,
including the notion of conditional relative entropy.
In Section \ref{collis} we give a formal definition of 3-collisions and    
define {\it 3-Monte shuffles}, the general model of card shuffling
to which our main theorem will apply.
In section \ref{entsec} we prove the main
technical lemma of the paper, which relates to 
how multiplying a random permutation by a product of 3-collisions involving mutually
disjoint 3-cycles reduces its 
relative entropy. 
In Section \ref{entmaintheorem} we prove the main theorem about the mixing time of a 3-Monte shuffle.
In Section \ref{fitsin} we give some background on the torus shuffle.
%and hence the main theorem applies. 
In Section \ref{threestage} we prove a technical lemma involving the motion
of three tiles in the torus shuffle. 
Finally, in Section \ref{payoff} we use the technical lemma and the main theorem
about 3-Monte shuffles 
to prove an upper bound for the mixing time of the torus shuffle.

\section{Background} 
\label{bk}
In this section we give some basic definitions and review some of the parts of \cite{morris} that are needed here.
We start with a formal definition of the mixing time. 
Let $X_t$
be a Markov chain on a finite state space $V$
with a uniform stationary distribution.
For probability measures $\mu$ and $\nu$ on $V$,
define the total variation distance
$|| \mu - \nu || := \half \sum_{x \in V} |\mu(x) - \nu(x) |$, 
and 
define the mixing time
\begin{equation}
\label{mixingtime}
T_{\rm mix} := \min \{t: || \P(X_t = \cdot) - \u(V) || \leq \quarter \} \,,% \mbox{ for all $x \in V$}\} \;,
\end{equation}
where $\u(V)$ denotes the uniform distribution over the set $V$.

In the present paper we use the convention that a permutation
$\pi$ represents the order of cards in a deck, with
$\pi(i) = \mbox{position of card $i$}$.
We shall consider 
position $1$ to be the top of the deck and
position $n$ to be the bottom. 

We shall now give some background on 
relative entropy and the chain rule. The remaining content of this section 
is closely related to Chapter 2 of \cite{covthomas}, specialized to 
the setting of random permutations. 

For   
probability distributions
$a = (a_i)_{i \in V}$ and
$b = (b_i)_{i \in V}$ on $V$, define the
{\it relative entropy} of $a$ with respect to $b$ by
\[
  D(a \re b) := \sum_{i \in V} a_i \log {a_i \over b_i} \;,
\]
where $0 \log 0 := 0$.
We are most interested in the case where $b$ is the uniform distribution over $V$. 
\begin{definition}
  \label{relent}
If  $p := (p_i)_{i \in V}$ is a probability distribution on $V$, then
we denote by $\ent(p)$ the relative entropy of $p$ with respect to $\u(V)$,
that is, 
\[
  \ent(p) := \sum_{i \in V} p_i \log (|V| p_i) \;,
\]
and we will call $\ent(p)$ simply the relative entropy of $p$.
\end{definition}
\begin{comment}
Note that
\[
  \ent(p) = H(\u(V)) - H(p) \;.
\]
\end{comment}
%If $Y$ is a random variable taking values in $V$ then we write $\ent(Y)$ for the relative
%entropy of the distribution of $Y$. 

%Let $\u(V)$ denote the uniform distribution over $V$.

Pinsker's inequality links relative entropy to total variation distance:
\[
  || a - b ||%_{\rm TV}
  \leq \sqrt{ \half D(a \re b) } \;.
\]
In the case where $b$ is the uniform distribution over $V$,
Pinsker's inequality gives
\begin{equation}
\label{totent}
|| p - \u(V) ||%_{\rm TV}
\leq \sqrt{ \half \ent(p)} \;.
\end{equation}
If $X$
is a random variable %(or random permutation) 
with finitely many possible values, 
define $\ent(X)$ to be the relative entropy of $\law(X)$. %, with a similar definition for $H(X)$. 
Note that if $\P(X = i) = p_i$  for $i \in V$ then
$\ent(X) = \e(\log (|V| p_X))$. \\
~\\
We will need the following two lemmas.

\begin{lemma}
  \label{convy} If $p^1, \dots, p^m$ are probability distributions on a finite set $V$ and
  \[
    p = \sum_{j=1}^m w_j p^j,
    \]
where $w_j \geq 0$ and $\sum_{j=1}^m w_j = 1$, then
\[
\ent(p) \leq \sum_{j=1}^m w_j \ent(p^j) \;.
\]
\end{lemma}
\begin{proof}
The lemma is a special case of Theorem 2.7.2 of \cite{covthomas}.
\end{proof}

%\begin{corollary}
%  \label{convycor}
%  Suppose $\pi^1, \pi^2, \dots, \pi^m$ are random permutations in $\sym_n$ with distributions
%  $p^1, p^2, \dots, p^m$, respectively.
%Suppose that
%  \[
%    p = \sum_{j=1}^m w_j p^j,
%    \]
%where $w_j \geq 0$ and $\sum_{j=1}^m w_j = 1$, 
%and suppose that $\pi$ is a random permutation in $\sym_n$
%with disribution $p$. Suppose that 
%$\pi^1, \pi^2, \dots, \pi^m$ are all equally likely to be odd or even.
%Then
%\[
%  \e( \ent( \pi \given \sgn( \pi))) \leq \sum_{j=1}^m w_j \e( \ent (\pi^j \given \sgn(\pi^j))) \;.
%\]
%\end{corollary}
%\begin{proof}
%  Since $\pi^1, \pi^2, \dots, \pi^m$ are all equally likely to be odd or even,
%and therefore  $\pi$ is also 
%equally likely to be odd or even,
%we have
%  \begin{align}
%    \e (\ent( \pi^j \given \sgn(\pi^j))) = \ent(\pi^j) \;\;\;\mbox{for $j$ with $1 \leq j \leq m;$}  \\
%    \e (\ent( \pi \given \sgn(\pi))) = \ent(\pi),
%  \end{align}
%  by Remark \ref{simplify}. Combining this with Lemma \ref{convy} yields the corollary. 
%\end{proof}

\begin{lemma}
  \label{determ}
  If $\pi$ is a random permutation in $\sym_n$
%that is equally likely to be odd or even
  and $\nu_0$ is a deterministic permutation in $\sym_n$, 
then $\ent(\pi) = \ent(\pi \nu_0)$.
\end{lemma}
\begin{proof}
We have
\begin{eqnarray*}
  \ent(\pi \nu_0)     
  &=&
      \sum_{\lambda \in \sym_n}
      \P( \pi \nu_0 = \lambda) \log  \left(  n! \P( \pi \nu_0 = \lambda  )   \right)
  \\
&=&      \sum_{\lambda \in \sym_n}
  \P( \pi  = \lambda  \nu_0^{-1}   ) \log  \left(   n! \P( \pi = \lambda \nu_0^{-1}  \right)   \\
&=&  \ent( \pi) \;,
\end{eqnarray*}
where the last line holds because if $\lambda$ ranges over all permutations in $\sym_n$ then so does $\lambda \nu_0^{-1}$. 
\end{proof}

If $a(x, y)$ is a probability
distribution we write
$a(X)$ 
and 
$a(Y \given X = x)$ for $\law(X)$ and 
$\law( Y \given X = x)$, respectively,  when $(X, Y)$
is distributed according to $a$.

If $a(x, y)$ and $b(x, y)$ 
are probability
distributions we define
\[
  D(a(Y \given X) \re b(Y \given X))
  := \sum_{x,y} a(x,y) D( a(Y \given X = x) \re
b(Y \given X = x))
\]
where the sum is over $(x, y)$ such that $a(x, y) > 0$. The 
chain rule for relative entropy
(see Theorem 2.5.3 of \cite{covthomas})
states
\begin{equation}
  \label{relchain}
  D(a \re b) = 
    D(a(X) \re b(X)) +   D(a(Y \given X) \re b(Y \given X))
\end{equation}
\subsection{Conditional Relative Entropy}
In this subsection we discuss how to extend our notation for relative
entropy to handle
conditional distributions. The main issue is what reference 
distributions relative entropies should
be computed with with respect to. Roughly speaking,
if a random variable $X$ is being compared to a uniform random 
variable $\yY$, then
a conditional distribution involving $X$ should be compared to a 
corresponding conditional
distribution involving $\yY$. More precisely:
\begin{definition}
Suppose $X$ and $\yY$ are random variables taking values 
in the finite set $V$ and suppose that $\yY$ has the uniform distribution 
over $V$. 
If $g$ is a function on $V$ 
denote by $\law(X \given g(X) = z)$ the conditional distribution of $X$ given
that $g(X) = z$, denote by $\ent( X \given g(X) = z)$
the relative entropy of $\law( X \given g(X) = z)$ with respect
to $\law( \yY \given g( \yY) = z)$, and define
\[
\ent( X \given g(X)) := \sum_z \P(g(X) = z) \ent( X \given g(X) = z)\;.
\]
\end{definition}
\begin{definition}
If $h$ is a function on $V$
such that $\law( h(\yY) \given g(\yY) = z)$ is uniform over some set then
we denote by
$\ent( h(X) \given g(X) = z)$ the relative entropy of
\[
  \law( h(X) \given g(X) = z)
\]
with respect to
$\law( h(\yY) \given g(\yY) = z)$.
\end{definition}
\indent We shall now illustrate these concepts with two examples involving random permutations.
In the examples below, 
let $\pi$ be a random permutation in $\sym_n$, not necessarily with the uniform distribution, 
and
let $\piu$ be a uniform random permutation in $\sym_n$.
\begin{example} \label{relex1}
The quantity
$\ent( \pi \given \pi^{-1}(1) = 1)$
is the relative entropy
of   \\
$\law( \pi \given \pi^{-1}(1) = 1)$ with respect to
$\law( \piu \given \piu^{-1}(1) = 1)$. 
That is, the reference distribution here is the uniform distribution over
the set of permutations that have card $1$
at the top of the deck (a set of size $(n-1)!$).
\end{example}
\begin{example}
  \label{relex2}
  The quantity $\ent( \pi^{-1}(2) \given \pi^{-1}(1) = 1)$ is the relative
entropy of 
$\law( \pi^{-1}(2) \given \pi^{-1}(1) = 1)$ with
respect to
$\law( \pi_u^{-1}(2) \given \pi_u^{-1}(1) = 1)$. 
Here, the reference distribution is
the uniform distribution
  over $\{2, 3, \dots, n\}$.
\end{example}

\section{Collisions and Monte shuffles}
\label{collis}
Define a 3-collision (respectively, 2-collision) as a
random 
permutation that has equal probability 
of being either a
fixed 3-cycle 
(respectively, transposition) or the identity permutation.
Morris's paper \cite{morris}
is related to card shuffles that mix up the deck using $2$-collisions. 
Such shuffles were called {\it Monte shuffles}, after the game three-card Monte.
In the present paper we study shuffles that
mix up the deck using 3-collisions. We shall call such shuffles {\it 3-Monte},
and hereafter we shall call the shuffles 
studied in \cite{morris} {\it 2-Monte}.   ~\\

Each step of a 3-Monte shuffle is a random permutation followed by a
product of 3-collisions
involving mutually disjoint triples of cards.
We can imagine a dealer who moves cards around face-down on a table and periodically
executes a sleight-of-hand step in which she
chooses three of the cards, puts them together and then separates them so
quickly that the gambler can't tell whether a certain 3-cycle was performed
or whether the cards stayed unchanged. \\
~\\
Theorem \ref{maintheorem} 
bounds the change in relative entropy after many steps of a 3-Monte shuffle,
under a certain assumption about the behaviour of triples of cards. 
The main idea of the proof is that, under the conditions of the theorem
there is a certain subset of the 3-collisions
(which we call ``special'' collisions) that alone are enough to get a
good reduction in the relative entropy,
even if we condition on everything else.
(This is made more precise in Section \ref{entmaintheorem},
but here is a rough explanation: even if the gambler sees everything about the way the dealer moves the cards
around except for the outcomes of these ``special''   
sleight-of-hand steps, the deck is still getting mixed.) \\

Our Theorem \ref{maintheorem} plays a similar role to Theorem 9 in \cite{morris}
and the high-level strategy of the proof is similar.
However, there are some complications that make 3-Monte shuffles
more challenging to analyze than 2-Monte shuffles.
For example, note that
if $\mu_1$ and $\mu_2$ are independent, identically distributed 2-collisions,
then $\mu_1 \mu_2$ has the same distribution as $\mu_1$. This fact is used at the start of the proof of Theorem 9
in \cite{morris} to show that if $X_t$ is the random permutation 
representing the deck after $t$ steps of a 2-Monte shuffle then
$X_t$ has the same distribution as
$X_t \circ \alpha$, where $\alpha$ is a product
of a certain random collection of 2-collisions. 
The analogous fact does
not apply in the case of 3-collisions so we have to  proceed differently. 

\subsection{3-collisions and entropy decomposition}
We shall now give a more formal definition of  a 3-collision, 
which is the basic ingredient of the card
shuffles analyzed in the present paper.
Let $\cset_n$ denote the set of 3-cycles in $\sym_n$. 
Let $\id$ denote the identity permutation.
If $(i,j,k) \in \cset_n$ and 
$\mu$ is a random permutation in $\sym_n$ such that
\[
\mu= 
\left\{\begin{array}{ll}
\id & 
\mbox{with probability $\half$;}\\
(i,j,k) & \mbox{with probability $\half$,}\\
\end{array}
\right.
\]
then we will call $\mu$ a {\it 3-collision of $(i,j,k)$.}
If $\alpha$ is a 3-cycle then we will often write $c(\alpha)$ for
a 3-collision of $\alpha$.
%Note that there are 3 ways to write a given 3-cycle in cycle notation. So,
%for example, a 3-collision of $(i,j,k)$ can also be described as a 3-collision
%of $(j,k,i)$ or $(k,i,j)$.
%If $c$ is a 3-collision of $(i,j,k)$ then we say that
%$(i,j,k)$ is the 3-cycle of the collision $c$. 

%We will generalize the definition to allow for 3-collisions of random 3-cycles. 
%Suppose that $(I,J,K)$ is a random 3-cycle in $\sym_n$.
%If $\mu$ is a random permutation
%such that conditional on $(I,J,K) = (i,j,k)$, 
%$\mu$ is a 3-collision of $(i,j,k)$, then we say that $\mu$ is a 3-collision
%of $(I,J,K)$.
Call two cycles {\it disjoint} if they move disjoint sets of elements
and a collection of 3-collisions
{\it mutually disjoint} if the corresponding  3-cycles
are mutually disjoint.

%The {\it outcome} of a 3-collision
%means whether or not it is a 3-cycle. \\

  In the present paper, we shall focus on shuffles that mix up the deck using
3-collisions. One complication here is that multiplication by a \tc does not change the sign of the
permutation. Thus, we focus on how the conditional distribution of a permutation mixes, given its sign.
\\
~\\
Suppose $\mu$ is a permutation in $\sym_n$.
For $1 \leq i \leq n$,
let $\tail^{\mu}_{i} := (\mu^{-1}(i), \dots, \mu^{-1}(n))$
denote
the bottom $(n - i + 1)$ cards in the deck 
and let $\tail^{\mu}_{n+1}$ be the empty tuple.
We shall call $\tail^{\mu}_i$ the {\it $i$-tail} of $\mu$. 
Define $\tailt_i^\mu := (\sgn(\mu), \tail_i^\mu)$.  We shall call $\tailt_i^\mu$ the
{\it signed $i$-tail} of $\mu$. 
When the permutation referred to is clear from context,
we shall write $\tail_i$
(respectively, $\tailt_i$)
in place of $\tail^\mu_i$
(respectively, $\tailt_i^\mu$).

Let $\pi$ be a random permutation
in $\sym_n$
and let $\piu$ be a uniform random permutation in $\sym_n$. 
The function $f$ on $\sym_n$ defined by 
\[
f(\nu) = (\sgn(\nu), \nu)
\]
is one-to-one, and recall that by definition we have 
$\ent(\pi) = D( \law(\pi) \re \law(\piu))$. 
It follows that
\[
\ent(\pi) = D(\law(\sgn(\pi), \pi) \re \law(\sgn(\piu), \piu)) \;.
\] 
Hence the chain rule 
for relative entropy (see equation (\ref{relchain}))
gives
\begin{eqnarray}
\label{kindof}
\ENT(\pi)   &=&   \ENT(\sgn(\pi)) + \ENT( \pi \given \sgn(\pi)) \;.        % \label{threecomp} 
\end{eqnarray}

Now fix an integer $i$ with $2 \leq i \leq n$. 
Since the 
function $g$ on $\sym_n$ defined by 
\[
g(\nu) = (\nu^{-1}(n), \dots, \nu^{-1}(i+1), \nu)
\]
is one-to-one, it follows that 
\begin{equation}
\label{ding}
\ENT(  \pi \given \sgn(\pi)) = 
\ENT( \pi^{-1}(n), \dots, \pi^{-1}(i+1), \pi \given \sgn( \pi)) \;.
\end{equation}
If we apply the chain rule 
for relative entropy $(n - i)$ 
times in order to evaluate the 
relative entropy in the righthand side of equation (\ref{ding}),
we get
\[
  \ENT(\pi \given \sgn(\pi))  = 
\sum_{k=i+1}^{n} \ENT(\pi^{-1}(k) \given \tailt_{k+1}^{\pi}) +  
  \ENT(\pi \given
  \tailt_{i+1}^{\pi}) \; .
\]
Combining this with equation (\ref{kindof}) gives
\begin{equation}
 \ENT(\pi) 
=  \ENT(\sgn(\pi)) +    \sum_{k=i+1}^{n} 
\ENT(\pi^{-1}(k) \given \tailt_{k+1}^{\pi}) +  
 \ENT(\pi \given
  \tailt_{i+1}^{\pi}) \;.         \label{threecomp} 
\end{equation}
\begin{remark}
\label{sametill}
Suppose $\pi$ and $\pis$
are random permutations in $\sym_n$ that are defined on the same probability space and
suppose that $\pi$ and $\pis$ have the same signed $(i+1)$-tail with probabilty $1$. That
is,
$\pi$ and $\pis$ have the same sign and $\pi^{-1}(j) = \pis^{-1}(j)$ for all $j \in \{i+1, \dots, n\}$.
  Then all the entropies appearing in the first two terms of the expression (\ref{threecomp})
  are unchanged if we replace $\pi$ with $\pis$ and hence
\begin{equation}
\label{sametill1}
\ENT(\pi) - \ENT(\pis) = 
  \ENT(\pi \given   \tailt_{i+1}^{\pi})  - 
    \ENT(\pis \given
    \tailt_{i+1}^{\pis}) 
  \; .
\end{equation}
If we write $\tailt_{i+1}$ for the common signed $(i+1)$-tail of $\pi$ and $\pis$ 
then we can drop superscripts in the notation and write
  $\ENT(\pi \given   \tailt_{i+1})$ for  $\ENT(\pi \given   \tailt_{i+1}^{\pi})$ and
  $\ENT(\mu \given   \tailt_{i+1})$ for  $\ENT(\mu \given   \tailt_{i+1}^{\mu})$.
  Equation (\ref{sametill1}) then becomes
\begin{eqnarray}
  & & \ENT(\pi) - \ENT(\pis)   \nonumber   \\
  &=& 
     \ENT(\pi \given   \tailt_{i+1}) - 
     \ENT(\pis \given
  \tailt_{i+1})        \nonumber       \\
    &=& 
        \sum_{(s,\Tail)} \P \left( \tailt_{i+1} = (s,\Tail)\right)
\Bigl(  \ENT(\pi \given
    \tailt_{i+1} = (s,\Tail)) - 
\ENT(\pis \given
        \tailt_{i+1} = (s,\Tail))    \Bigr)  \;,    \label{sametill2}
\end{eqnarray}
where the last line follows from conditioning on the value of $\tailt_{i+1}$. 
Here, the sum is over the set of possible signed $(i+1)$-tails $(s,\Tail)$
(where
$s \in \{-1, 1\}$ is a sign and 
$\Tail$ is an $(i+1)$-tail).  
The relative entropies in the third line  are with respect to the uniform distribution over the set of permutations
with signed $(i+1)$-tail equal to $(s,\Tail)$; this is a set of size ${i! \over 2}$. 
\end{remark}

Define $\Et_k := \ENT(\pi^{-1}(k) \given \tailt_{k+1}^{\pi})$.
%Since $\pi$ can be completely specified by $\sgn(\pi)$
%and  
Since the values of
$\pi^{-1}(k)$ for $3 \leq k \leq n$
specify $\pi$ up to a transposition, it follows that
$\pi$ can be completely specified by $\sgn(\pi)$
and
$\pi^{-1}(k)$ for $3 \leq k \leq n$.
Therefore, we have $\ENT(\pi \given \tailt_3^{\pi}) = 0$,
and hence
substituting $i=2$ into (\ref{threecomp}) allows us to decompose $\ent(\pi)$ as follows:
\begin{eqnarray}
\label{decomp}
\ENT(\pi)   &=&   \ENT(\sgn(\pi)) +    \sum_{k=3}^{n} \Et_k, \;        % \label{threecomp} 
\end{eqnarray}
and in particular, if $\pi$ is equally likely to be odd or even, then
\begin{eqnarray*}
  & & \ENT(\pi)  =   \sum_{k=3}^{n} \Et_k \; .       % \label{threecomp} 
\end{eqnarray*}
It is helpful to think of $\Et_k$ as the portion of relative entropy contributed by position $k$
in the deck.

\section{Main Entropy  Lemma}
%In this section we call elements of $\{1, 2, \dots, n\}$ {\it cards} or {\it cards in the deck.}
\label{entsec}
If $\rho$ and $\mu$ are permutations in $\sym_n$, then we write $\rho \mu$
for the composition $\mu \circ \rho$.
\begin{definition}
  \label{beforeandafter}
  If $A$ is a set of mutually disjoint 3-cycles in $\sym_n$
  define $A_1: \{1, 2, \dots, n\} \to
  \{1, 2, \dots, n\}$ by
\[
A_1(x) =  
\left\{\begin{array}{ll}
y  \;\;         & 
\mbox{if for some $z$, the cycle $(x, y, z) \in A$,}\\
x \;\;  & 
\mbox{if no cycle in $A$ contains $x$,}\\
\end{array}
\right.
\]
and define
$A_2: \{1, 2, \dots, n\} \to
  \{1, 2, \dots, n\}$ by
\[
A_2(x) =  
\left\{\begin{array}{ll}
z  \;\;         & 
\mbox{if for some $y$, the cycle $(x, y, z) \in A$,}\\
x \;\;  & 
\mbox{if no cycle in $A$ contains $x$.}\\
\end{array}
\right.
\]
So
if there is a cycle in $A$ containing $x$ then
$A_1(x)$ (respectively, $A_2(x)$) is the element that comes after $x$
  (respectively, comes before $x$) in this cycle.
\end{definition}

Recall that $\cset_n$ is the set of 3-cycles in $\sym_n$. 
For $\alpha \in \cset_n$ let $c(\alpha)$ be a \tcoll of $\alpha$. 
%%Recall that we can decompose the entropy of a random permutation $\mu$ by
%  separating out the relative entropy attributable to the sign:
%\begin{align}
%   \ENT(\mu) = \E\left(\ENT(\mu \given \sgn(\mu))\right) + \ENT(\sgn(\mu)).
%\end{align}
%Since multiplying a random permutation by a 3-collision cannot change its sign,
%3-collisions can only affect thec  conditional relative entropy of a random permutation
%5given its sign.
For a given random permutation $\mu$ and set $A$ of  mutually disjoint 3-cycles, we can ask 
how much the  relative entropy  decreases when we multiply $\mu$ by all the 3-collisions
$c(\alpha)$ for $\alpha \in A$ (where we assume the
random permutations 
$\{c(\alpha) \}_{\alpha \in \cset_n}$ are mutually independent
and also
$(c(\alpha))_{\alpha \in \cset_n}$ is
independent of $\mu$). 
The following technical lemma
gives a bound on the  {\it mean} decrease in relative entropy when
the set of 3-cycles is itself random. 

\begin{lemma} 
  \label{enttheorem}
  Let $\seta$ be a random set of mutually disjoint 3-cycles in $\sym_n$
  and for $k \in \{1, \dots, n\}$ let $\seta_1(k)$ and $\seta_2(k)$
  be defined as in Definition \ref{beforeandafter}. 
  Suppose that for every element $k$ with $3 \leq k \leq n$ there is a constant
  $\Boy_k$ with  $0 \leq \Boy_k \leq 1$ such that
  for every $j \in \{1, \dots, k-1\}$ we have
$\P(\seta_2(k)=j, \seta_1(k) < k) \geq \frac{\Boy_k}{k}$.  
Let $\mu$ be an arbitrary random permutation and
for each set $A$ of mutually disjoint 3-cycles in $\sym_n$ define
\[
  \nu^A := \prod_{\alpha \in A} c(\alpha),
\]
where
each $c(\alpha)$ is a 3-collision of $\alpha$ and 
$\{c(\alpha)\}_{\alpha \in A}$ are independent and independent of $\mu$.  Then 
\begin{align}
  \nonumber
  \sum_{A} \P(\seta = A) \Bigl[ 
\ENT(\mu \nu^A) - \ENT(\mu) \Bigr] \leq \frac{-C}{\log n} \sum_{k=3}^{n}\Boy_kE_k,
\end{align}
where $E_k = \ENT(\mu^{-1}(k) \given \tailt_{k+1}^\mu)$ 
and $C$ is a universal constant. 
The sum in the lefthand side is over possible values $A$ of $\seta$.
\end{lemma}

\begin{proof}%[Proof of Theorem \ref{maintheorem}.]
  Fix a possible value $A$ of $\seta$. Note that by writing each 3-cycle in $\aset$
  with its highest element appearing first, we can write
\begin{align}
  \label{pprod}
\nu^A =  \prod_{\substack{i:A_1(i) < i \\ A_2(i)< i}} c(i,A_1(i), A_2(i)) \;.
\end{align}
%where  
%$c(u,v,w)$ is a \tcoll of $(u,v,w)$
%and the $c(u,v,w)$ are independent.
%where the sum is over $i$ such that $A_1(i)$ and $A_2(i)$ are defined. 
For $k$ with $2 \leq k \leq n$, define
\begin{align}
  \label{partprod}
\nu_k^A = \prod_{\substack{i:A_1(i) < i \leq k \\ A_2(i)< i}} c(i,A_1(i), A_2(i)), 
\end{align}
where the empty product is defined to be the identity.
In words, $\nu_k^A$ is the product of the \tcolls in $\nu^A$ that
involve only positions $1, \dots, k$ of the deck.
Note also that $\nu_k^A = \nu_{k-1}^A$ unless $A$
contains a \tc of the form $(k, i, j)$ for some $i < k$ and $j < k$
(that is, $k$ is the highest element appearing in
a certain \tc of $A$; see Figure \ref{nudiff} for an illustration).

\begin{figure}[t]
  \centering
\begin{tikzpicture}[>=Stealth, scale=0.4, every node/.style={scale=0.75}]
    % Draw vertices
    \foreach \i in {1,...,18} {
        \node[circle, fill=black, inner sep=1.5pt] (v\i) at (0, -\i) {};
    }

    % Draw hatch marks for unlabeled vertices (except those specifically labeled)
    \foreach \i in {2, 3,4,5,6, 7,8,9,10,11,12,13,14,15,16,17} {
        \draw (-0.15, -\i) -- (0.15, -\i);
    }
\coordinate (v1left) at (-0.17, -1);
\coordinate (v2left) at (-0.17, -2);
\coordinate (v3left) at (-0.17, -3);
\coordinate (v4left) at (-0.17, -4);
\coordinate (v5left) at (-0.17, -5);
\coordinate (v6left) at (-0.17, -6);
\coordinate (v7left) at (-0.17, -7);
\coordinate (v8left) at (-0.17, -8);
\coordinate (v9left) at (-0.17, -9);
\coordinate (v10left) at (-0.17, -10);
\coordinate (v11left) at (-0.17, -11);
\coordinate (v12left) at (-0.17, -12);
\coordinate (v13left) at (-0.17, -13);
\coordinate (v14left) at (-0.17, -14);
\coordinate (v15left) at (-0.17, -15);
\coordinate (v16left) at (-0.17, -16);
\coordinate (v17left) at (-0.17, -17);
\coordinate (v18left) at (-0.17, -18);
    
    % Add labels
    \node[left] at (v1left) {$1$};
    \node[left] at (v2left) {$A_2(k)$};
    \node[left] at (v6left) {$A_1(k)$};
    \node[left] at (v10left) {$k$};
     \node[left] at (v12left) {$\vdots$};
     \node[left] at (v18left) {$n$};

    % Draw edges (assume all edges between adjacent vertices are straight)
%    \foreach \i in {1,...,17} {
%        \draw[-] (v\i) -- (v\the\numexpr\i+1\relax);
%    }

    % Curved directed edges
    \draw[->, >=stealth, scale=0.4, bend right=40] (v3) to (v4);
    \draw[->, >=stealth, scale=0.4, bend right=40] (v4) to (v5);
    \draw[->, >=stealth, scale=0.4, bend right=40] (v5) to (v3);

    \draw[->, >=stealth, scale=0.4, bend right=40] (v7) to (v8);
    \draw[->, >=stealth, scale=0.4, bend right=40] (v8) to (v9);
    \draw[->, >=stealth, scale=0.4, bend right=40] (v9) to (v7);

    \draw[dotted, ->, >=stealth, scale=0.4, bend right=30] (v10) to (v6);
    \draw[dotted, ->, >=stealth, scale=0.4, bend right=30] (v6) to (v2);
    \draw[dotted, ->, >=stealth, scale=0.4, bend right=35] (v2) to (v10);

    % Dotted line between k and n (v11 and v18)
 %   \draw[dotted, thick] (v11) -- (v18);

\end{tikzpicture}
\caption{All 3-collisions in $\nu_{k-1}^A$ involve only positions in $\{1, \dots, k-1\}$. The 3-collision
of $(k, A_1(k), A_2(k))$ (shown in dotted lines) is in $\nu_k^A$ but not 
$\nu_{k-1}^A$
}
  \label{nudiff}
\end{figure}
In order for the product in (\ref{partprod}) to be nonempty we must have $k \geq 3$.
It follows that $\nu_2^A$ is the identity. Note also that $\nu_n^A = \nu^A$.
Hence
\begin{equation}
\begin{aligned}
&    \ENT(\mu\nu^A) - \ENT(\mu) \\
&=  \ENT(\mu\nu_n^A) - \ENT(\mu \nu_2^A) \\
&= \sum_{k=3}^{n} \left[ \ENT(\mu\nu_k^A) - \ENT(\mu\nu_{k-1}^A) \right ].
\end{aligned}
\end{equation}
Thus, it is enough to show that there is a universal constant $C > 0$ such that
for every $k$ with $3 \leq k \leq n$ we have 
\begin{align}
  \sum_A \P(\seta = A) \Bigl[ \ENT(\mu\nu_k^A) - \ENT(\mu\nu_{k-1}^A) \Bigr]
  \leq \frac{-C\Boy_kE_k}{\log n}  \;,
\label{eq:result_enough_for_theorem}
\end{align}
since summing (\ref{eq:result_enough_for_theorem}) over $k$ with $3 \leq k \leq n$ will yield the theorem.

Fix a possible value $A$ of $\seta$.
Note that the permutations $\nu_k^A$ and $\nu_{k-1}^A$ are products of
3-cycles (which are even permutations)
involving positions in $\{1, 2, \dots, k\}$.
Therefore, if we multiply any permutation by either $\nu_k^A$ or $\nu_{k-1}^A$
we don't change its signed $(k+1)$-tail. 
It follows that:
\begin{remark}
  \label{sametail}
  $\mu \nu_k^A$ and $\mu \nu_{k-1}^A$ both have the same signed $(k+1)$-tail as $\mu$.
Hereafter, we shall write $\tailt_{k+1}$ (dropping the superscript) for this common signed $(k+1)$-tail. 
  \end{remark}
Since $\mu \nu_k^A$ and $\mu \nu_{k-1}^A$ have the same signed $(k+1)$-tail,
equation (\ref{sametill2}) of Remark \ref{sametill} 
implies that  
\begin{equation}
\begin{aligned}
  &    \ENT(\mu\nu_k^A) - \ENT(\mu\nu_{k-1}^A)     \\
  &=
  \sum_{(s,\Tail)} \P \left( \tailt_{k+1} = (s,\Tail)\right)
\Bigl(  \ENT(\mu \nu_k^A \given
    \tailt_{i+1} = (s,\Tail)) - 
\ENT(\mu \nu_{k-1}^A \given
        \tailt_{k+1} = (s,\Tail))    \Bigr)  \;,     \label{needtobound}
\end{aligned}
\end{equation}
%In equation (\ref{needtobound2}) we have written $\tailt_{k+1}$ for this common signed $(k+1)$-tail, avoiding superscipts in the notation
%for improved readability. 
where the sum is over possible signed $(k+1)$-tails $(s, \Tail)$.
Note that we have written the relative entropy difference
$\ENT(\mu\nu_k^A) - \ENT(\mu\nu_{k-1}^A)$ as a weighted average of
conditional relative entropy differences of the form
\begin{equation}
  \label{needtobound2}
\ENT(\mu \nu_k^A \given
    \tailt_{k+1} = (s,\Tail)) - 
\ENT(\mu \nu_{k-1}^A \given
        \tailt_{k+1} = (s,\Tail))  \;.    
      \end{equation}
      Now, fix a signed $(k+1)$-tail $(s,\Tail)$.
      We wish to bound the conditional relative entropy difference (\ref{needtobound2}).
If 
$k$ is not the largest element of some 3-cycle in $A$
then $\nu_k^A = \nu_{k-1}^A$ and hence 
the conditional relative entropy difference in (\ref{needtobound2}) is $0$. Otherwise, we have $A_1(k) < k$ and $A_2(k) < k$,
and
\begin{equation}
  \label{change}
  \nu_k^A = \nu_{k-1}^A c(k, A_1(k), A_2(k)) \;.
\end{equation}
Our analysis in this case  will use techniques similar to \cite{morris}.
The collision $c(k, A_1(k), A_2(k))$ is equally likely to be the \tc $(k, A_1(k), A_2(k))$
or the identity, independently of $\mu$.
Thus, equation (\ref{change}) implies that
we can construct the random permutation  $\mu\nu_k^A$ by flipping a fair coin that is independent of $\mu$,
and then setting
\[
\mu \nu_k^A =  
\left\{\begin{array}{ll}
\mu \nu_{k-1}^A  \;\;         & 
\mbox{if the coin lands tails;}\\
\mu \nu_{k-1}^A(k, A_1(k), A_2(k)) \;\;  & 
\mbox{if the coin lands heads.}\\
\end{array}
\right.
\]
Define
\[
  \lambdam := \mu\nu^A_{k-1};   \;\;\;\;\;\;\;\;\; \lambdahm := \mu\nu^A_{k-1} (k, A_1(k), A_2(k)) \;,
\]
so that $\lambdam$ (respectively, $\lambdahm$)
is the value of $\mu \nu_k^A$ if the coin lands tails (respectively, heads).
Note that $\lambdahm = \lambdam (k, A_1(k), A_2(k))$.
Recall that the signed $(k+1)$-tail of $\lambdam$ is $\tailt_{k+1}$
(see Remark \ref{sametail}).
Since the 3-cycle $(k, A_1(k), A_2(k))$ moves only elements in positions $1, \dots, k$,
it follows that $\lambdahm$ also has signed $(k+1)$-tail equal to $\tailt_{k+1}$.
Furthermore, note that
$\lambdahm$ and $\lambdam$ differ only at three positions in the deck,
namely $k, A_1(k)$ and $A_2(k)$. Since the cycles in $A$ are disjoint
and the cycle $(k, A_1(k), A_2(k)) \in A$, it follows that none of the collisions
in $\nu_{k-1}^A$ involve positions  $A_1(k)$ or $A_2(k)$.
Furthermore, by the definition of $\nu_{k-1}^A$ none of the collisions in $\nu_{k-1}^A$ involve position $k$. It
follows that 
for $j \in \{k, A_1(k), A_2(k)\}$ we have
$(\lambdam)^{-1}(j) = \mu^{-1}(j)$. 
Since $\lambdahm = \lambdam (k, A_1(k), A_2(k))$ we also have
\begin{align}
    (\lambdahm)^{-1}(k) = (\lambdam)^{-1}(A_2(k)) = \mu^{-1}(A_2(k));     \label{m1}     \\[1.5ex]
    (\lambdahm)^{-1}(A_1(k)) = (\lambdam)^{-1}(k) = \mu^{-1}(k);     \label{m2}               \\[1.5ex] 
    (\lambdahm)^{-1}(A_2(k)) = (\lambdam)^{-1}(A_1(k)) = \mu^{-1}(A_1(k))       \label{m3}           .
\end{align} 
%Now, fix a possible value $(s, \Tail)$ of $\tailt_{k+1}^A$.
%(Here,
%$\Tail$ is a possible $(k+1)$-tail and $s \in \{-1, 1\}$ is a possible sign.)
Define $\plaw = \plaw(\mu, k, A, s, \Tail)$ and
$\qlaw = \qlaw(\mu, k, A, s, \Tail)$ by
\[
  \plaw := \law( \lambdam  \given \tailt^\mu_{k+1} = (s, \Tail)) ;
  \]
    \[
      \qlaw := \law( \lambdahm  \given \tailt^\mu_{k+1} = (s, \Tail)).
    \]  
%    Note that $\plaw$ and $\qlaw$ depend on $k, A, \Tail$ and $s$, but we shall
%    leave that out of the notation for the sake of readability.
    Here, $\qlaw$ and $\plaw$ are probability distributions over
    the set $\V = \V(k, s, \Tail)$ of permutations in $\sym_n$
with signed $(k+1)$-tail equal to $(s,\Tail)$. Note that $|\V| = {k! \over 2}$. 
Since $\tail_{k+1} = \tail_{k+1}^\mu$
is a function of $\mu$, and the coin used to construct
$\mu \nu_k^A$ is independent of $\mu$, it follows that
\begin{equation}
  \begin{aligned}
    \label{inlaws}
 \law(\mu \nu_k^A \given \tailt_{k+1} = (s,\Tail))  &=\frac{\plaw + \qlaw}{2}.    
\end{aligned}
\end{equation}
Also, since $\lambdahm = \lambda (k, A_1(k),A_2(k))$
and the \tc $(k, A_1(k), A_2(k))$ is a deterministic permutation, Lemma \ref{determ} implies
that $\ENT( \plaw ) = \ENT(\qlaw)$. Hence
\begin{equation}
  \label{same}
  \ENT(\plaw) = \frac{1}{2} \ENT(\plaw) + \frac{1}{2} \ENT(\qlaw),  
\end{equation}
which is equivalent to 
\begin{equation}
\begin{aligned}
    \sum_{\eta \in \V} \plaw_{\eta}\log \left(\frac{k!}{2} \cdot \plaw_{\eta} \right) 
    &= \frac{1}{2}\sum_{\eta \in \V} \plaw_{\eta}\log \left(\frac{k!}{2} \cdot \plaw_{\eta}\right) +
    \frac{1}{2}\sum_{\eta \in \V} \qlaw_{\eta}\log \left(\frac{k!}{2} \cdot \qlaw_{\eta} \right).
\end{aligned}
\end{equation}
%where the sum is over the possible values $\eta$ of $\lambdahm$ when $\tailt_{k+1}^A = (s, )$. 
We now recall  some definitions and background from \cite{morris}.
For real numbers $p, q \geq 0$, define 
\begin{align}
    d(p, q) := \frac{1}{2}p\log p + \frac{1}{2}q\log q - \frac{p+q}{2}\log\left(\frac{p+q}{2}\right),
\end{align}
where $0\log 0 := 0$. 
The function $d(\cdot, \cdot)$ is nonnegative  and for a
fixed $q \geq 0$, the function $d(\cdot, q)$ is convex.

If the sequences $p = (p_i)_{i \in V}$ and $q = (q_i)_{i \in V}$
both represent probability distributions
on a finite set $V$, we define the ``distance'' 
\begin{align}
    d(p, q) := \sum_{i\in V} d(p_i, q_i),
\end{align}
so that
\[
  \ENT \Bigl( {p + q \over 2} \Bigr) = \half \ENT(p) +  \half \ENT(q) - d(p,q) .
\]
Applying this to
$\law(\mu \nu_k^A \given \tailt_{k+1} = (s, \Tail))$, which is equal to $\frac{\plaw + \qlaw}{2}$,
gives
\begin{eqnarray*}
  \ENT( \mu \nu_k^A \given\tailt_{k+1} = (s,\Tail) ) &=& \frac{1}{2} \ENT( \plaw ) + \frac{1}{2} \ENT( \qlaw )
                                                             - d( \plaw, \qlaw) \\
                                                 &=& \ENT(\plaw)- d( \plaw, \qlaw)     \\ 
                                                 &=&  \ENT( \mu \nu_{k-1}^A \given \tailt_{k+1} = (s,\Tail)     ) - d( \plaw, \qlaw),
\end{eqnarray*}
where the second line follows from equation (\ref{same}). Our aim is now to bound
\[
  d(\plaw, \qlaw) =
  d \Bigl( \law( \lambdam  \given \tailt_{k+1} = (s, \Tail)), \law( \lambdahm  \given \tailt_{k+1} = (s, \Tail))  \Bigr).
  \]
To that end, we now recall the projection lemma proved in \cite{morris}:

\begin{lemma} \cite{morris}
  Let $X$ and $Y$ be random variables with distributions $p$ and $q$, respectively. Fix a function $g$ and
  let $P$ and $Q$ be the distributions of $g(X)$ and $g(Y)$, respectively. Then $d(p, q) \geq d(P, Q)$.
\end{lemma}

Applying the lemma to the function
$g: \sym_n \to \{1, \dots, n\}$ defined by $g(\lambda) = \lambda^{-1}(k)$ gives 
\begin{equation}
\begin{aligned}
  &d\Bigl(\law (\lambdahm \given \tailt_{k+1} = (s, \Tail)), \law (\lambdam \given \tailt_{k+1} = (s, \Tail))\Bigr)  \\
  &\geq d\Bigl(\law( (\lambdahm)^{-1}(k)  \given \tailt_{k+1} = (s, \Tail)), \law( (\lambdam)^{-1}(k) \given \tailt_{k+1} =
  (s, \Tail))\Bigr)\\
  &= d\Bigl(\law(\mu^{-1}(A_2(k)) \given \tailt_{k+1} = (s, \Tail)), \law(\mu^{-1}(k) \given \tailt_{k+1} =
  (s, \Tail))\Bigr),
\end{aligned}
\end{equation}
where the last line follows from (\ref{m1}) and (\ref{m2}). It follows that
the conditional relative entropy difference
\begin{eqnarray*}
  & &   \ENT( \mu \nu_k^A \given\tailt_{k+1} = (s,\Tail) ) - \ENT( \mu \nu_{k-1}^ A\given \tailt_{k+1}^A = (s,\Tail)     )   \nonumber     \\
  &\leq& - d\Bigl
         (\law(\mu^{-1}(A_2(k)) \given \tailt_{k+1} = (s, \Tail)), \law(\mu^{-1}(k) \given \tailt_{k+1} =
         (s, \Tail))\Bigr). 
\end{eqnarray*}
This is true for every possible value $A$ of $\seta$
such that  $A_1(k) < k$ and $A_2(k) < k$ (and recall 
that the conditional relative entropy difference is $0$ for all other $A$).
It follows that for every possible value $A$ of $\seta$ we have
\begin{eqnarray*}
& & \ent( \mu \nu_k^A \given \tailt_{k+1} = (s,\Tail)) - \ent( \mu \nu_{k-1}^A \given \tailt_{k+1} = (s,\Tail))  \\
  &\leq&  - \sum_{j=1}^{k-1} \one( A_2(k) = j, A_1(k) < k)
         d\Bigl(\law(\mu^{-1}(j) \given \tailt_{k+1}
         = (s, \Tail)), \law(\mu^{-1}(k) \given \tailt_{k+1} = (s, \Tail))\Bigr) \\
&=&  - \sum_{j=1}^{k-1} \one( A_2(k) = j, A_1(k) < k)     D(\mu, j,k,s,\Tail), 
\end{eqnarray*}
where $D(\mu, j,k,s,\Tail)$ is defined by
\[
  D(\mu, j,k,s,\Tail) := d\Bigl(\law(\mu^{-1}(j) \given \tailt_{k+1}
  = (s, \Tail)), \law(\mu^{-1}(k) \given \tailt_{k+1} = (s, \Tail))\Bigr) \;.
\]
Combining this with (\ref{needtobound}) gives
\begin{equation}
\begin{aligned}
  &    \ENT(\mu\nu_k^A) - \ENT(\mu\nu_{k-1}^A)     \\
  &\leq  -
         \sum_{(s,\Tail)} \P \left( \tailt_{k+1} = (s,\Tail)\right)
         \sum_{j=1}^{k-1} \one( A_2(k) = j, A_1(k) < k) D(\mu, j,k,s,\Tail) \,,
\end{aligned}
\end{equation}
and hence
\begin{equation}
\begin{aligned}
  & \sum_{A} \P(\seta = A) \Bigl(  \ENT(\mu\nu_k^A) - \ENT(\mu\nu_{k-1}^A)  \Bigr)    \\
  &\leq  -
  \sum_A \P(\seta = A) 
         \sum_{(s,\Tail)} \P \left( \tailt_{k+1} = (s,\Tail)\right)
         \sum_{j=1}^{k-1} \one( A_2(k) = j, A_1(k) < k) D(\mu, j,k,s,\Tail) \\
  &=  -
         \sum_{(s,\Tail)} \P \left( \tailt_{k+1} = (s,\Tail)\right)
         \sum_{j=1}^{k-1}  \sum_A \P(\seta = A) 
 \one( A_2(k) = j, A_1(k) < k) D(\mu, j,k,s,\Tail) \,.
\end{aligned}
\end{equation}
Since $D(\mu, j,k,s,\Tail)$ is independent of $A$, we can factor it out of the inner sum in the expression above and
rewrite the expression as
\begin{equation}
\nonumber
\begin{aligned}
%  & \sum_{A} \P(\seta = A) \Bigl(  \ENT(\mu\nu_k^A) - \ENT(\mu\nu_{k-1}^A)  \Bigr)    \\
%  &\leq  -
&  -        \sum_{(s,\Tail)} \P \left( \tailt_{k+1} = (s,\Tail)\right)
         \sum_{j=1}^{k-1}  D(\mu, j,k,s,\Tail) \sum_A \P(\seta = A) 
 \one( A_2(k) = j, A_1(k) < k)   \\
  &=  -
         \sum_{(s,\Tail)} \P \left( \tailt_{k+1} = (s,\Tail)\right)
         \sum_{j=1}^{k-1}  D(\mu, j,k,s,\Tail) 
         \P( \seta_2(k) = j, \seta_1(k) < k)  \\
         &\leq  -
                  \sum_{(s,\Tail)} \P \left( \tailt_{k+1} = (s,\Tail)\right)
       \sum_{j=1}^{k-1}  D(\mu, j,k,s,\Tail) {\Boy_k \over k} \;,
\end{aligned}
\end{equation}
where the last line follows from the fact that  $\P( \A_2(k) = j, \A_1(k) < k) \geq {\Boy_k \over k}$,
which is an assumption of the theorem. Since $D(\mu, k,k,s,\Tail) = 0$, the expression above does not change if we
change the upper limit of summation in the inner sum to $k$.
Furthermore, we can factor the constant $\Boy_k$ out of both sums. 
It follows that
\begin{equation}
\begin{aligned}
  & \sum_{A} \P(\seta = A) \Bigl(  \ENT(\mu\nu_k^A) - \ENT(\mu\nu_{k-1}^A)  \Bigr)    \\
         &\leq  - \Boy_k
                  \sum_{(s,\Tail)} \P \left( \tailt_{k+1} = (s,\Tail)\right)   \label{wated}
                  \sum_{j=1}^{k} {1 \over k} D(\mu, j,k,s,\Tail) 
                  \;.
\end{aligned}
\end{equation}
Recall that for any fixed $q \geq 0$, the function $d(\cdot, q)$ is convex. If follows that %Jensen's inequality implies
\begin{eqnarray}
  & &                  \sum_{j=1}^{k} {1 \over k} D(\mu, j,k,s,\Tail)  \nonumber \\
  &=&     \sum_{j=1}^{k} {1 \over k} d(\law(\mu^{-1}(j) \given \tailt_{k+1}^A = (s,\Tail)), \law(\mu^{-1}(k) \given \tailt_{k+1}^A = (s,\Tail))) \nonumber
  \\ % \label{lastline}
&\geq&   d \Bigl( {1 \over k} \sum_{j=1}^{k}  \law(\mu^{-1}(j) \given \tailt_{k+1}^A = (s,\Tail)),
    \law(\mu^{-1}(k) \given \tailt_{k+1}^A = (s,\Tail))\Bigr),  \label{av} 
\end{eqnarray}
where the last line follows from Jensen's inequality. 
For every $(k+1)$-tail $\Tail$, let $\ut$ denote the uniform distribution over the values of
$\{1, \dots, n\}$ that do not appear in $\Tail$.
Recall that if we think of $\mu$ as representing the order of
a deck of cards then $\mu^{-1}(l)$ represents the card
in position $l$. Suppose that the signed $(k+1)$-tail of $\mu$ is $(s,\Tail)$.
Then every element of $\{1, \dots, n\}$ that is not in the $(k+1)$-tail $\Tail$ must appear 
exactly once in one of the positions $1, \dots, k$. It follows that
\[
  {1 \over k} \sum_{j=1}^{k}  \law(\mu^{-1}(j) \given \tailt_{k+1}^A = (s,\Tail)) = \ut \;,
\]
and hence the quantity (\ref{av}) is equal to 
\begin{equation}
  \label{newex}
d \Bigl( \ut,
\law(\mu^{-1}(k) \given \tailt_{k+1}^A = (s,\Tail))\Bigr)
\end{equation}

We now recall a lemma from \cite{morris} that relates two notions of a distance from the uniform distribution. 

\begin{lemma}\cite{morris}
  \label{avlemma}
 If $\u(V)$ is the uniform distribution on $V$ and $q$ is any distribution on $V$, then
\begin{align*}
    d(\u(V), q) \geq \frac{C}{\log |V|}\ENT(q)
\end{align*}
for a universal constant $C>0$.
\end{lemma}
Note that the universal constant $C$ that appears in our Lemma \ref{enttheorem}
is the same as the universal constant $C$ from Lemma \ref{avlemma}.

Applying Lemma \ref{avlemma} to the expression (\ref{newex}) for the quantity (\ref{av})
implies that the quanitity (\ref{av}) is at least 
\begin{eqnarray*}
  & &  \frac{C}{\log k}\, \ENT(\mu^{-1}(k) \given \tailt_{k+1}^A = (s,\Tail)) \;,
\end{eqnarray*}
and hence, by (\ref{wated}), 
\begin{equation}
\begin{aligned}
  & \sum_{A} \P(\seta = A) \Bigl(  \ENT(\mu\nu_k^A) - \ENT(\mu\nu_{k-1}^A)  \Bigr)    \\
         &\leq  - {C\Boy_k \over \log k}
                  \sum_{(s,\Tail)} \P \left( \tailt_{k+1} = (s,\Tail)\right) 
                   \ENT(\mu^{-1}(k) \given \tailt_{k+1}^A = (s,\Tail)) \\
         &=  - {C\Boy_kE_k \over \log k} \;.       \nonumber
\end{aligned}
\end{equation}
Since $k \leq n$ and hence $\log k \leq \log n$, this verifies equation
(\ref{eq:result_enough_for_theorem}),
proving the lemma.
\end{proof}

\section{3-collisions/entropy theorem}
\label{entmaintheorem}
\subsection{Preliminaries}
\label{backsec}
In Theorem 9 of \cite{morris}, Morris proves a lower bound on the decrease of
relative entropy after many steps of a shuffle that involves $2$-collisions.
In this section we prove an analogous result for 3-collisions.
%First, define a \textit{3-collision} to be a random permutation $\pi$ in $\sym_n$ such that there exist distinct $x, y, z \in \{1, \dots, n\}$% such that $\pi$ is either the identity with probability $\frac{1}{2}$ or the 3-cycle $(x, y, z)$. with probability $\frac{1}{2}$.

Recall that if $\rho$ and $\mu$ are permutations in $\sym_n$, then we write $\rho \mu$
for the composition $\mu \circ \rho$, and $\cset_n$ denotes the set of 3-cycles in $\sym_n$.
We will consider shuffles such that the shuffle  $\pi_m$ at time $m$ can be written in the form
\begin{align}
\label{formof}
\pi_m = \Phi_m \prod_{\alpha \in \T_m} c(m, \alpha)
\end{align}
where $\Phi_m$
is an arbitrary random permutation,
$\T_m$ is a random set of mutually disjoint 3-cycles,
and $c(m, \alpha)$ is a 3-collision of $\alpha$. 
The random set $\T_m$ 
may depend on $\Phi_m$,
but 
we assume that the 3-collisions $\{c(m, \alpha)\}_{\alpha \in \cset_n}$ are independent and
independent of $(\T_m, \Phi_m)$.
(Think of $\{c(m, \alpha)\}_{\alpha \in \T_m}$
as being the 3-collisions that occur at time $m$. )
%(That is, whether each 3-collision
%is a 3-cycle or the identity is independent of everything else.)
Such permutations
will be
called \textit{3-Monte}, in keeping with the terminology used in \cite{morris}. While any
random
permutation can be considered to be 3-Monte in the trivial sense of having no
3-collisions (in this case
our main theorem will have no content),
our aim will be to write random 
permutations usefully in 3-Monte form.

Before we get to the main theorem we shall first show how the torus shuffle can be described as a 3-Monte permutation.
Recall that the torus shuffle
in
an $n \times n$ grid
has the following transition rule: at each step, we select
a row or column, uniformly at random from all $2n$ possibilities,
and then cyclically rotate it by one unit in a random direction.
To avoid periodicity,
we add a holding probability of $\half$. That is, with probability
$\half$, do nothing; otherwise do a move as described above.
We shall refer to the steps when the chain is idle as {\it lazy steps.}
Henceforth the term {\it torus shuffle}
shall refer to the lazy version of the chain. 

A key observation is that if a row move $r$ is followed by a column move $c$,
the result is the same as $c$ followed by $r$, followed by a
certain 3-cycle.
That is, $rc = cr\Gamma$, where $\Gamma$ is a 3-cycle.
%Or, equivalently, $r^{-1}c^{-1}rc = \Gamma$, where $\Gamma$ is a 3-cycle.
(Recall that the shuffles are written in left-to-right order.)
Thus, the Markov chain that does two steps
of the torus shuffle at a time can be simulated as follows.
Generate two steps of the torus shuffle.
If the
two steps are both row moves or both column moves or contain lazy steps,
perform them as usual.
Otherwise, say the moves are the row move $r$ and the column move $c$,
in some order. In this case:
\begin{enumerate}
\item Perform $c$ and then $r$; then 

\item  Perfom the 3-cycle $\Gamma = r^{-1}c^{-1} rc$ with probability $\half$.
\end{enumerate}

Since Step 2 above is a 3-collision, this process is 3-Monte.
See Figure \ref{fig:3-collision} for an illustration.
The process described here can be applied to any
model of card shuffling where any two steps $s_1$ and $s_2$ {\it almost}
commute in the sense that $s_1s_2$ and $s_2s_1$ differ by at most  a 3-cycle.
(Indeed, this is the case for the overlapping cycles shuffle
analyzed in \cite{bmo}. It is also the case there that two steps of the shuffle can be expressed
as a single step of a certain 3-Monte shuffle.)
Note that in these cases when the shuffle
is expressed in 3-Monte form the number of 3-collisions at each step is
either zero or one. However, in the general definition of a 3-Monte shuffle we also
allow multiple 3-collisions
at each step, so long as they are disjoint. 

\begin{figure}[!h]
  \centering
\begin{tikzpicture}
%first sequence
\draw[step=0.5cm,color=black] (-1,-1) grid (1,1);
\node at (-0.75,+0.75) {1};
\node at (-0.25,+0.75) {4};
\node at (+0.25,+0.75) {9};
\node at (+0.75,+0.75) {16};
\node[step=0.5cm,color=black] at (-0.75,+0.25) {2};
\node[step=0.5cm,color=black] at (-0.25,+0.25) {3};
\node[step=0.5cm,color=black] at (+0.25,+0.25) {8};
\node[step=0.5cm,color=black] at (+0.75,+0.25) {15};
\node at (-0.75,-0.25) {5};
\node at (-0.25,-0.25) {6};
\node at (+0.25,-0.25) {7};
\node at (+0.75,-0.25) {14};
\node at (-0.75,-0.75) {10};
\node at (-0.25,-0.75) {11};
\node at (+0.25,-0.75) {12};
\node at (+0.75,-0.75) {13};
\node at (-1.20,+0.25) {$\to$};
\node at (+1.20,+0.25) {$\to$};
\node at (+1.5, 0) {$\implies$};
\draw[step=0.5cm,color=black] (2,-1) -- (2,1);
\draw[step=0.5cm,color=black] (2,-1) grid (4,1);
\node at (+2.25,+0.75) {1};
\node[step=0.5cm,color=black][step=0.5cm,color=black][step=0.5cm,color=black] at (+2.75,+0.75) {4};
\node at (+3.25,+0.75) {9};
\node at (+3.75,+0.75) {16};
\node at (+2.25,+0.25) {15};
\node[step=0.5cm,color=black][step=0.5cm,color=black][step=0.5cm,color=black] at (+2.75,+0.25) {2};
\node at (+3.25,+0.25) {3};
\node at (+3.75,+0.25) {8};
\node at (+2.25,-0.25) {5};
\node[step=0.5cm,color=black][step=0.5cm,color=black][step=0.5cm,color=black] at (+2.75,-0.25) {6};
\node at (+3.25,-0.25) {7};
\node at (+3.75,-0.25) {14};
\node at (+2.25,-0.75) {10};
\node[step=0.5cm,color=black][step=0.5cm,color=black][step=0.5cm,color=black] at (+2.75,-0.75) {11};
\node at (+3.25,-0.75) {12};
\node at (+3.75,-0.75) {13};
\node at (+2.75,+1.20) {$\downarrow$};
\node at (+2.75,-1.20) {$\downarrow$};
\node at (+4.5, 0) {$\implies$};
\draw[step=0.5cm,color=black] (5,-1) -- (5,1);
\draw[step=0.5cm,color=black] (5,-1) grid (7,1);
\node at (+5.25,+0.75) {1};
\node at (+5.75,+0.75) {11};
\node at (+6.25,+0.75) {9};
\node at (+6.75,+0.75) {16};
\node at (+5.25,+0.25) {15};
\node[fill=gray!20] at (+5.75,+0.25) {4};
\node[fill=gray!20] at (+6.25,+0.25) {3};
\node at (+6.75,+0.25) {8};
\node at (+5.25,-0.25) {5};
\node[fill=gray!20] at (+5.75,-0.25) {2};
\node at (+6.25,-0.25) {7};
\node at (+6.75,-0.25) {14};
\node at (+5.25,-0.75) {10};
\node at (+5.75,-0.75) {6};
\node at (+6.25,-0.75) {12};
\node at (+6.75,-0.75) {13};
%second sequence
\draw[step=0.5cm,color=black] (-1,-4) grid (1,-2);
\node at (-0.75,-2.25) {1};
\node[step=0.5cm,color=black][step=0.5cm,color=black] at (-0.25,-2.25) {4};
\node at (+0.25,-2.25) {9};
\node at (+0.75,-2.25) {16};
\node at (-0.75,-2.75) {2};
\node[step=0.5cm,color=black][step=0.5cm,color=black] at (-0.25,-2.75) {3};
\node at (+0.25,-2.75) {8};
\node at (+0.75,-2.75) {15};
\node at (-0.75,-3.25) {5};
\node[step=0.5cm,color=black][step=0.5cm,color=black] at (-0.25,-3.25) {6};
\node at (+0.25,-3.25) {7};
\node at (+0.75,-3.25) {14};
\node at (-0.75,-3.75) {10};
\node[step=0.5cm,color=black][step=0.5cm,color=black] at (-0.25,-3.75) {11};
\node at (+0.25,-3.75) {12};
\node at (+0.75,-3.75) {13};
\node at (-0.25,-1.80) {$\downarrow$};
\node at (-0.25,-4.20) {$\downarrow$};
\node at (+1.50,-3.00) {$\implies$};
\draw[step=0.5cm,color=black] (2,-4) -- (2,-2);
\draw[step=0.5cm,color=black] (2,-4) grid (4,-2);
\node at (+2.25,-2.25) {1};
\node at (+2.75,-2.25) {11};
\node at (+3.25,-2.25) {9};
\node at (+3.75,-2.25) {16};
\node[step=0.5cm,color=black] at (+2.25,-2.75) {2};
\node[step=0.5cm,color=black] at (+2.75,-2.75) {4};
\node[step=0.5cm,color=black] at (+3.25,-2.75) {8};
\node[step=0.5cm,color=black] at (+3.75,-2.75) {15};
\node at (+2.25,-3.25) {5};
\node at (+2.75,-3.25) {3};
\node at (+3.25,-3.25) {7};
\node at (+3.75,-3.25) {14};
\node at (+2.25,-3.75) {10};
\node at (+2.75,-3.75) {6};
\node at (+3.25,-3.75) {12};
\node at (+3.75,-3.75) {13};
\node at (+1.80,-2.75) {$\rightarrow$};
\node at (+4.20,-2.75) {$\rightarrow$};
\node at (+4.5, -3) {$\implies$};
\draw[step=0.5cm,color=black] (5,-4) -- (5,-2);
\draw[step=0.5cm,color=black] (5,-4) grid (7,-2);
\node at (+5.25,-2.25) {1};
\node at (+5.75,-2.25) {11};
\node at (+6.25,-2.25) {9};
\node at (+6.75,-2.25) {16};
\node at (+5.25,-2.75) {15};
\node[fill=gray!20] at (+5.75,-2.75) {2};
\node[fill=gray!20] at (+6.25,-2.75) {4};
\node at (+6.75,-2.75) {8};
\node at (+5.25,-3.25) {5};
\node[fill=gray!20] at (+5.75,-3.25) {3};
\node at (+6.25,-3.25) {7};
\node at (+6.75,-3.25) {14};
\node at (+5.25,-3.75) {10};
\node at (+5.75,-3.75) {6};
\node at (+6.25,-3.75) {12};
\node at (+6.75,-3.75) {13};
\end{tikzpicture}
\caption{A possible 3-collision. The sequence on top
  shows the 2nd row shifted right, then the 2nd column shifted down. The sequence on the bottom
 shows the 2nd column shifted down, then the 2nd row shifted right. Note that in the final
configuration only the shaded tiles differ, by a 3-cycle.}
  \label{fig:3-collision}
\end{figure}
%\subsection{Background}
Suppose $\tau$ and $\rho$ are permutations in $\sym_n$
and let $\rhot = \tau \rho \tau^{-1}$.
We call $\rhot$ the conjugate of $\rho$ by $\tau$. Note that
\[
  \rhot \tau = \tau \rho \;.
\]
If we think of $\tau$ as a reordering of
cards $1, \dots, n,$ then we
can think of $\rhot$ as the permutation that has the same effect
when applied before the reordering that 
$\rho$ has if applied after the reordering.
Conjugation preserves the cycle structure of a permutation.
We will mostly be interested in the case where $\rho$ is a 3-cycle.
In this case $\rhot$ is also a 3-cycle. More precisely,
if $\rho$ is the 3-cycle $(a, b, c)$
then $\rhot$ is the 3-cycle $(\tau^{-1}(a), \tau^{-1}(b), \tau^{-1}(c))$.
(This
can be verified using the definition of $\rhot$.
It is true
because the cards that are in positions $a, b,$ and $c$, respectively,
after being reordered according to $\tau$, were in positions
$\tau^{-1}(a), \tau^{-1}(b),$ and  $\tau^{-1}(c)$, respectively,
before the reordering.)

We will use the following three facts about conjugation:
\begin{enumerate}
\item If $\alpha_1$ and $\alpha_2$ are disjoint cycles then
  $\tau \alpha_1 \tau^{-1}$ and $\tau \alpha_2 \tau^{-1}$ are disjoint cycles. 
\item If $c$ is a 3-collision of $(a,b,c)$ then $c$ conjugated by $\tau$ is
  a 3-collision of $(\tau^{-1}(a), \tau^{-1}(b), \tau^{-1}(c))$.
%  Note that in the case where $\tau$ is random, $c$ conjugated by $\tau$
%  is a 3-collision of a random 3-cycle.
\item For $i$ with $1 \leq i \leq k$, let $\tal_i$ be the conjugate of $\alpha_i$ by  $\tau$. Then the conjugate of
  $\alpha_1 \alpha_2 \cdots \alpha_k$ by $\tau$ is $\tal_1 \tal_2 \cdots \tal_k$.  
\end{enumerate}

If $\pi_m$ is a a 3-Monte permutation 
written in the form (\ref{formof}),
we say that $\pi_m$ {\em has a collision of
$(i, j, k)$} if  
$(i,j,k) \in \T_m$. Note that in this case we could also describe $\pi$ has
having a 
collision of $(j,k,i)$ or 
$(k,i,j)$. We say $\pi_m$ {\it has the collision $c(m,\alpha)$} if $\alpha \in \T_m$. 
%(but not, say, a collision of $j,i$ and $k$, respectively; the order of
%$i, j$ and $k$ in the notation has to correspond to the potential 3-cycle).
%We say that $\pi$ {\it has a collision that involves $x$}
%if there exist $y$ and $z$ such that $\pi$ has a collision
%of $x$, $y$ and $z$, respectively.

%\textbf{Proposition 8.} Let $\mu$ be a random permutation with some distribution, and suppose $\pi$ is some fixed permutation. Then
%\begin{align}
%\ENT(\mu\pi) = \ENT(\mu)
%\end{align}
%Note that both $\mu\pi$ and $\mu$ have the same distribution up to relabeling of the indices, so they must have the same 

Let $\pi_1, \pi_2 \dots$ be
i.i.d.~3-Monte shuffles. Let $\pi_{(0)}$
be the identity
and for $t \geq 1$, let $\pi_{(t)} := \pi_1 \cdots \pi_t$.
Think of $\pi_t$ as the shuffle performed at time $t$ and think of
$\pi_{(t)}$ as the result of all $t$ shuffles performed up to time $t$. 
For $1 \leq x \leq n$, denote by \textit{card $x$} the card initially in position $x$.

%\begin{remark}
%\label{conjremark}
%  Let $m \geq 1$ be a time step and 
%suppose that $\pi_m$ has a 3-collision of positions $(i, j, k)$. 
%Then $\pim$ can be written as $\pim = \alpha c(i,j,k)$ for some permutation $\alpha$. 
%Hence
%the conjugate of $(i,j,k)$ by $\pim$ is
%\begin{eqnarray}
%%%      & & \pim (i,j,k) \pim^{-1} \\
%      &=& \alpha c(i,j,k) (i,j,k) c(i,j,k)^{-1} \alpha^{-1} \\
%      &=& \alpha (i,j,k) \alpha^{-1},
%\end{eqnarray}
%which is equal to the 3-cycle $(\alpha^{-1}(i), \alpha^{-1}(j), \alpha^{-1}(k))$.
%independent of whether the collision of $(i,j,k)$ was equal to $(i,j,k)$ or the identity, 
%Suppose that this is the 3-cycle  $(x,y,z)$. 
%Then $x, y$ and $z$ are the {\em initial positions} of the cards involved
%
%Furthermore, the ordered 3-tuple $(\alpha(x), \alpha(y), \alpha(z))$
%is equal to $(i, j, k)$ or $(j,k,i)$ or $(k,i,j)$.
%\end{remark}

%Remark \ref{conjremark}  motivates the following definition. 
For distinct cards $x, y, z$, % (with $x,y,z$ distinct),
say that cards $(x, y, z)$
{\em collide at time $m$} if
there exists a 3 cycle $(i, j, k) \in \cset_n$ such that $\pi_m$ has a collision of  $(i,j,k)$
and the conjugate of $(i,j,k)$ by $\pim$ is $(x,y,z)$.
We can interpret this as follows. The values of $x, y$ and $z$ are the {\it initial positions}
of the cards involved in the 3-collision of $(i,j,k)$ at time $m$. That is,
the unordered set $\{\pim(x), \pim(y), \pim(z)\} = \{i, j, k\}$.
Furthermore, $(\pim(x), \pim(y), \pim(z))$ has the same cyclic order as $(i, j, k)$.

  If $\pi_m$ has a collision $c$ of the 3-cycle  $(i, j, k)$ and
$\pim(x) \in \{i,j,k\}$ 
then we say that collision $c$
{\it involves card $x$.}
The conjugate of $(i,j,k)$ by $\pim$ is the 3-cycle
\[
  (\pim^{-1}(i), \pim^{-1}(j), \pim^{-1}(k)) \;.
  \]
Hence, collision $c$ involves $x$ if and only if the conjugate of $(i,j,k)$
by $\pim$ is a 3-cycle containing $x$.
We can interpret this as follows. If collision $c$ involves card $x$ 
then a card in position $i, j$ or $k$ at time $m$ was initially in position $x$. 

%We are almost ready for the main theorem
%about mixing times of 3-Monte shuffles. First, we state some definitions. 
% a 3-Monte shuffle on $n$ cards.
Fix an integer $t > 0$ and suppose that
$T$ is a random variable taking values
in $\{1, \dots,t\}$ that is independent of the
shuffles $(\pi_i)_{i \geq 0}$.
Such a random variable
can be used to define
a certain random set of mutually disjoint 3-cycles,
which we shall call the 
{\it 3-cycle set corresponding to $(T, t)$}. It is defined as follows. 
\begin{comment}
  and denote by $\aset(T)$;
often we shall suppress the dependence on $T$ and simply write $\aset$. 
\end{comment}
\begin{definition}
  \label{deflat}
Let $t > 0$ and 
$T$ be a random variable taking values
in $\{1, \dots,t\}$ that is independent of the
shuffles $(\pi_i)_{i \geq 0}$.
For distinct cards $x, y$ and $z$,
let $T(\{x,y,z\})$ be the time of the first
3-collision in the time interval $\{T, \dots, t\}$
that involves any card in the set $\{x, y, z\}$ (with $T(\{x,y,z\})$ undefined
if no such collision occurs).
The 3-cycle set $\seta$ corresponding to $(T, t)$ is defined by
\[
  \seta := \{ (x, y, z): \mbox{cards $(x,y,z)$ collide at time $T(\{x,y,z\})$} \}
\]
\end{definition}
\begin{definition}
  \label{matchdef}
Let $m \in \{1, 2, \dots, t\}$ and $\beta \in \cset_n$.
We call the pair $(m, \beta)$ a {\it special pair}
if there exists a 3-cycle $(x,y,z) \in \cset_n$ such that
$m = T(\{x,y,z\})$ and $\pi_{(m)} \beta \pi_{(m)}^{-1} = (x , y, z)$.

If $(m, \beta)$ is a special pair then
call the collision $c(m, \beta)$ a {\it special collision}.
\end{definition}
\noindent Note that $(m, \beta)$ is a special pair if and only if
\begin{enumerate}
  \item   $T \leq m \leq t \;;$
  \item $\beta \in \T_m\;;$
  \item If 
$\pi_{(m)} \beta \pi_{(m)}^{-1} = (x , y, z)$ 
    then none of $x, y$ or $z$ was
    involved in any collision in the time interval $\{T, \dots, m-1\}$
    (because this would contradict $T(\{x,y,z\}) = m$).
  \end{enumerate}
\noindent Note that
  $\seta$ can be written as
  \begin{equation}
    \label{whatais}
    \seta = \{ \mbox{$\pim  \beta \pimin: 1 \leq m \leq t, \beta \in \cset_n, (m, \beta)$ is a special pair}\}  
\end{equation}  
\begin{definition}
\label{deflat2}
  For nonnegative integers $m$ define
  \[
    \seta^m = \{ \mbox{$\pim \beta \pimin$ : $\beta \in \cset_n, (m, \beta)$ is a special  pair}\} \;. 
      \]
\end{definition}
\noindent Note that $\seta = \cup_{m = 1}^t \seta^m$. \\
~\\
We now verify that the 3-cycles in  $\seta$ are indeed mutually disjoint. 

\begin{lemma}
  \label{disjoint}
With probability $1$ the random set $\seta$ is a collection of mutually disjoint 3-cycles.
\end{lemma}
\begin{proof}
Note that $\seta^m$ is empty if $m < T$. 
Hence,  to show that the 3-cycles in $\seta$ are mutually disjoint it is enough to show that
      the 3-cycles in each $\seta^m$ are mutually disjoint and,
      for every $l, m \in \{T, \dots, t\}$ with $l < m$ 
      the 3-cycles in $\seta^m$
are disjoint from the 3-cycles in 
$\seta^l$. 
First, the 3-cycles in each $\seta^m$ are mutually disjoint because
they are conjugates by $\pim$  of 3-cycles in $\T_m$, which by assumption are
mutually disjoint in a 3-Monte shuffle. Now, suppose 
$l, m \in \{T, \dots, t\}$ with $l < m$ and suppose
that $(x,y,z) \in \seta^m$. Then none of $x, y$ or $z$ was
involved in any 3-collision in the time interval $\{T, \dots, m-1\}$,
which implies that none of $x, y$ or $z$ is in an element of any 3-cycle in $\seta^l$.
It follows that  the 3-cycles in $\seta^m$ are disjoint from
the 3-cycles in $\seta^l$.
This completes the proof of the lemma. 
\end{proof}
An important theme of this section is that if $\rho$ and $\tau$
are permutations in $\sym_n$ that are ``close''
in some sense then $\rho$ can be replaced with $\tau$ in
some calculations (for example, the conjugate of a certain permutation
by $\rho$ is equal to the conjugate by $\tau$).
The next two lemmas quantify this in various ways.
\begin{lemma}
  \label{samer}
  Let $\Sset$ and $\B$ be subsets of $\{1, \dots, n\}$ and let
$\{\alpha_i\}_{i=1}^l$ and $\{\beta_j\}_{j=1}^s$ 
 be collections of
  of mutually disjoint 3-cycles. 
Suppose that each $\alpha_i$ (respectively, $\beta_j$)
is a 3-cycle consisting entirely of elements of $\Sset$
(respectively, $\B$).
  Define $\alpha := \prod_{i=1}^l \alpha_i$
  and $\beta := \prod_{j=1}^s \beta_j$,
  and let 
  \[
    \rho := \alpha \tau  \beta \;.
  \]
Then $\rho^{-1}(\B) \cup \Sset = \tau^{-1}(\B) \cup \Sset$.
\end{lemma} 
\begin{proof}
%  so that $\rho = \alpha \tau \beta$.
  Note that
  \begin{equation}
    \label{rhobig}
    \rho^{-1}(\B) = (\alpha \tau)^{-1}  (\beta^{-1}(\B)) \;.
  \end{equation}
  (Recall that the permutations are written in left-to-right order.)
But since each 3-cycle $\beta_j$ consists entirely of elements of $\B$, it follows that the restriction 
of $\beta$ to $\B$ is a permutation on $\B$, and hence $\beta^{-1}(\B) = \B$. 
It follows that
  \begin{equation}
    \label{rhobigsec}
    \rho^{-1}(\B) = (\alpha  \tau)^{-1}(\B) \;,
  \end{equation}
so it remains to show that 
\begin{equation}
\label{remm}
(\alpha   \tau)^{-1}(\B) \cup \Sset = \tau^{-1}(\B) \cup \Sset \;.
\end{equation}
We need to show that if $x \in \{1, \dots, n\}$ then
\[
\mbox{$x \in 
(\alpha \tau)^{-1}(\B) \cup \Sset$ if and only if $x \in \tau^{-1}(\B) \cup \Sset$.}
\]
This is clear if $x \in \Sset$, so assume $x \notin \Sset$. 
We need to show that if $x \notin \Sset$ then 
\[
  \mbox{$x \in (\alpha \tau)^{-1}(\B)$ if and only if $x \in \tau^{-1}(\B)$.}
\]
But if $x \notin \Sset$ then $x$ is not in any of the
3-cycles $\alpha_i$ and 
hence $x$ is a fixed point of $\alpha$.
This implies that $(\alpha \tau)(x) = \tau(x)$ and
hence 
$x \in 
(\alpha \tau)^{-1}(\B)$ if and only if $x \in \tau^{-1}(\B)$. 
This proves the lemma.
\end{proof}

\begin{definition}
If $\alpha$ is a 3-cycle in $\cset_n$ and $S \subset \{1, \dots, n\}$, we
write $\alpha \perp S$ if $\alpha$ moves none of the elements of $S$.
\end{definition}

\begin{lemma}
\label{sameconj}
  Let $\Sset$ be a subset of $\{1, \dots, n\}$ and let
$\{\alpha_i\}_{i=1}^l$ and $\{\beta_j\}_{j=1}^s$ 
 be collections of
  of mutually disjoint 3-cycles. 
Suppose that each $\alpha_i$ %(respectively, $\beta_j$)
is a 3-cycle consisting entirely of elements of $\Sset$.
%(respectively, $\B$).
  Define $\alpha := \prod_{i=1}^l \alpha_i$
  and $\beta := \prod_{j=1}^s \beta_j$,
  let 
  \[
    \rho := \alpha \tau  \beta \;,
  \]
  and
  for $r$ with $1 \leq r \leq s$ let $\btil_r$ (respectively, $\beta'_r$) be $\beta_r$ conjugated by $\rho$ (respectively, $\tau$).
  Then
  \begin{enumerate}
  \item
\label{sameconja}
    $\btil_r \perp \Sset$ if and only if
  $\beta'_r \perp \Sset\;;$ 
\item
\label{sameconjb}
  when $\btil_r \perp \Sset$ and $\beta'_r \perp \Sset$ we have $\btil_r = \beta'_r$.  
\end{enumerate}
\end{lemma}

\begin{proof}
  It is sufficient to show that
  if either $\btil_r \perp \Sset$ or $\beta'_r \perp \Sset$, then $\btil_r = \beta'_r$. 
Suppose that 
either $\btil_r \perp \Sset$ or $\beta'_r \perp \Sset$.
Note that
\begin{eqnarray}
\rho &=&  \alpha  \tau \beta_r  \left( \prod_{j \neq r} \beta_j \right)     \\
 &=& \alpha \beta'_r \tau  \left( \prod_{j \neq r} \beta_j \right)    \label{rho1}       \;,
\end{eqnarray}
where the second line follows from the fact that $\beta'_r$ is $\beta_r$ conjugated by $\tau$. 
Also, note that
    \begin{eqnarray}
      \rho &=& \rho \beta_r \beta^{-1}_r   \\
           &=& \btil_r \rho \beta^{-1}_r \\
           &=& \btil_r \alpha
               \tau \left( \prod_{j \neq r} \beta_j \right)      \label{rho2}   \;,
    \end{eqnarray}
    where the second line follows from the fact that $\btil_r$ is $\beta_r$
    conjugated by $\rho$.
%    and the third line follows from the fact that the $\beta_j$ are mutually disjoint.
Equating the expressions for $\rho$ in (\ref{rho1}) and (\ref{rho2}), respectively, and then canceling the common terms gives
\begin{equation}
  \label{tworhos}
\alpha \beta'_r   = \btil_r \alpha \;.
\end{equation}
If $\beta'_r \perp \Sset$ then $\beta'_r$ commutes with every $\alpha_i$ and hence
we can rewrite the lefthand side of (\ref{tworhos}) as
\[
\beta'_r  \alpha \;,
\]
so $\beta'_r = \btil_r$. On the other hand, if $\btil_r \perp \Sset$
then $\btil_r$ commutes with every $\alpha_i$ and hence
we can rewrite the righthand side of (\ref{tworhos}) as
\[
\alpha \btil_r \;,
\]
so $\beta'_r = \btil_r$.
Thus, 
if either $\btil_r \perp \Sset$ or $\beta'_r \perp \Sset$
 then we have $\btil_r = \beta'_r$, which proves the lemma. 
\end{proof}
We will also need the following lemma. 
\begin{lemma}
        \label{sigprop}
        If 
      $\f_1$ and $\f_2$ are independent $\sigma$-fields and $\f_3$ is independent of $\sigma(\f_1, \f_2)$, then
      $\f_1$ is independent of $\sigma(\f_2, \f_3)$.
    \end{lemma}
    \begin{proof}
        If 
      $\f_1$ and $\f_2$ are independent $\sigma$-fields and $\f_3$ is independent of $\sigma(\f_1, \f_2)$, then
      the $\sigma$-fields $\f_1, \f_2$ and $\f_3$ are independent. Since for a collection of
      independent $\sigma$-fields,
the $\sigma$-fields generated by disjoint subcollections are independent
(see \cite{durrett}, Theorem 2.1.9),
      it follows that
      $\f_1$ is independent of $\sigma(\f_2, \f_3)$.
\end{proof}

If $\f$ is a $\sigma$-field and $X$ is a random variable then we write
$X \in \f$ to indicate that $X$ is measurable with respect to $\f$.

Let $I$ be a finite set and suppose that 
$(X(i))_{i \in I}$ is an
indexed family of discrete random variables.
%Formally, we define
%$( X(i): i \in I)$ as
%the set of ordered pairs
%\[
%  \{(i, X(i)): i \in I\}\;.
%\]
If $S$ is a random subset of $I$,
we denote by 
%we define the $\sigma$-field
$\sigma\Bigl( X(i): i \in S\Bigr)$
the $\sigma$-field generated by $(X(i))_{i \in S}$.
More precisely, 
\[
  \sigma \Bigl( X(i) : i \in S  \Bigr) :=
  \sigma \Bigl( S, (X(i) \one(i \in S))_{i \in I} \Bigr) \;.
\]
Note that $\sigma \Bigl( X(i): i \in S \Bigr)$
is the smallest $\sigma$-field that contains all events of the form
\[ 
  \{S = S_0\} \cap \{X(i) =x_i \mbox{ for $i \in S_0$} \} \;.
\]
%for each $S_0 \subset I$
%and for all possible values $x_i$ of the $X(i)$ for $i \in S_0$.

If $Y$ is a random variable, we write
\[
  \sigma\Bigl( Y, ( X(i): i \in S) \Bigr)
\]
for the $\sigma$-field generated by $Y$ and
$\sigma\Bigl( X(i): i \in S\Bigr)$. 

We will need the following lemma.
\begin{lemma}
  \label{exto}
  For $i \in I$ let $X(i)$ and $Y(i)$ be discrete random variables
  and suppose $\f$ is a $\sigma$-field.
Suppose that each 
$Y(i) \in \f$ and suppose that 
$X(i) = Y(i)$ whenever $i \in S$, where
$S$ is a random subset of $I$
that is measurable with respect to $\f$.

Then
$(X(i))_{i \in S}$
is measurable with respect to $\f$.
\end{lemma}
\begin{proof}
For any $S_0 \subset I$ and $(x_j)_{j \in I}$, the event 
\[
    \{S = S_0\} \cap \{ \mbox{$X(i) = x_i$ for $i \in S_0$}\}
\]
is equal to the event 
\[
  \{S = S_0\} \cap  \{ \mbox{$Y(i) = x_i$ for $i \in S_0$}\} \;, 
  \]
which is measurable with respect to $\f$
  since $S$ and the $Y(i)$ are measurable
  with respect to $\f$.
\end{proof}

We will also need the following lemma. 
\begin{lemma}
  \label{randset2}
  Fix a positive integer $m$ and suppose that
  $( X(i))_{1 \leq i \leq m}$
  are independent binary random variables
  with the uniform distribution
  and let $\{\a_i, \b_i\}$ be 
  the set of possible values of $X(i)$.   
  Let $Y$ be a random variable
that is independent of $( X(i))_{1 \leq i \leq m}$.
%For $S \subset \{1, \dots, m\}$, let $\scomp$
%denote the complement of $S$ in $\{1, \dots, m\}$.
Let $B$ be a random subset of $\{1, 2, \dots, m\}$ with the property that
for every $i$ with $1 \leq i \leq m$,
the event $\{i \in B\}$
is measurable with respect to $\sigma\Bigl( Y,
(X(i): i \in  (\bcomp \cap \{1, 2, \dots, i-1) )    \Bigr)$,
where $\bcomp$ 
denotes the complement of $B$ in $\{1, \dots, m\}$.
Then, conditional on $\sigma\Bigl(Y, (X(i) : i \in \bcomp ) \Bigr)$ 
the random variables $X(i)$ for $i \in B$ are independent
with $X(i)$ having the uniform
distribution over $\{\a_i, \b_i\}$. 

\end{lemma}
\begin{remark} Lemma \ref{randset2} can be interpreted as follows. There are $m$ cards face-down on a table.
  Card $i$ has value $X(i)$. 
  % (So $B$ is the set of cards that remain face down.)
  You consider each card, one by one, and either turn it face up or leave it face down.
  The set $B$ consists of the cards you leave face down. 
  The rule is that the decision for each card
 is based only on the values of the cards you have previously turned face up and the random variable
 $Y$ (which is independent of the values of the cards).
 The lemma says that the
distribution of the
 values of the cards that remain face down
is not affected by the values of the cards that you turn over or by $Y$. 
  \end{remark}
  \begin{proof}[Proof of Lemma \ref{randset2}]
        For subsets $S$ of $\{1, \dots, m\}$ and $j$ with $0 \leq j \leq m$, define
    \[
      S^\sexp{j} := S \cap \{1, \dots, j\} \;.
    \]
    and for $j$ with $1 \leq j \leq m$ let $S(j)$ be the
    indicator of the event $\{j \in S\}$.

    Let $B_0 \subset \{1, \dots, m\}$, let $y$ be a possible value of $Y$
    and for $i$ with $1 \leq i \leq m$,
    let $x_i \in \{\a_i, \b_i\}$, and suppose that
    \begin{equation}
      \label{nzprob}
      \P(\mbox{$Y = y, B=B_0$ and $X(i) = x_i$ for $i \in \bcomp_0$}) > 0 \;.
    \end{equation}

    We claim that
    $B = B_0$ on the event
    \[
      \{Y = y \} \cap \{ \mbox{$X(i) = x_i$ for $i \in \bcomp_0$} \} \,.
      \]
    We will prove that for all $j$ with $0 \leq j \leq m$
    we have
    \begin{equation}
      \label{litin}
      \mbox{$B^\sexp{j} = B_0^\sexp{j}$ on the event $\{Y = y \}
        \cap \{X(i) = x_i$ for $i \in \bcomp_0$\}}
    \end{equation}
    by induction on $j$. 
      The base case is trivial because $B^\sexp{0} = B_0^\sexp{0} = \emptyset$. 
      Now fix $j$ with $0 \leq j < m$ and assume that (\ref{litin}) holds for the value $j$.
      We need to show that on the event 
      \[
        \{B^{\,\sexp{j}} = B_0^\sexp{j} \} \cap \{ Y = y \} \cap
        \{ \mbox{$X(i) = x_i$ for $i \in \bcomp_0$}   \} \;,
      \]
      we must also have $B(j+1) =  B_0(j+1)$.   
     By assumption of the Lemma,  the function $B(j+1)$ is measurable with
    respect to $\sigma \Bigl( Y, (X(i) : i \in  \bcomp^\sexp{j} )\Bigr)$.
Hence
    $B(j+1)$
    is constant on the event 
    \[
      \{Y = y \} \cap \{ B^\sexp{j} = B_0^{\sexp{j}}\} \cap
      \{ \mbox{$X(i) = x_i$ for
      $i \in \bcomp^{\,\sexp{j}}$} \} \;.
      \]
     Therefore, we must have
     $B(j+1) = B_0(j+1)$
    on the event 
    \[
      \{Y = y \} \cap \{B^\sexp{j} = B_0^\sexp{j}\}  \cap \{
      \mbox{$X(i) = x_i$ for $i \in \bcomp^{\,\sexp{j}}$} \} \;,
      \]
      since otherwise equation (\ref{nzprob}) would be violated. This completes the proof of the claim
      by mathematical induction. 
The claim implies that
$B = B_0$ on the event
\[
  \{Y = y \} \cap \{ \mbox{$X(i) = x_i$ for $i \in \bcomp_0$} \}\;,
\]
and hence
   \begin{eqnarray*}
      & &
      \P( \mbox{$X(i) = x_i$ for $i \in B_0$} \given \mbox{$Y = y, B=B_0$ and $X(i) = x_i$ for $i \in \bcomp_0$}) \\
      &=&
      \P( \mbox{$X(i) = x_i$ for $i \in B_0$} \given \mbox{$Y = y$ and $X(i) = x_i$ for $i \in \bcomp_0$}) \\
     &=& \P( \mbox{$X(i) = x_i$ for $i \in B_0$}),
    \end{eqnarray*}
    where the last line follows from the fact that the $X(i)$ are independent and $Y$ is independent
    of the $X(i)$ and so Lemma \ref{sigprop}
    implies that $(X(i))_{i \in B_0}$ is independent of
    \[
      \sigma(Y, (X(i) : i \in \bcomp_0) )\;.
      \]
    This completes the proof.
    \end{proof}

\begin{corollary}
  \label{rando2}
  Fix a positive integer $l$, define  $I := \{1, 2, \dots, l\}$,
 and let $J$ be a finite set. 
Suppose that
$( X(i,j))_{i \in I, j  \in J}$ are independent binary random variables
with the uniform distribution
and let
$\{\a_{i,j}, \b_{i,j}\}$
be the set
of possible values of 
$X(i,j)$. Let $Y$ be a random variable
that is independent of $( X(i,j))_{i \in I, j  \in J}$.

Let $B$ be a random subset of $I \times J$ with the property that
for every $(i, j) \in I \times J$,
the event $\{(i,j) \in B\}$
is measurable with respect to $\sigma( Y, \f_{i-1})$
where $\f_0$ is the trivial $\sigma$-field and for $k$ with $1 \leq k \leq l-1$ we define
\[
  \f_k :=  \sigma(X(i,j) : (i,j)
  \in (\bcomp \cap (\{1, \dots, k\} \times J)) ) \;.
\]
Here, $\bcomp$ denotes the complement of $B$ in
in $I \times J$.

%For $S \subset I \times J$, let $\scomp$
%denote the complement of $S$ in $I \times J$.
Then, conditional on $\sigma(Y, (X(i) : i \in \bcomp))$ the
random variables $X(i,j)$ for $i \in B$
are independent with $X(i,j)$ having the uniform distribution over $\{\a_{i,j}, \b_{i,j}\}$. 
\end{corollary}
\begin{remark}
  Corollary \ref{rando2} can be interpreted in a way similarly to Lemma \ref{randset2}.
  This time, the cards are in $l$ piles.
  You consider each card, one by one, and either turn it face up or leave it face down.
You consider all the cards in one pile before moving to the next pile. 
The rule is that the decision for each card
 is based only on the values of cards you turned face up in previous piles, and the random variable
 $Y$.
The corollary says that the
distribution of the
 values of the cards that remain face down
is not affected by the values of the cards that you turn over or by $Y$. 
\end{remark}
\begin{proof}
  Order the elements of $J$ arbitrarily, and for
  integers $r$ with $1 \leq r \leq |I \times J|$, let
  $(i_r, j_r)$ be the $r$th element in the lexicographical order of $I \times J$. 
  Define $X'(r) := X(i_r, j_r)$ and let $B'$ be the random subset of $\{1, \dots, |I \times J|\}$
such that $r \in B'$ if and only if $(i_r, j_r) \in B$. The sequence $X'(r)$ and the
random set $B'$ satisfy the conditions of Lemma \ref{randset2}, so Corollary \ref{rando2}
follows.
\end{proof}

We are now ready for our main lemma about 3-Monte shuffles. The lemma says, roughly speaking,
that
we can write the random permutation  that represents the deck after $t$ shuffles
in such a way that
the special collisions have been ``moved to the front.''
\begin{lemma}
  \label{represent}
Let $\pi_1, \pi_2, \dots$ be a 3-Monte shuffle on $n$ cards.
Suppose that $T$ is a random variable taking values in $\{1, \dots, t\}$
that is independent of the shuffle
and let $\seta$ be the 3-cycle set corresponding to $(T, t)$
(see Definition \ref{deflat}). 
The random permuation $\pt$ can be written as
\[
  \pt = \left( \prod_{\alpha \in \aset} \c(\alpha)  \right) \gamma,
\]
for some random permutation $\gamma$ in $\sym_n$, where
conditional on 
$(\aset, \gamma)$ the 
$\c(\alpha)$  for $\alpha \in \aset$
are distributed as independent 3-collisions
with $\c(\alpha)$ being a 3-collision of $\alpha$.
\end{lemma}
\begin{proof}
Define $I := \{1, \dots, t\}$ and suppose that for $m \in I$ 
the shuffle $\pi_m$ at time $m$ is written in 3-Monte form as
\begin{align}
\label{formof2}
\pi_m = \Phi_m \prod_{\beta \in \T_m} c(m, \beta),
\end{align}
where $\Phi_m$ is a random permutation in $\sym_n$, $\T_m$
is a random subset of $\cset_n$
and each $c(m, \beta)$ is a collision of $\beta$.
%Then  $\{c(m, \beta): \beta \in \T_m\}$ is the set of collisions that occur at time $m$. 

Define $Y := \Bigl( T,   ((\Phi_m, \T_m))_{1 \leq m \leq t} \Bigr)$.
The random variable $Y$ contains all the information about $T$ and everything about
the process 
except for the outcomes of the collisions.
Note that the collisions
$\{ c(m, \beta) \}_{m \in I, \beta \in \cset_n}$ are independent of $Y$. 
%Let $B$ be the set of special collisions in time steps $\{1, 2, \dots, t\}$.
\begin{comment}
For $m \in I$, let
\[
  B_m := \{ (m, \beta): \mbox{$(m, \beta)$ is a special collision}\} \;,
\]
be the set of special collisions at time $m$ and let
\[
B :=  \cup_{m \in I} B_m
\]
be the set of all special collisions up to time $t$. 
\end{comment}
Let $B \subset I \times \cset_n$ denote the set of all special pairs
(see Definition \ref{matchdef})
up to time $t$ and
let $\bcomp$ denote the complement of $B$ in $I \times \cset$. 
Define $B^0 := \emptyset$ and $\bcomp^0 := \emptyset$, and for $m \in I$ 
define
\[
  B^m := B \cap (\{m\} \times \cset_n) \;;
\]
\[
  \bcomp^m := \bcomp \cap (\{m\} \times \cset_n) \;.
\]
Thus 
$B^m$ (respectively,
$\bcomp^{m}$) is the set of special pairs (respectively, nonspecial pairs)
at time $m$. (Note that a nonspecial pair might not even be in $\T_m$.) 
For $m \in I$ 
define
\begin{eqnarray*}
  B^{\sexp{m}} &:=& \cup_{l=1}^m B^l \\
               &=& B \cap (\{1, 2, \dots, m\} \times \cset_n) \;;
\end{eqnarray*}
\begin{eqnarray*}
  \bcomp^{\,\sexp{m}} &:=& \cup_{l=1}^m \bcomp^l \\
                      &=& \bcomp \cap (\{1, 2, \dots, m\} \times \cset_n) \;.
\end{eqnarray*}
Thus 
$B^{\sexp{m}}$ (respectively,
$\bcomp^{\,\sexp{m}}$) is the set of special pairs (respectively, nonspecial pairs)
in time steps
$\{1, 2, \dots, m\}$.
Note that $B^{\sexp{t}} = B$. 
Let $\f_{-1}$ and $\f_0$ be the trivial $\sigma$-field, and 
for $m \in I$ let
$\f_m$ be the $\sigma$-field generated by all nonspecial collisions
in time steps
$\{1, 2, \dots, m\}$, that is,
\[
  \f_m := \sigma \Bigl( c(l, \beta): (l,\beta) \in \bcomp^{\,\sexp{m}} \Bigr)\;.
\]
Let $\seta^m$ be as in Definition \ref{deflat2}.
  and
  for $r$ with $0 \leq r \leq t$ define
  \[
    \seta^{\sexp{r}} := \bigcup_{m=1}^r \seta^m \;.
\]
  By Lemma \ref{disjoint},
$\aset$
%set $\aset := \{ (i, M_1(i),v M_2(i): i \in \iset \}$
is a collection of mutually disjoint cycles.
Note that
\[
  \aset^{\sexp{0}} = \emptyset,\;\;\;\;\;\;\;\;\; \aset^\sexp{t} = \aset \;.
\]

\begin{definition}
  \label{wdef}
For nonnegative integers $s$ let $\Sset_s$ be the set of cards that
are involved in a 3-collision at any time $m$ in the interval  $T \leq m \leq s$.
Thus, $(m, \beta)$ is a special pair
if and only if
\begin{itemize}
\item $T \leq m \leq t$ ;
\item $\beta \in \T_m$; 
\item $\pim \beta \pi_{(m)}^{-1} \perp \Sset_{m-1}$. 
\end{itemize}
\end{definition}

The lemma will follow from the following:
\begin{claim}
\label{biginduct}
%\begin{enumerate}
%\item
%The
For all $m$ with $0 \leq m \leq t$:
\begin{enumerate}
\item \label{e1} $\Sset_m$ is measurable with respect to $\sigma(Y, \f_{m-1})\;$;
%\item \label{e2} $\seta^\sexp{(m)}$ is measurable with respect to 
%$\sigma(Y, \f_{m-1})\;$;  
\item  \label{e3} For all $\beta \in \cset_n$, the event
  $\{(m, \beta) \in B\}$ is measurable with respect
  to $\sigma(Y, \f_{m-1}) \;;$    
\item \label{e4b}
  There are random permutations $\c(\alpha)$ for $\alpha \in \aset^\sexp{m}$
  and a function $\h: B^{\sexp{m}} \to \aset^{\sexp{m}}$ such that
   each $\c(\alpha)$ is $\{\id, \alpha\}$-valued, and 
\begin{enumerate}
\item \label{itsam}
  $(\h(m, \beta))_{(m, \beta) \in B^{\sexp{m}}}$ is
  measurable with respect to $\sigma(Y, \f_{m-1})$;
 \item \label{itsab}
  $\h: B^{\sexp{m}} \to \aset^{\sexp{m}}$ is a bijection;
\item \label{coupled}
  if $\alpha \in \aset^{\sexp{m}}$, then $\c(\alpha) = \alpha$ if and only
  if $c(\h^{-1}(\alpha))$ is a 3-cycle;
  \item \label{e4}
  the random permutation $\pim$ can be written as
\begin{equation}
      \label{indu}
\pim =  \left( \prod_{\alpha \in \aset^\rsexp{m}} \c(\alpha) \right) \gamma_m,
\end{equation}
where
  $\gamma_m$ is a random permutation in $\sym_n$
  that is measurable with respect to $\sigma(Y, \f_m) \;.$
\end{enumerate}
\end{enumerate}
\end{claim}

\begin{remark}
  Condition \ref{coupled} says that the outcome of each $\c(\alpha)$ for $\alpha \in \aset^{\sexp{m}}$
is coupled with the outcome of  a certain special collision that occurred in the first $m$ time steps. 
\end{remark}

We shall prove Claim \ref{biginduct} by induction on $m$.  
The base case $m = 0$ is trivial
because 
\[
  \Sset_0 = \emptyset; \;\;\;\;\;\;\;\; \aset^{\sexp{0}} = \emptyset; \;\;\;\;\;\;\;\; B^{\sexp{0}} = \emptyset; 
  \;\;\;\;\;\;\;\; \pi_{(0)} = \id,
\]
and hence
equation (\ref{indu}) holds with $\gamma_0 = \id$.

Now fix $m$ with $0 \leq m < t$ and suppose
that $\Sset_m, B$, the random permutations
\[
  \mbox{$Z(\alpha)$ for $\alpha \in \aset^\sexp{m},$}
\]
$\h$ and $\pim$ satisfy the induction hypothesis for the value $m$.  

The permutation $\pi_{m+1}$ can be written as
\[
\pi_{m+1} = \Phi_{m+1} \prod_{\beta \in \T_{m+1}} c(m+1, \beta)\;,
\]
where $\Phi_{m+1}$ is a  permutation in $\sym_n$ and 
$\T_{m+1}$ is a set of mutually disjoint 3-cycles.
  
Thus
\begin{eqnarray}
  \pmp &=& \pim \pi_{m+1}      \nonumber    \\
       &=& \pim
\Phi_{m+1} \left( \prod_{\beta \in \T_{m+1}} c(m+1, \beta) \right)             \label{onemore}     \;.
\end{eqnarray}

By the induction hypothesis,  % $\Sset_m \in \sigma(Y, \f_m)$ and 
$\pim$
can be written as in equation (\ref{indu})
so by equation
(\ref{onemore}) 
we can write $\pmp$ as
\begin{eqnarray}
  \label{almostdone}
  \pmp
  &=&  \left(
    \prod_{\alpha \in \aset^\rsexp{m}} \c(\alpha)
  \right)
  \gamma_m 
      \Phi_{m+1}
      \left(
      \prod_{\beta \in \T_{m+1}} c(m+1, \beta)
      \right) 
  \\
&=&  \left(
    \prod_{\alpha \in \aset^\rsexp{m}} \c(\alpha)
  \right) 
    \tau_{m+1}
\left(
   \prod_{\beta \in \T_{m+1}} c(m+1, \beta)
      \right)     \label{goodform}  \;,
\end{eqnarray}
where $\tau_{m+1} := \gamma_m \Phi_{m+1}$.
By the induction hypothesis, $\gamma_m$ is measurable with respect to $\sigma(Y, \f_m)$,
and note that $\Phi_{m+1}$ is a function of $Y$. It follows that
\begin{remark}   
  \label{taumeas}
  $\tau_{m+1} \in \sigma(Y, \f_m)$. 
\end{remark}
We now show that $\Sset_{m+1}$ is measurable with respect to $\sigma(Y, \f_{m})$.
Define  $\B_{m+1}$ by
\[
\B_{m+1} = \mbox{the set of elements in any 3-cycle in $\T_{m+1}$}\;.
\]
Then 
\[
\Sset_{m+1} = 
\left\{\begin{array}{ll}
         \Sset_m \cup \pi_{(m+1)}^{-1}(\B_{m+1}) \;
         & 
\mbox{if $T \leq m+1 \leq t$;}\\
\Sset_m & 
\mbox{otherwise.}\\
\end{array}
\right.
\]
Since $T$ is a function of $Y$ and hence $T \in \sigma(Y, \f_{m})$, 
it is sufficient to show that 
\[
  \Sset_m \cup \pi_{(m+1)}^{-1}(\B_{m+1})
\]
is measurable 
with respect to $\sigma(Y, \f_{m})$. 

Note that since $\pmp$ can be written as in equation (\ref{goodform}), each 3-cycle $\alpha \in \aset^\sexp{m}$ 
consists entirely of elements of $\Sset_{m}$ and each 
$\beta \in \T_{m+1}$ consists entirely of elements of $\B_{m+1}$,
Lemma \ref{samer} implies that 
\begin{equation}
  \label{expr2}
  \Sset_m \cup \pi_{(m+1)}^{-1}(\B_{m+1}) = \Sset_m \cup \tau_{m+1}^{-1}(\B_{m+1}) \;.
\end{equation}
But by the induction hypothesis, $\Sset_m \in \sigma(Y, \f_{m-1}) \subset \sigma(Y, \f_m)$,
and by Remark \ref{taumeas} we have $\tau_{m+1} \in \sigma(Y, \f_m)$.
Also, note that $\B_{m+1}$ is a function of $Y$. 
Hence equation (\ref{expr2}) implies that
$\Sset_m \cup \pi_{(m+1)}^{-1}(\B_{m+1})$ is measurable with respect to 
$\sigma(Y, \f_m)$. 
Hence
$\Sset_{m+1}$ is measurable with respect to $\sigma(Y, \f_{m})$, verifying item \ref{e1} 
of the induction step.

Next we show that for all $\beta \in \cset_n$, the event 
$\{(m+1, \beta) \in B\}$ is measurable with respect to $\sigma(Y, \f_m)$.
Fix $\beta \in \cset_n$.
Recall that $B$ is the set of special pairs and hence
$(m+1, \beta) \in B$ if and only if
\begin{enumerate}
\item $T \leq m+1 \leq t\;;$
\item $\beta \in \T_{m+1}\;;$
\item $\pi_{(m+1)} \beta \pi_{(m+1)}^{-1} \perp \Sset_m\;.$
\end{enumerate}
Note that each 3-cycle $\alpha \in \aset^\rsexp{m}$
consists entirely of elements of $\Sset_m$.
Hence, part \ref{sameconja} of  Lemma
\ref{sameconj} applied to the permutation $\pi_{(m+1)}$ 
written as in equation (\ref{goodform}),  implies that 
\[
  \mbox{$\pi_{(m+1)} \beta \pi_{(m+1)}^{-1} \perp \Sset_m$ if and only if
    $\tau_{m+1} \beta \tau_{m+1}^{-1} \perp \Sset_m$.}
  \]
It follows that
$(m+1, \beta) \in B$ if and only if
\begin{enumerate}
\item $T \leq m+1 \leq t\;;$
\item $\beta \in \T_{m+1}\;;$
\item $\tau_{m+1} \beta \tau_{m+1}^{-1} \perp \Sset_m\;.$
\end{enumerate}
But $T$ is a function of $Y$ (and hence $T \in \sigma(Y, \f_m)$),
by Remark \ref{taumeas} we have
$\tau_{m+1} \in \sigma(Y, \f_m)$,
and by the  induction
hypothesis
$\Sset_m \in \sigma(Y, \f_{m-1}) \subset \sigma(Y, \f_m)$. It follows that
the event  
$\{(m+1, \beta) \in B\}$
is measurable with respect to $\sigma(Y, \f_m)$. This verifies item \ref{e3} of the induction step.

We can write $\T_{m+1}$ as the disjoint union $\T_{m+1} = \ttil  \sqcup \that$, where
\[
  \ttil :=  \{ \beta \in \T_{m+1}: (m+1, \beta) \in B^{m+1}\} \;,\;\;\;\;\;\;
  \that :=  \{ \beta \in \T_{m+1}: (m+1, \beta) \in \bcomp^{m+1}\} \;.
\]
Note that $\T_{m+1}$ is a function of $Y$ and
recall that we showed  that for all $\beta \in \cset_n$ the event  
$\{(m+1, \beta) \in B\}$
(which is the event  $\{(m+1, \beta) \in B^{m+1}\}$) 
is measurable with respect to $\sigma(Y, \f_m)$ in verifying item \ref{e3} of the induction step.
It follows that
\begin{remark}
  \label{thatmeas}
  $\ttil$ and $\that$ are measurable with respect to $\sigma(Y, \f_{m})$.
\end{remark}
Equation (\ref{goodform}) implies
\begin{eqnarray}
  \label{almostdone}
 \nonumber \pmp &=&
\left(
    \prod_{i \in \aset^\rsexp{m}} \c(\alpha)
  \right)
  \tau_{m+1} 
\left(
    \prod_{\beta \in \ttil} c( m+1, \beta)
    \right)
           \left(  \prod_{\beta \in \that} c(m+1, \beta)  \right)    \\
&=& \left(
    \prod_{i \in \aset^\rsexp{m}} \c(\alpha)
  \right)
  \tau_{m+1} 
\left(
    \prod_{\beta \in \ttil} c(m+1,\beta)
    \right)
           \nu_{m+1}   \label{getit} \;,
\end{eqnarray}
where
\begin{equation}
  \label{nudef}
  \nu_{m+1} := 
 \left(
\prod_{\beta \in \that} c(m+1,\beta)
           \right)    \;.
         \end{equation}
By       Remark \ref{thatmeas}, $\that$ is measurable with respect to 
$\sigma(Y, \f_{m}) \subset \sigma(Y,\f_{m+1})$,
        and for each $\beta \in \that$ we
         have $(m+1,\beta) \in \bcomp^{\,\sexp{(m+1)}}$. It follows that
         \begin{remark}
           \label{numeas}
  $\nu_{m+1}$ is measurable with respect to $\sigma(Y, \f_{m+1})$.
\end{remark}    

Our next aim will be to verify items
items \ref{itsam}, \ref{itsab} and \ref{coupled}
of the induction step.

We shall now
  \begin{itemize}
 \item 
 extend the function $\h$ to a function $\hmp$ on $B^{\sexp{(m+1)}}$ such that 
 $\hmp: B^{\sexp{(m+1)}} \to \aset^{\sexp{(m+1)}}$ is a bijection,
and $(\hmp(l, \beta))_{(l, \beta) \in B^{\sexp{(m+1)}}}$ is measurable 
with respect to $\sigma(Y, \f_m)$; 
\item 
   define random permutations $\{ \c(\alpha)\}_{\alpha \in \aset^{m+1}}$
such that   $\c(\alpha)$ is $\{\id, \alpha\}$-valued and 
if $\alpha \in \aset^{\sexp{(m+1)}}$, then 
$\c(\alpha) = \alpha$ if and only
  if $c(h_{m+1}^{-1}(\alpha))$ is a 3-cycle. 
\end{itemize}

   To extend $h_m$ to a function $h_{m+1}$ on $B^{\sexp{(m+1)}}$ we need
to define $h_{m+1}$ on $B^{m+1}$. 
For $(m+1, \beta) \in B^{m+1}$, define $\hmp(m+1, \beta)$ by
\[
  \hmp(m+1, \beta) = \tau_{m+1} \beta \tau_{m+1}^{-1} \;.
\]

We have already shown that for all $\beta \in \cset_n$ the event  
$\{(m+1, \beta) \in B\}$ is
measurable with respect to $\sigma(Y, \f_m)$ in verifying item \ref{e3} of the induction step.
It follows that we have $B^{m+1} \in \sigma(Y, \f_m)$.
Furthermore, $\tau_{m+1} \in \sigma(Y, \f_m)$ by Remark \ref{taumeas}
and hence for every $\beta \in \cset_n$ the random
permutation $\tau_{m+1} \beta \tau_{m+1}^{-1} \in \sigma(Y, \f_m)$.
Thus Lemma \ref{exto} implies
that
\[
  (\hmp(m+1, \beta))_{(m+1, \beta) \in B^{m+1}}
\]
is measurable with respect to $\sigma(Y, \f_m)$.
Combining this with
the fact that

\[
  (h_m(l, \beta))_{(l, \beta) \in      B^{\sexp{m}}}
\]
is
measurable with respect to $\sigma(Y, \f_{m-1})$ (which is part of the induction hypothesis)
gives that 
$(h_{m+1}(l, \beta))_{(l, \beta) \in      B^{\sexp{(m+1)}}}$ is 
measurable with respect to $\sigma(Y, \f_{m})$.
This verifies
item \ref{itsam} of the induction step.

Recall that $\pmp$ can be written as in equation (\ref{goodform})
and each 3-cycle $\alpha \in \aset^\sexp{m}$ 
consists entirely of elements of $\Sset_{m}$.
Hence
Part \ref{sameconjb} of Lemma \ref{sameconj} applied to the permutation $\pi_{(m+1)}$ 
written as in equation (\ref{goodform}),  implies that 
when $(m+1, \beta) \in B^{m+1}$ we have
\begin{equation}
  \label{samey1}
\pi_{(m+1)} \beta \pi_{(m+1)}^{-1} = \tau_{m+1} \beta \tau_{m+1}^{-1} \;.
\end{equation}

Hence, the image of $B^{m+1}$ under $h_{m+1}$ is 

\begin{eqnarray*}
\{ \tau_{m+1} \beta \tau_{m+1}^{-1}: (m+1, \beta) \in B^{m+1} \} &=&
\{\pi_{(m+1)} \beta \pi_{(m+1)}^{-1}: (m+1, \beta) \in B^{m+1} \}        \\
              &=& \aset^{m+1}\;,
\end{eqnarray*}
where the second line follows from Definition \ref{deflat2}.
Since conjugation by $\tau_{m+1}$ is one-to-one, it follows that:
\begin{remark}
  \label{restrict}
the restriction
of $\hmp$ to $B^{m+1}$ is a bijection from $B^{m+1}$ to $\aset^{m+1}$.
\end{remark}
Since $\aset^{m+1}$ is disjoint from $\aset^{\sexp{m}}$, it follows that the
extended function $\hmp: B^\sexp{(m+1)} \to \aset^{\sexp{(m+1)}}$ is a bijection.
This verifies item \ref{itsab} of the induction step.

Next,
we define $\c(\alpha)$ for $\alpha \in \aset^{m+1}$:
\begin{definition}
  \label{hmdef}
For $\alpha \in \aset^{m+1}$ let $\c(\alpha)$ be the conjugate
of $c( h_{m+1}^{-1}(\alpha))$ by $\tau_{m+1}$.
\end{definition}
%That is, if $h_{m+1}^{-1}(\alpha) = (m+1, \beta)$, then
Note that if $(m+1, \beta) \in B^{m+1}$ and $h_{m+1}(m+1, \beta) = \alpha$
for some $\alpha \in \aset^{(m+1)}$
(which means that $\alpha$ is the conjugate of $\beta$ by $\tau_{m+1}$),
then
\[
\c(\alpha) = 
\left\{\begin{array}{ll}
\alpha \;\;& 
\mbox{if $c(m+1, \beta) = \beta$;}\\
\id & \mbox{if $c(m+1, \beta) = \id$.}\\
\end{array}
\right.
\]

Note that $\c(\alpha)$ is a $\{\id, \alpha\}$-valued and 
$\c(\alpha) = \alpha$
if and only if the permutation $c( h_{m+1}^{-1}(\alpha)) = c(m+1, \beta)$ is a 3-cycle.
This is true for all $\alpha \in \aset^{m+1}$.
Combining this with the induction
hypothesis, we see that for all 
$\alpha \in \aset^{\,\sexp{(m+1)}}$ we have $\c(\alpha) = \alpha$  if and only
if $c(h_{m+1}^{-1}(\alpha))$ is a 3-cycle. This verifies
item \ref{coupled} of the induction step.

It remains verify item \ref{e4} of the induction step by showing that
$\pmp$ can be written as in equation (\ref{indu}),
where $\gamma_{m+1} \in \sigma(Y, \f_{m+1})$. 
Note that if $\beta \in \ttil$
then $(m+1, \beta) \in B^{m+1}$ and hence the conjugate
of
$c(m+1, \beta)$ by $\tau_{m+1}$ is $\c( h_{m+1}(m+1, \beta))$
by Definition \ref{hmdef}.
It follows that the conjugate of
\[
  \prod_{\beta \in \ttil} c(m+1,\beta)
\]
by $\tau_{m+1}$ is
\begin{eqnarray*}
  \prod_{\beta \in \ttil} \c(h_{m+1}(m+1,\beta)) &=&
                                                     \prod_{(m+1, \beta) \in B^{m+1}} \c(h_{m+1}(m+1,\beta)) \\
                                                 &=&  \prod_{\alpha \in \aset^{(m+1)}} \c(\alpha)\;,
                                              \end{eqnarray*}
where the last line follows from Remark \ref{restrict}. It follows that
\[
  \tau_{m+1} 
\left(
    \prod_{\beta \in \ttil} c(m+1,\beta)
    \right)   = \left(
\prod_{\alpha \in \aset^{(m+1)}} \c(\alpha) \right)   \tau_{m+1}    \;.
\]
Combining this with 
equation (\ref{getit}) gives
\begin{eqnarray*}
  \label{almostdone2}
  \pmp &=&
           \left(
    \prod_{\alpha \in \aset^\rsexp{m}} \c(\alpha)
  \right)
\left(
    \prod_{\alpha \in \aset^{m+1}} \c(\alpha)
    \right)
  \tau_{m+1} 
\nu_{m+1} \;, \\
&=&    \left(
    \prod_{\alpha \in \aset^\rsexp{\,(m+1)}} \c(\alpha)
  \right)
  \tau_{m+1} 
\nu_{m+1} \;, \\
&=&    \left(
    \prod_{\alpha \in \aset^\rsexp{(m+1)}} \c(\alpha)
  \right)
  \gamma_{m+1} \;, \\
  \end{eqnarray*}
  where
  \[
      \gamma_{m+1} :=  \tau_{m+1} 
      \nu_{m+1}
      \;.
    \]
    This verifies that $\pmp$ can be written as in equation (\ref{indu}).
  To complete the verification of item \ref{e4} of the induction step we must verify that
    $\gamma_{m+1}$ is measurable with respect to $\sigma(Y, \f_{m+1})$.
    By Remark \ref{taumeas},
    \[
      \tau_{m+1} \in \sigma(Y, \f_m) \subset
      \sigma(Y, \f_{m+1})
      \]
    and by Remark \ref{numeas}, $\nu_{m+1} \in \sigma(Y, \f_{m+1})$. It follows that
    $\gamma_{m+1}   :=  \tau_{m+1}\nu_{m+1}$ is measurable with respect to $\sigma(Y, \f_{m+1})$.
This verifies item \ref{e4} of the induction step and 
completes the proof of Claim \ref{biginduct} by mathematical induction.

By Claim \ref{biginduct}, for all 
$m$ with $1 \leq m \leq t$
and $\beta \in \cset_n$, the event  $\{(m, \beta) \in B\}$ is measurable with respect
to $\sigma(Y, \f_{m-1})$.
Furthermore, the collisions $\{c(m, \beta) \}_{(m, \beta)
  \in I \times \cset_n}$
are independent of $Y$.
Thus, the collisions $\{c(m, \beta) \}_{(m, \beta) \in I \times \cset_n}$,
the random variable $Y$
and the random subset $B$ of $I \times \cset_n$ satisfy the conditions of
Corollary \ref{rando2} and hence
conditional on
\[
  \sigma\Bigl(Y, (X(i): i \in \bcomp) \Bigr)
\]
the
$c(m, \beta)$ for  $(m, \beta) \in B$
are distributed as
independent 3-collisions. 

Note that
\[
  \sigma(Y, \f_t) = \sigma \Bigl( Y, (X(i): i \in \bcomp)  \Bigr)\;,
\]
recall that $\aset = \aset^{\,\sexp{t}}, B = B^\sexp{t}$ and define $h := h_t$. 
Then item \ref{coupled} of Claim \ref{biginduct} with $m = t$ implies that
%conditional on 
%$\sigma(Y, (X(i): i \in \bcomp))$, 
the $\c(\alpha)$ for $\alpha \in \aset$ 
are coupled in a one-to-one fashion with the collisions
$c(m, \beta)$ for  $(m, \beta) \in B$, with $\c(\alpha)$ coupled with $c(h^{-1}(\alpha))$.
It follows that conditional on 
$\sigma(Y, (X(i): i \in \bcomp))$, 
the $\c(\alpha)$ for $\alpha \in \aset$ 
are distributed as independent 3-collisions.

In addition, by item \ref{e4}, the random permutation $\gamma := \gamma_t$ satisfies 
\[
  \gamma \in \sigma(Y, (X(i): i \in \bcomp))\;.
\]
Furthermore,  
item \ref{itsab} of Claim \ref{biginduct} with $m = t$ implies that
the image of $B$ under $h$ is $\aset$ 
and since by item \ref{e1} 
$(h(m, \beta))_{(m, \beta) \in B}$ is measurable with respect to
\[
  \sigma(Y, \f_{t-1}) \subset \sigma \Bigl( Y, (X(i): i \in \bcomp) \Bigr) \;,
\]
we must also have
\[
  \aset \in \sigma(Y, (X(i): i \in \bcomp)) \;.
\]
So $(\aset, \gamma) \in \sigma \Bigl( Y, (X(i): i \in \bcomp) \Bigr)$ and hence
conditional on $(\aset, \gamma)$, 
the $\c(\alpha)$ for $\alpha \in \aset$ 
are distributed as independent 3-collisions.
Finally, item \ref{e4} of Claim \ref{biginduct} implies that we can write
the random permutation $\pi_{(t)}$ as
\begin{equation}
\pi_{(t)} =  \left( \prod_{\alpha \in \aset} \c(\alpha) \right) \gamma \;.
\end{equation}
This completes the proof.
\end{proof}

   \subsection{Statement of main 3-collisions/entropy theorem}

   We are almost ready for the main theorem
about mixing times of 3-Monte shuffles. First, we state some definitions. 
For the first definition, it is
helpful to know that ``intrine'' is an English word meaning ``to unite in a group of three.'' 
\begin{definition}
\label{intrined}
Let $\pi_1, \pi_2, \dots$ be a 3-Monte shuffle on $n$ cards.
Fix an integer $t > 0$ and suppose that
$T$ is a random variable taking values
in $\{1, \dots,t\}$ that is independent of the
shuffle. %s $\{\pi_i : i \geq 0\}$.
Let
$T(\{x,y,z\})$ be the time of the first
3-collision  in the time interval $\{T, \dots, t\}$
that involves any card in the set $\{x, y, z\}$.
%From the definition of $\aset(T)$ in subsection \ref{backsec},
If 
cards $(x, y, z)$ collide at time $T(\{x,y,z\})$
then we
say that $(x,y,z)$ is {\it $(T,t)$ intrined.}
\end{definition}

Note that $(x,y,z)$ is $(T, t)$ intrined if and only if $(x,y,z) \in \aset$,
where $\aset$ is the 3-cycle set corresponding to $(T,t)$. 

\begin{definition}
  \label{matchy}
Define $\M_1: \{1, 2, \dots, n\} \to
  \{1, 2, \dots, n\}$ by
\[
\M_1(x) =  
\left\{\begin{array}{ll}
y  \;\;         & 
\mbox{if for some $z$, the triple $(x, y, z)$ is $(T,t)$ intrined,}\\
x \;\;  & 
          \mbox{if 
          for all $y$ and $z$, the triple
          $(x,y,z)$ is not $(T, t)$ intrined}\\
\end{array}
\right.
\]
and define
$\M_2: \{1, 2, \dots, n\} \to
  \{1, 2, \dots, n\}$ by
\[
\M_2(x) =  
\left\{\begin{array}{ll}
z  \;\;         & 
\mbox{if for some $y$, the triple $(x, y, z)$ is $(T,t)$ intrined,}\\
x \;\;  & 
          \mbox{if 
          for all $y$ and $z$, the triple
          $(x,y,z)$ is not $(T, t)$ intrined}\\
\end{array}
\right.
\]
\end{definition}

The following is the main theorem about mixing times of 3-Monte shuffles. 
\begin{theorem} 
  \label{maintheorem}
%%Fix an integer $t > 0$, let $T$ be a $\{1, \dots, t\}$-valued
%random variable
%that is independent of the shuffle
%and for cards $k \in \{1, \dots, n\}$
%let $\M_1(k)$ and $\M_2(k)$
%be the front and back match, respectively, of $k$. 
  Let $\pi_1, \pi_2, \dots$ be a 3-Monte shuffle on $n$ cards.
  Suppose that for every card $k$ with $3 \leq k \leq n$ there is a constant
$\Boy_k$ with  $0 \leq \Boy_k \leq 1$
such that
for every $j \in \{1, \dots, k-1\}$ we have
$\P(\M_2(k)=j, \M_1(k) < k) \geq \frac{\Boy_k}{k}$.  
Let $\mu$ be a random permutation that is independent of the shuffle. % $\{\pi_i: i \geq 0\}$.
%and suppose that $\mu$ is equally likely to be odd or even.
Then
\begin{align}
  \nonumber
\ENT(\mu\pi_{(t)}) - \ENT(\mu) \leq \frac{-C}{\log n} \sum_{k=1}^{n} \Boy_kE_k,
\end{align}
where $E_k = \ENT(\mu^{-1}(k) \given \tailt_{k+1}^\mu)$ 
and $C$ is a universal constant. 
\end{theorem}
\begin{remark} In any application of Theorem \ref{maintheorem}, some care is required in choosing the
  values of $T$ and $t$. Since the decrease in relative entropy depends on $\sum_{k=1}^n \Boy_k E_k$,
  the values of $T$ and $t$ should be chosen so that the bounds $\Boy_k$ tend to be significant when
  the corresponding $E_k$ is significant. For the application in the present paper, $T$ is deterministic
  and $t$ is of the same order as $T$. 
\end{remark}

\begin{proof}
  By Lemma \ref{represent},  $\pt$ can be written as
\[
  \pt = \left( \prod_{\alpha \in \aset} \c(\alpha)  \right) \gamma,
\]
for some random permutation $\gamma$ in $\sym_n$, where
$\aset$ is the 3-cycle set corresponding to $(T,t)$
%(see Definition \ref{deflat})
and
conditional on 
$(\aset, \gamma)$ the 
$\c(\alpha)$  are
are distributed as independent 3-collisions
with $\c(\alpha)$ being a 3-collision of $\alpha$.

Recall that
by Lemma \ref{disjoint}, the set $\seta$ consists of
mutually disjoint 3-cycles.
For sets $A$ of mutually disjoint 3-cycles and
permutations $g$ in $\sym_n$, let
$p_\mu^{A, g}$ denote the conditional
distribution of $\mu \pt$ given that $\seta = A$ and $\gamma = g$.
By Lemma \ref{convy},
\begin{equation}
  \label{lessthan}
  \ent( \mu \pt) \leq \sum_{A,g} \P( \seta = A, \gamma = g) \ent(p_\mu^{A,g}) \;,
\end{equation}
where the sum is over possible values $(A, g)$ of $(\seta, \gamma)$. 
Note that $p_\mu^{A, g}$ is the distribution of
\[
  \mu \nu^A g \;,
\]
where
\[
  \nu^A = \prod_{\alpha \in A} c(\alpha),
\]
for independent collisions  $\{c(\alpha) \}_{\alpha \in A}$.
Since 
$g$ is a deterministic permutation, 
Lemma \ref{convy} implies that
\[
  \ent( p_\mu^{A, g} ) = \ent( \mu \nu^A) \;.
\]
Combining this with equation (\ref{lessthan}) gives
\begin{eqnarray}
  \ent( \mu \pt) &\leq& \sum_{A,g} \P( \seta = A, \gamma = g) \ent( \mu \nu^A   )   \nonumber  \\
                 &=& \sum_A \P( \seta = A) \ent( \mu \nu^A)    \label{nextto}  \;.
\end{eqnarray}
Note that the definition of 
$\M_1(k)$ and $\M_2(k)$ in Definition \ref{matchy} is consistent with 
Definition \ref{beforeandafter} applied to the set $\seta$, and hence
%an assumption of the theorem is that
%for each $j \in \{1, \dots, k-1\}$ we have
%\[
%  \P(\seta_2(k)=j, \seta_1(k) < k) \geq \frac{\Boy_k}{k} \;.
%\]
%Thus
Theorem \ref{enttheorem} implies that
\[
  \sum_A \P( \seta = A)  \Bigl[ \ent( \nu^A) - \ent( \mu)  \Bigr] \leq -  \frac{C}{\log n} \sum_{k=3}^{n} \Boy_k E_k \;,
\]
and hence
\[
  \sum_A \P( \seta = A) \ent( \nu^A)  \leq \ent( \mu) -  \frac{C}{\log n} \sum_{k=3}^{n} \Boy_k E_k \;.
\]
Combining this with equation (\ref{nextto}) proves the theorem.
\end{proof}

\section{The Torus Shuffle}
\label{fitsin}

In this section we give more details about how the torus shuffle
fits into the setup described in previous sections. 
In order to use Theorem \ref{maintheorem} it will be helpful to label the
positions of the tiles so that
tiles that are close in the grid have a similar label. To this end, we
label the positions 
so that each $\ell \times \ell$ square $B_\ell$ in the lower
left part of the grid
contains the labels $1, \dots, \ell^2$. Any method of labeling with this
property will suit our purposes.
(Figure \ref{fig:grid_labeling} below shows one such labeling.) 
\begin{figure}[!h]
  \centering
\begin{tikzpicture}
%first sequence
\draw[step=0.5cm,color=gray] (-1,-1) grid (1,1);
\node at (-0.75,+0.75) {10};
\node at (-0.25,+0.75) {11};
\node at (+0.25,+0.75) {12};
\node at (+0.75,+0.75) {13};
\node at (-0.75,+0.25) {5};
\node at (-0.25,+0.25) {6};
\node at (+0.25,+0.25) {7};
\node at (+0.75,+0.25) {14};
\node at (-0.75,-0.25) {2};
\node at (-0.25,-0.25) {3};
\node at (+0.25,-0.25) {8};
\node at (+0.75,-0.25) {15};
\node at (-0.75,-0.75) {1};
\node at (-0.25,-0.75) {4};
\node at (+0.25,-0.75) {9};
\node at (+0.75,-0.75) {16};
\end{tikzpicture}
  \caption{Labeling the grid in the case $n=4$.}
  \label{fig:grid_labeling}
\end{figure}

Recall that the torus
shuffle has the following transition rule: at each step,
with probability $\half$
select a row or column, uniformly at random from all $2n$ possibilities,
and then cyclically rotate it by one unit in a random direction.

For the convenience of the reader, we state again how the torus shuffle
can be written in 3-Monte form. 
As observed in Subsection \ref{backsec},
the Markov chain that does two steps
of the torus shuffle at a time can be simulated as follows.
Generate two steps of the torus shuffle. 
If the two steps are both row moves or both column moves or contain lazy steps, perform them as usual.
Otherwise, say the moves are the row move $r$ and the column move $c$, in some order. In this case:
\begin{enumerate}
\item Perform $c$ and then $r$; then 

\item  Perfom the 3-cycle $\Gamma = r^{-1}c^{-1} rc$ with probability $\half$.
\end{enumerate}

Since Step 2 above is a 3-collision, this process is 3-Monte.
Call a triple of adjacent squares in the grid {\it gamma-shaped}
if the squares are not all in the same row or column. 
Note that the 3-cycle $\Gamma$ described in Step 2 above involves squares that form a gamma-shaped
triple. (See Figure \ref{fig:3-collision}
of Subsection \ref{backsec} for a typical case. )
%The gamma shape could be inverted or rotated depending on the directions of the row and column moves.)
The ``middle'' square of the ``gamma''
is where the row and column intersect and can be anywhere on the board.

\subsection{Communicating classes}
The communicating classes of the torus shuffle Markov chain were
determined by Amano et al \cite{solve}, who were motivated by the torus puzzle.  
In the torus puzzle, the goal is to put the tiles in order starting from some specified starting state.
Amano et al investigated the question of
which configurations are reachable from
the identity permutation, that is, which configurations in the torus puzzle
are solvable. 
They showed that if $n$ is even then all states are reachable, which implies that the torus shuffle Markov chain
is irreducible in this case.
When $n$ is odd, they showed that there are two communicating classes, namely the set of even permutations and the
set of odd permutations. (Note that when $n$ is odd, each move of the chain is a cycle of odd length, hence an even permutation.)
\begin{remark}
  \label{even}
In the case where $n$ is odd we shall assume that the starting state is
equally likely to be an odd or even permutation to ensure that the chain still
converges to the uniform distribution. 
\end{remark}

\section{The 3-stage procedure}
\label{threestage}

Theorem \ref{maintheorem} reduces bounding the mixing time of a card shuffle to studying the
behavior of triplets of cards. Recall that in the context of the torus shuffle, we call the cards
{\it tiles}. Thus tile $i$ refers to the tile initially in position $i$. 
The following is our main technical lemma. 

\begin{lemma}
  \label{mainlemma}
  Suppose $n \geq 4$ and consider the torus shuffle on an $n \times n$
  grid.  Fix $l \geq 2$.  There exists a real number $\C \in [4, 5]$ and a universal constant
  $D>0$ such that if $T' := \C\ell^2 n$, then $T'$ is an even 
integer and for any triplets of positions $(i, j, k)$
  and $(i', j', k')$ in an $\ell \times \ell$ box $B_l$ within the
  $n \times n$ grid, the probability that tile $i$ travels to position
  $i'$, tile $j$ travels to position $j'$, and tile $k$ travels to
  position $k'$ in $T'$ steps is at least $\frac{D}{\ell^6}$.
\end{lemma}

Lemma \ref{mainlemma} is a kind of crude
local central limit theorem for tiles in the torus shuffle. The main complication
in the proof is that the tiles don't move independently when they are in the same row or column. The way we handle this
is to define three small boxes $\boxi$, $\boxj$ and $\boxk$ that are far away from each other,
both horizontally and vertically. These boxes have side lengths that are a small constant times $\ell$.
Then we divide the time steps into three stages, and we lower bound the probability that 
\begin{enumerate}
\item In Stage 1, tiles $i$, $j$, and $k$ move to the boxes $\boxi$, $\boxj$ and $\boxk$, respectively. 
\item In Stage 2 (with the tiles now moving independently), each tile moves to the ``correct spot''
      in its box. 
\item In Stage 3, the tiles that move to positions $i'$, $j'$, and $k'$ come from
      boxes $\boxi$, $\boxj$ and $\boxk$, respectively.
\end{enumerate}
In Stage $2$, ``correct spot'' means the exact spot that will go to the desired position in Stage 3. \\
~\\
The 3-stage procedure is illustrated in Figure \ref{fig:3_part_process}. 
\begin{figure}[!h]
  \centering
\begin{tikzpicture}
%big boxes
\draw (-5,0) rectangle (-1,4);
\draw (0,0) rectangle (4,4);
\draw (5,0) rectangle (9,4);
%small boxes
\draw (4/9-5, 28/9) rectangle (8/9-5, 32/9);
\draw (16/9-5, 16/9) rectangle (20/9-5, 20/9);
\draw (28/9-5, 4/9) rectangle (32/9-5, 8/9);
\draw (4/9, 28/9) rectangle (8/9, 32/9);
\draw (16/9, 16/9) rectangle (20/9, 20/9);
\draw (28/9, 4/9) rectangle (32/9, 8/9);
\draw (4/9+5, 28/9) rectangle (8/9+5, 32/9);
\draw (16/9+5, 16/9) rectangle (20/9+5, 20/9);
\draw (28/9+5, 4/9) rectangle (32/9+5, 8/9);
%label boxes
\node at (10/9-5, 33/9){$\mathcal B_i$};
\node at (22/9-5, 21/9){$\mathcal B_j$};
\node at (34/9-5, 9/9){$\mathcal B_k$};
%label points
\node at (4/9-5, 9/9){.};
\node at (2/9-5, 9/9){$i$};
\node at (7/9-5, 7/9){.};
\node at (5/9-5, 7/9){$j$};
\node at (10/9-5, 5/9){.};
\node at (8/9-5, 5/9){$k$};
\node at (26/9+5, 30/9){.};
\node at (27/9+5, 32/9){$i'$};
\node at (28/9+5, 26/9){.};
\node at (29/9+5, 28/9){$j'$};
\node at (32/9+5, 20/9){.};
\node at (33/9+5, 22/9){$k'$};
\node at (18/9-2.5, 18/9){$\implies$};
\node at (18/9+2.5, 18/9){$\implies$};
%label arrows
\draw[->] (4/9-5, 9/9) -- (5/9-5, 30/9);
\draw[->] (7/9-5, 7/9) -- (17/9-5, 17/9);
\draw[->] (10/9-5, 5/9) -- (29/9-5, 6/9);
\draw[->] (5/9, 30/9) .. controls (8/9, 28/9) .. (7/9, 31/9);
\draw[->] (17/9, 17/9) .. controls (18/9, 19/9) .. (19/9, 18/9);
\draw[->] (29/9, 6/9) .. controls (28/9, 7/9) .. (30/9, 5/9);
\draw[->] (7/9+5, 31/9) -- (26/9+5, 30/9);
\draw[->] (19/9+5, 18/9) -- (28/9+5, 26/9);
\draw[->] (30/9+5, 5/9) -- (32/9+5, 20/9);
\end{tikzpicture}
  \caption{The three-stage procedure.}
  \label{fig:3_part_process}
\end{figure}

We will show that Stages 1 and 3 are completed successfully with probability
bounded away from $0$ and Stage 2 is completed successfully with probability $\bigomega(\ell^{-6})$. \\
\\
{\bf Method of Proof:}
For Stages 1 and 3, we approximate the motion of the tiles by three
independent random walks and then apply the central
limit theorem. (Note that the number of steps is chosen so that the
standard deviation of a tile's position is on the order of $l$.)
There is some error introduced by ``interference,''
that is, the non-independent moves that occur when two tiles are in the same
row or column. We handle this by  showing
that the size of this error is small compared
to the length of the boxes. This amounts to bounding the number of moves
where two tiles are in the same row or column, which we accomplish
using some martingale arguments. 

  For Stage 2, we need to bound the probability that the tiles
go to three specific positions in their boxes. For this we use
three applications of the local central limit theorem, which we can do because
the tiles are all now far away from each other (assuming the success of
Stages 1 and 3) and hence moving independently. 

\begin{proof}[Proof of Lemma \ref{mainlemma}]
First we will show that the lemma holds when $l$ is sufficiently large. 
  Consider the $\ell \times \ell$ box $B_\ell$ whose lower left corner is at position $(0,0)$ and
  whose upper right corner is at $(\ell, \ell)$.
Fix a positive even integer $K$ large enough so that
$e^{-{1 \over 3}(K-2)^2} \leq {e^{-4} \over 4 \sqrt{3}}$,
and suppose that $c$ and $\delta$ are constants small enough so that
\[
  cK < {1 \over 3}; \;\;\;\;\; \delta < {c \over 2} \,.
\]
(The constant $K$ won't be used until the analysis of Stage 2.)
%For simplicity, we shall assume hereafter that $c \ell$ is an even integer.
%This does not change the analysis significantly. 
  Let $\boxi$, $\boxj$ and $\boxk$ be  boxes (i.e., squares) of side length
  $cl$ centered at $( {\ell \over  6}, {5l \over 6})$, $({\ell \over  2}, {l \over 2})$, and
  $( {5\ell \over  6}, {l \over 6})$, respectively.  
  Let $\li$, $\lj$ and $\lk$ be  boxes of side length
  $cl - 2 \delta l$, (also) centered at $({\ell \over  6}, {5l \over 6})$,
  $( {\ell \over  2}, {l \over 2})$, and
  $( {5\ell \over  6}, {l \over 6})$, respectively.  
  The boxes $\li$, $\lj$ and $\lk$ are slightly smaller and fit inside $\boxi$, $\boxj$ and $\boxk$, respectively.
  (See Figure \ref{fig:3-part setup}.)

\begin{figure}[!h]
  \centering
\begin{tikzpicture}
%left box
%big box
\draw (-5,0) rectangle (-1,4);
%small boxes
\draw (4/9-5, 28/9) rectangle (8/9-5, 32/9);
\draw (16/9-5, 16/9) rectangle (20/9-5, 20/9);
\draw (28/9-5, 4/9) rectangle (32/9-5, 8/9);
\node at (10/9-5, 33/9){$\mathcal B_i$};
\node at (22/9-5, 21/9){$\mathcal B_j$};
\node at (34/9-5, 9/9){$\mathcal B_k$};
%draw braces for length measurements
%horizontal braces
%\draw[decoration={brace,raise=2pt},decorate]     (-5,4) -- node[above=4pt] {$\frac{\ell}{9}$} (4/9-5,4);
\draw[decoration={brace,raise=2pt},decorate] 
    (4/9-5,4) -- node[above=4pt] {$c\ell$} (8/9-5,4);
%\draw[decoration={brace,raise=2pt},decorate]    (8/9-5,4) -- node[above=4pt] {$\frac{2\ell}{9}$} (16/9-5,4);
\draw[decoration={brace,raise=2pt},decorate]     (16/9-5,4) -- node[above=4pt] {$c\ell$} (20/9-5,4);
%\draw[decoration={brace,raise=2pt},decorate]     (20/9-5,4) -- node[above=4pt] {$\frac{2\ell}{9}$} (28/9-5,4);
\draw[decoration={brace,raise=2pt},decorate]     (28/9-5,4) -- node[above=4pt] {$c\ell$} (32/9-5,4);
%\draw[decoration={brace,raise=2pt},decorate]     (32/9-5,4) -- node[above=4pt] {$\frac{\ell}{9}$} (36/9-5,4);
%vertical braces
%\draw[decoration={brace,raise=2pt},decorate]     (-5,0) -- node[left=4pt] {$\frac{\ell}{9}$} (-5,4/9);
\draw[decoration={brace,raise=2pt},decorate]     (-5,4/9) -- node[left=4pt] {$c\ell$} (-5,8/9);
%\draw[decoration={brace,raise=2pt},decorate]     (-5,8/9) -- node[left=4pt] {$\frac{2\ell}{9}$} (-5,16/9);
\draw[decoration={brace,raise=2pt},decorate]     (-5,16/9) -- node[left=4pt] {$c\ell$} (-5,20/9);
%\draw[decoration={brace,raise=2pt},decorate]     (-5,20/9) -- node[left=4pt] {$\frac{2\ell}{9}$} (-5,28/9);
\draw[decoration={brace,raise=2pt},decorate]    (-5,28/9) -- node[left=4pt] {$c\ell$} (-5,32/9);
%\draw[decoration={brace,raise=2pt},decorate]     (-5,32/9) -- node[left=4pt] {$\frac{\ell}{9}$} (-5,36/9);
\draw [dashed] (-5,32/9) -- (4/9-5,32/9);
\draw [dashed] (-5,28/9) -- (4/9-5,28/9);
\draw [dashed] (-5,20/9) -- (16/9-5,20/9);
\draw [dashed] (-5,16/9) -- (16/9-5,16/9);
\draw [dashed] (-5,8/9) -- (28/9-5,8/9);
\draw [dashed] (-5,4/9) -- (28/9-5,4/9);
\draw [dashed] (4/9-5,4) -- (4/9-5,32/9);
\draw [dashed] (8/9-5,4) -- (8/9-5,32/9);
\draw [dashed] (16/9-5,4) -- (16/9-5,20/9);
\draw [dashed] (20/9-5,4) -- (20/9-5,20/9);
\draw [dashed] (28/9-5,4) -- (28/9-5,8/9);
\draw [dashed] (32/9-5,4) -- (32/9-5,8/9);
%right box
\draw (0.5,0.5) rectangle (3.5,3.5);
\node at (3.8, 0.4){$\mathcal B_k$};
\draw (1,1) rectangle (3,3);
\node at (3.2, 0.9){$\mathcal L_k$};
\draw[decoration={brace,mirror,raise=2pt},decorate] 
    (0.5,0.5) -- node[below=4pt] {$c\ell$} (3.5,0.5);
\draw[decoration={brace,mirror,raise=2pt},decorate] 
    (3.5,0.5) -- node[right=4pt] {$c\ell$} (3.5,3.5);
\draw[decoration={brace,raise=2pt},decorate] 
    (0.5,3.5) -- node[above=4pt] {$\delta\ell$} (1,3.5);
\draw[decoration={brace,raise=2pt},decorate] 
    (3,3.5) -- node[above=4pt] {$\delta\ell$} (3.5,3.5);
\draw[decoration={brace,raise=2pt},decorate] 
    (0.5,3) -- node[left=4pt] {$\delta\ell$} (0.5,3.5);
\draw[decoration={brace,raise=2pt},decorate] 
    (0.5,0.5) -- node[left=4pt] {$\delta\ell$} (0.5,1);
\draw [dashed] (0.5, 3) -- (1,3);
\draw [dashed] (0.5, 1) -- (1,1);
\draw [dashed] (1, 3.5) -- (1,3);
\draw [dashed] (3, 3.5) -- (3,3);
\end{tikzpicture}
  \caption{The boxes involved in the 3-Stage procedure.}
  \label{fig:3-part setup}
\end{figure}
We shall run the torus shuffle for $2 \ell^2 n$ steps in Stages 1 and 3,
and for roughly $c^2 \ell^2 n$ steps in Stage 2. \\
%Note that the box lengths might not be integers,
%and we shall think of the boxes as squares in $\R^2$. 
%We shall consider a tile
%to be inside a box if it is entirely contained within the box. \\
~\\
{\bf Stage 1 Analysis:} \\
~\\
For the first stage, we wish to show that if $\ell$ is sufficiently large,
then with probability bounded away from $0$, after $2\ell^2n$ steps the tiles have moved to their respective boxes.
Note that when a tile moves it moves
  like a simple random walk in the torus. In fact, we can construct the torus shuffle using a
  collection of random walks as follows. 
  For every tile $\tilet$,
  let $\{m_q^\tilet\}_{q=1}^{\infty}$ be a sequence of random walk increments; that is, 
  the $m_q^\tilet$ are independent and identically distributed
uniformly over $\{(\pm 1, 0), (0, \pm 1)\}$. Then at each step, move the tiles in the grid in the following way:
\begin{enumerate}
    \item Flip a coin (to decide whether to be lazy). If the coin lands heads, then continue with steps 2 and 3 below. Otherwise, do nothing.
    \item Choose $n$ tiles such that no two share the same row or column. Call these  {\it master} tiles.
    \item Choose a master tile uniformly at random; call it $\widehat{\tilet}$.
      Suppose $\widehat{\tilet}$ has already moved $q-1$ times. Move $\widehat{\tilet}$ according to $m_q^{\widehat{\tilet}}$
and cyclically rotate the other tiles as necessary.  
\end{enumerate}
This is equivalent to the torus shuffle: the algorithm chooses each possible move of the torus shuffle with probability $\sfrac{1}{4n}$, independently
of the previous moves. Note that
a given tile can move even when it is not a master tile. 
We shall use the following method of selecting master tiles. 

\begin{enumerate}
\item Select $i$ as master tile;
    \item If $j$ is not in the same row or column as $i$, select $j$ as a master tile;
    \item If $k$ is not in the same row or column $i$ or $j$, select $k$ to be a master tile; 
    \item Choose remaining master tiles arbitrarily.
\end{enumerate}
This choice of master tiles ensures that
\begin{itemize}
\item Tile $i$ is always a master tile, so tile $i$ always follows its random walk sequence $\{m_q^i\}$.
\item Most of the time, tiles $j$ and $k$ are master tiles. Therefore, 
 most of the time, tiles $j$ and $k$ follow their random walk sequences; the only times
 tile $j$ does not follow its random walk sequence are times when it is in the same row or column as tile $i$
 when tile $i$ moves and the only times when tile $k$
 does not follow its random walk sequence are times when it in the same row or column as tile $i$ or tile $j$
 when one of them moves. 
\end{itemize}

When tiles $j$ or $k$ do not follow their
random walk sequences we call it \textit{interference}; Sometimes,  tile $i$
interferes with tile $j$ and sometimes tile  $i$ or tile  $j$ interferes with tile $k$.
Our aim will be to show that the number of steps with interference is small, and
hence the motion of the tiles is well-approximated by independent random
walks.
Let $X_t$ be the torus shuffle.
The idea will be to define processes $Y_t$ and $W$
that 
are successively more ``independent versions'' of the torus shuffle.
(More precisely, $Y_t$ and $W$ are processes that
model the motions of tiles $i$, $j$ and $k$ in the torus shuffle;
the other tiles are not relevant.)
Then we compare $X_t$ with $Y_t$ and then compare $Y_t$ with $W$.
In process $W$ the motions of the tiles are completely independent.
%(In fact, the processes $P_1$, $P_2$ and $P_3$
%only follow the motions of tiles $i$, $j$ and $k$; the other tiles are not relevant.)

Let $X_t({\tilet})$ be the position of tile ${\tilet}$ at time $t$ of the torus shuffle.
Note that the position of each
tile $\tilet$ is reduced componentwise $\bmod \; n$ since the tiles are in the torus. However, to simplify
our calculations,
we define $X_t({\tilet})$
to be a $\Z^2$-valued process that represents the
position of tile $\tilet$ not reduced componentwise $\bmod \; n$. 
For example if $\tilet$ is at position $(0,0)$ and then it moves to the left,
then its new position is $(-1,0)$.
Note that not reducing $\bmod \; n$ can only decrease the probability that particular tiles
go to particular spots within the grid. 
%Then define $P_{1, t}({\tilet}) = X_t({\tilet})$, for ${\tilet} = i, j, k$.
For $\tilet \in \{i, j, k\}$, let
$\nu_\tilet$ be the number of times tile $\tilet$ moves after $2\ell^2n$ steps of the torus shuffle.
%Let $\nu_i$ be the number of times $i$ moves after $2\ell^2n$ steps of the torus, with a similar definition for  $\nu_j$ and $\nu_k$.
Then 
\begin{align}
    X_{2\ell^2 n}(i) = X_0(i) + \sum_{s=1}^{\nu_i}m_s^i,
\end{align}
and since we expect minimal interference, we also expect that 
\begin{align}
    X_{2\ell^2 n}(j) &\approx X_0(j) + \sum_{s=1}^{\nu_j}m_s^j; \\
    X_{2\ell^2 n}(k) &\approx X_0(k) + \sum_{s=1}^{\nu_k}m_s^k.
\end{align}

To make this more precise, we introduce what we call the \textit{oblivious process}, which we shall denote by $Y_t$. 
In the oblivous process,
tiles $i$, $j$ and $k$ start in the same positions as in the torus shuffle and
move at the same times, but the paths of the tiles are independent
random walks. 

%We now describe the construction of the oblivous process $Y_t$ in more detail. 
More precisely, 
for $\tilet \in \{i, j, k\}$ we
write $Y_{t}({\tilet})$ for the position of tile $\tilet$ at time $t$. 
We let $Y_{ 0}({\tilet}) = X_{ 0}({\tilet})$ and for all times $t \geq 0$,
tile $\tilet$ moves in the $Y_t$ process if and only if it moves in the torus shuffle and it moves with the following rule:
\begin{itemize}
\item Suppose tile $\tilet$ has already moved $q-1$ times. Then tile $\tilet$ moves in the direction of $m^{\tilet}_q$,
  regardless of whether it is a master tile or not. 
\end{itemize}
Note that when $\tilet$ is a master tile, it moves in the same direction as in the torus shuffle. When it
is not a master tile, it moves in a direction that is independent of the move in the torus shuffle. 
Suppose that at time $t$, tile ${\tilet}$ has moved $r$ times. Then  
\begin{align}
    Y_{ t}({\tilet}) = X_0({\tilet}) + \sum_{s=1}^{r}m_s^{\tilet},
\end{align}
and hence
\begin{align}
Y_{ 2\ell^2 n}(i) &= X_0(i) + \sum_{s=1}^{\nu_i}m_s^i; \\
Y_{ 2\ell^2 n}(j) &= X_0(j) + \sum_{s=1}^{\nu_j}m_s^j; \\
Y_{ 2\ell^2 n}(k) &= X_0(k) + \sum_{s=1}^{\nu_k}m_s^k. \\
\end{align}
(Recall that $\nu_\tilet$ is the number of times tile $\tilet$ moves after $2\ell^2n$ steps of the torus shuffle.)

Note that the trajectories of the tiles in the $Y_t$ process are not quite independent, because whether a certain tile moves
in a given step is not independent of whether the others do.
(The random variables $\nu_i$, $\nu_j$ and $\nu_k$ each have the
Binomial($2 \ell^2 n, \frac{1}{2n}$) distribution, but they are not independent.) To address this slight dependence in the oblivious
process,
we introduce a third process $W$ that moves each tile exactly $\ell^2$ steps. Then we argue that
$W$ is likely a good approximation to the oblivious process $Y_t$, because each tile $\tilet$'s random walk is unlikely to
travel very far in the steps between $\nu_\tilet$ and $\ell^2$. Note that each $\nu_\tilet$
is approximately Poisson($\ell^2$). Hence $\nu_\tilet - \ell^2$ is likely to be on the order of $\ell$,
so the difference between tile $\tilet$'s position in $Y_t$ and $W$ is likely to be on the
order of $\sqrt{\ell}$, by Doob's maximal inequality. 

More precisely, 
for tiles ${\tilet} \in \{i,j,k\}$, let
\begin{align}
    W({\tilet}) := X_0({\tilet}) + \sum_{s=1}^{\ell^2} m_s^{\tilet} .
\end{align}

We shall call $W$ the {\it idealized  process}.
(But note that $W$ is not really a ``process,'' because there is no time dependence.
It is simply a 
triple of random variables that represents 
the result of three independent random walks, run for $\ell^2$ steps each.)

Recall that our aim is to show that with probability bounded away from zero,
for each ${\tilet} \in \{i, j, k\}$ we have that
$X_{ 2\ell^2 n}({\tilet})$ is in $\mathcal B_{\tilet}$. 
To accomplish this, we first show that with probability bounded away from $0$,
the idealized process $W$ puts each tile $\tilet$ inside
the smaller box $\mathcal L_{\tilet}$
within each target box. (The side lengths of the smaller boxes are $2 \delta l$ less than the target
boxes.) 
Then we argue that the error 
in  approximating the torus shuffle by $Y_t$, and
the error in approximating 
$Y_t$ by $W$, 
are both unlikely to exceed $\half \delta \ell$. (Here, by ``error'' we mean the maximum
$l^2$-distance between any tile's position in one process and its position in the other.)
By the triangle inequality, this will imply that with high probability each tile
lands in its target box. (See Figure \ref{fig:3-part stage 1} for an illustration.)

\begin{figure}[!h]
  \centering
\begin{tikzpicture}
\draw (0.5,0.5) rectangle (3.5,3.5);
\node at (3.8, 0.4){$\mathcal B_k$};
\draw (1,1) rectangle (3,3);
\node at (3.2, 0.9){$\mathcal L_k$};
\draw[fill=black] (2.5, 1.8) circle (0.05) node[left] {\scriptsize$W(k)$};
\draw[fill=black] (2.8, 2.1) circle (0.05) node[above] {\scriptsize$Y_{ 2\ell^2n}(k)$};
\draw[fill=black] (3.3, 1.7) circle (0.05) node[below right] {\scriptsize$X_{ 2\ell^2n}(k)$};
%\draw (2.5, 1.8)--(3.3, 1.7) node[below]{\scriptsize$\leq \delta\ell$};
\end{tikzpicture}
\caption{Tile $k$ ends up in ${\mathcal L_k}$ in the idealized process, and the total error in approximation is small;
  therefore, tile $k$ ends up in $\boxk$ in the torus shuffle.} 
  \label{fig:3-part stage 1}
\end{figure}

More precisely, we have
\begin{equation}
\begin{aligned}
    &\P(X_{ 2\ell^2 n}({\tilet}) \in B_{\tilet}, {\tilet}=i, j, k) \\
    &\geq \P \left (d_2(X_{ 2\ell^2 n}({\tilet}), \mathcal L_{\tilet}) \leq \delta\ell, {\tilet} = i, j, k\right)\\
    & \geq \P(W({\tilet}) \in \mathcal L_{\tilet}, {\tilet} = i, j, k) \\
    &\quad - \P\left
      (|| Y_{ 2\ell^2 n}({\tilet}) - W({\tilet}) ||_2 >
      \frac{1}{2}\delta\ell \text{ or }  ||X_{ 2\ell^2 n}({\tilet}) -
        Y_{ 2\ell^2 n}({\tilet})||_2 > \frac{1}{2}\delta\ell, \right.\\
    &\quad \quad \left. \vphantom{\frac{1}{2}}\text{ for at least one }{\tilet} \in \{i, j, k\}\right),
\end{aligned}
\end{equation}
where we write $d_2(X_{ 2\ell^2 n}({\tilet}), \mathcal L_{\tilet})$
to mean the $l^2$-distance between $X_{ 2\ell^2 n}({\tilet})$ and the box $\mathcal L_{\tilet}$.

By a union bound,

\begin{equation}
\begin{aligned}
    \P(\mbox{$X_{ 2\ell^2 n}({\tilet}) \in B_{\tilet}$ for ${\tilet} \in \{i, j, k\}$}) %\\
%    &\geq \bbP \left (d_{\infty}(X_{ 2\ell^2 n}({\tilet}), \mathcal L_{\tilet}) \leq \delta\ell, {\tilet} = i, j, k\right) \\
    & \geq \bbP(W({\tilet}) \in \mathcal L_{\tilet}, {\tilet} = i, j, k) \label{eq:Stage 1 terms}\\
    &\quad - \sum_{{\tilet} \in \{i, j, k\}}\bbP\left (\abs{Y_{ 2\ell^2 n}({\tilet}) - W({\tilet})}_2 > \frac{1}{2}\delta\ell \right)\\
    &\quad - \sum_{{\tilet} \in \{i, j, k\}}\bbP\left(\abs{X_{ 2\ell^2 n}({\tilet}) - Y_{ 2\ell^2 n}({\tilet})}_2 > \frac{1}{2}\delta\ell \right) .
\end{aligned}
\end{equation}

We start by lower bounding the first term
on the righthand side of
(\ref{eq:Stage 1 terms}),
which involves the idealized process. Since the tiles in the idealized process are independent, we can analyze each tile separately.
Consider tile $i$. It starts somewhere in the box $\boxl$ and then performs
$\ell^2$ steps of a random walk in the torus with increment steps $\{m_q^i\}_{q\geq 1}$.
Each $m_q^i$ is a random vector that is a
unit step in one of the 4 directions, chosen with probability $\frac{1}{4}$ each.

So we imagine that the increment steps $\{m_q^i\}_{q = 1}^{\ell^2}$ are unit steps
in $\Z^2$, chosen uniformly at random from the $4$ possible directions. 
The average step is $\langle 0,0\rangle$, the
covariance matrix of a step is $\begin{bmatrix} \frac{1}{2} & 0 \\ 0 & \frac{1}{2}\end{bmatrix}$, and the steps are independent and
identically distributed.

The displacement of tile $i$ is the sum of its $\ell^2$ random walk moves. %(That is, tile $i$ performs a 2-dimensional random walk.)
Let $S := \sum_{q=1}^{\ell^2} m_q^i$ be this sum.
By the ($2$-dimensional) central limit theorem, we know that $S$ should be approximately a normal random variable.
In particular, \cite{berry} gives a Berry-Esseen bound on how close this approximation is.
We state the $2$-dimensional version of the main theorem in \cite{berry} here:

\begin{theorem} \cite{berry}
  \label{berry}
  Let $Y_1, \dots, Y_N$ be independent random vectors in $\R^2$ with mean zero,
  and let $S = Y_1 + \cdots + Y_N$. Suppose that the covariance matrix $\Sigmam$, of $S$, is
  invertible. Let $Z$ be a centered normal random vector with covariance matrix  $\Sigmam$.
Then for any convex set $U \subset \R^2$, we have
\begin{align}
    \abs{\P(S \in U) - \P(Z \in U)} \leq \Cbe(2)^{\frac{1}{4}}\gamma,
    \label{eq:Berry-Esseen}
\end{align}
where $\Cbe$ is a universal constant and $\gamma = \sum_{q=1}^{N}\E\left[\abs{\Sigmam^{-1/2}Y_q}_2^3\right]$.
\end{theorem}

We apply the theorem with $N = \ell^2$ and $S = \sum_{q=1}^{\ell^2} m_q^i$.
Note that $S$ has covariance matrix $\Sigmam = \begin{bmatrix} \frac{\ell^2}{2} & 0 \\ 0 & \frac{\ell^2}{2}\end{bmatrix}$,
so $\Sigmam^{-1/2} = \begin{bmatrix} \frac{\sqrt{2}}{\ell} & 0 \\ 0 & \frac{\sqrt{2}}{\ell}\end{bmatrix}$ and hence
\begin{align}
    \abs{\Sigmam^{-1/2}m_q^i}_2 =\frac{\sqrt{2}}{\ell}.
\end{align}
Thus
\begin{align}
    \abs{\Sigmam^{-1/2}m_q^i}_2^3 =\frac{2\sqrt{2}}{\ell^{3}},
\end{align}
so
\begin{equation}
  \label{gammaex}
  \gamma = \frac{2\sqrt{2}}{\ell}.
\end{equation}

%The theorem applies to any convex set $U$.
For the convex set $U$, we use  $\li$, the small box of side length $c\ell-2\delta\ell$. 
Define $\li - X_0(i) := \{ l - X_0(i) : l \in \li \}$. Then
$\p(W(i) \in \li)$ is the probability that $S$ is in $\mathcal L_i - X_0(i)$.  
The Berry-Esseen result
allows us to approximate this by
the probability that the corresponding normal random variable $Z$ is in $\li - X_0(i)$. 
Since $Z$ is bivariate normal with a diagonal covariance matrix, the coordinates of $Z$ are independent. 
Hence, we can consider each coordinate of $Z$ separately. That is, 
if $Z = (Z_x, Z_y)$ and $\li - X_0(i) = I_x \times I_y$ for intervals $I_x$ and  $I_y$,
then 
$\P( W(i) \in \li)$ is approximately $\P(Z_x \in I_x) \P(Z_y \in I_y)$.
The interval $I_x$ has length $cl - 2 \delta l$
and is contained in the interval $[-l, l]$.
Since the density of a normal random variable decreases with distance from the mean, it follows that
\begin{equation}   \label{etadef}
  \begin{aligned}
    \P(Z_x \in I_x) 
    &\geq \P (Z_x \in ( \ell - (c\ell-2\delta\ell),\ell )) \\
    &= \P \Bigl(\bar{Z_x} \in \left (\sqrt{2}(1-c+2\delta), \sqrt{2}\right)\Bigr),
\end{aligned}
\end{equation}
where $\bar{Z_x} := \frac{\sqrt{2}}{\ell}Z_x$ is a standard normal random variable. 
Hence $\P(Z_x \in I_x) \geq \eta$ for some constant $\eta > 0$ that depends on $\delta$ and $c$.
(Recall that $2 \delta < c$.) 
A similar argument shows that $\P(Z_y \in I_y) \geq \eta$ as well. 
Hence
\begin{align}
    \P(Z \in \mathcal L_i - X_0(i)) \geq \eta^2.
\end{align}
Thus Theorem \ref{berry} and the expression for $\gamma$ in equation (\ref{gammaex}) 
imply that there is a universal constant $\widehat{C}$ such that 
\begin{align}
    \bbP(W(i) \in \mathcal L_i)
    =\bbP(S \in \mathcal L_i - X_0(i)) 
    \geq \eta^2 - \frac{\widehat{C}}{\ell}.
\end{align}

%This serves as a lower bound for one coordinate, and since the coordinates are independent, we square it to obtain a bound $\eta^2$ for both coordinates. Then, according to the Berry-Esseen result, we have that $\eta^2$ is on the order of $\frac{1}{\ell}$ away from the true probability bound. If we ensure that $\ell$ is sufficiently large, the probability that tile $i$ travels to its intended box after $\ell^2$ steps is bounded away from zero. As the prescribed processes are independent, this argument also applies to $j$ and $k$. If we let $\mathcal L_i$, $\mathcal L_j$, $\mathcal L_k$ denote the smaller boxes, then we have that
The same argument can be made for $j$ and $k$ as well, so we obtain similar bounds for them.
Since the random walks taken by each of the tiles are independent, we get
\begin{equation}
\begin{aligned}
    \P(W({\tilet}) \in \mathcal L_{\tilet}, {\tilet} = i, j, k) 
    &= \P(W(i) \in \mathcal L_i) \P(W(j) \in \mathcal L_j) \P(W(k) \in \mathcal L_k)\\
    &\geq \left( \eta^2 - \frac{\widehat{C}}{\ell}\right)^3, \label{eq:Stage 1 first term result}
\end{aligned}
\end{equation}
which is at least $\eta^6 \over 2$ when $\ell$ is sufficiently large.
%\begin{align}
%    \bbP(W \text{ processes go to smaller boxes}) 
%    &\geq \bbP(\widehat{I}^i \in \mathcal L_i, \widehat{I}^j \in \mathcal L_j, \widehat{I}^k \in \mathcal L_k) \\
%    &= \bbP(\widehat{I}^i \in \mathcal L_i) \cdot \bbP( \widehat{I}^j \in \mathcal L_j) \cdot \bbP( \widehat{I}^k \in \mathcal L_k) \\
%    &\geq \left (\eta^2 - \frac{\widehat{C}}{n} \right)^3
%\end{align}
%where $\widehat{C}$ is a constant obtained from the Berry-Esseen result.

%Change to tile k later
Now we consider the first sum appearing in the lower bound in (\ref{eq:Stage 1 terms}), which we recall is
\begin{align}
     \sum_{{\tilet} \in \{i, j, k\}}\P\left (\abs{Y_{ 2\ell^2 n}({\tilet}) - W({\tilet})}_2 > \frac{1}{2}\delta\ell\right).
\end{align}

We shall consider the term corresponding to  $\tilet = i$ first; the other terms are similar. Recall that $W(i)$ is the position of tile $i$ after $\ell^2$ random walk steps, while
$Y_{ 2\ell^2 n}(i)$
is the position of tile $i$ after 
$\nu_i$ steps, where $\nu_i$ is a binomial random variable with parameters $2\ell^2 n$ and  $\frac{1}{2n})$.
Since the two processes use the same random walk $\{m_q^i\}$,
any difference must come from performing a different number of random walk steps. We
shall argue that the difference in the number of steps is likely to be on the order of $\ell$,
which means (by Doob's maximal inequality) that the error is likely to be on the order of $\sqrt{\ell}$. 

Recall that $\eta$ is the constant defined after equation (\ref{etadef}). 
Let $\kappa$ be a constant large enough so that
\begin{equation}
  \label{kdef}
  {3 \over \kappa^2} \leq {\eta^6 \over 16},
  \end{equation}
and let $N_{\text{tol}} = \kappa \ell$.
($N_{\text{tol}}$ represents the maximum difference between $\nu_i$ and $\ell^2$ that
we are willing to ``tolerate''.)
Our approach will be in two parts: 
\begin{enumerate}
\item Bound $\P( |\nu_i - \ell^2 | > \ntol)$.
\item Bound the probability that
$|\nu_i - \ell^2| \leq \ntol$ and even so 
  the processes are still far apart; this part uses Doob's maximal inequality.
\end{enumerate}
For the first part we use Chebyshev's inequality: 
Since $\nu_i$ is a Binomial($2 \ell^2n, {1 \over 2n}$) random variable, its variance is
at most $\ell^2$, and hence
\begin{align}
  \label{cheby}
    \P( \abs{\nu_i - \ell^2} \geq N_{\text{tol}}) \leq \frac{\var(\nu_i)}{N_{\text{tol}}^2} \leq \frac{\ell^2}{\kappa^2 \ell^2} = \frac{1}{\kappa^2}.
\end{align}

For the second part, let $B$ be the event  
$\{\abs{\nu_i - \ell^2} < N_{\text{tol}}\}$. Note that
\begin{equation}
  \begin{aligned}
    \label{rwz}
    Y_{ 2\ell^2 n}(i) - W(i) = \sum_{s=1}^{\nu_i} m_s^i -\sum_{s=1}^{\ell^2} m_s^i \;.
%    &\leq \max_{\mu: \abs{\mu-\ell^2} < N_{\text{tol}}} \abs{\sum_{s=1}^{\mu} m_s^i -\sum_{s=1}^{\ell^2} m_s^i }_{\infty}\\
%    &\leq \max_{\mu: \abs{\mu-\ell^2} \leq N_{\text{tol}}} \abs{\sum_{s=1}^{\mu} m_s^i -\sum_{s=1}^{\ell^2} m_s^i }_{\infty}.
\end{aligned}
\end{equation}
For $N \geq 0$, define
\[
  S_N := \sum_{s=1}^N m_{\ell^2 + s}; \;\;\;\;\;\;\;\;\;
\shat_N := \sum_{s=1}^N m_{\ell^2 - s} \;\;.
\]
Note that $S_N$ and $\shat_N$ are both simple symmetric random walks on $\Z^2$, and by equation (\ref{rwz})
when $\abs{\nu_i - \ell^2} < N_{\text{tol}}$ the quantity
$Y_{ 2\ell^2 n}(i) - W(i)$ has  the value of one of these random walks on one of its first $\ntol$ time steps.
Since $\shat_N$ has the same transition rule as $S_N$, it follows that
\begin{equation}
  \label{samey}
  \P\Bigl( \mbox{$\Bigl\{\abs{Y_{ 2\ell^2 n}(i) - W(i)}_2 \geq \half \delta \ell \Bigr\} \cap B$} \Bigr) \leq
  2 \P \Bigl( \max_{1 \leq N \leq \ntol} \abs{S_N}_2 \geq  \half \delta \ell \Bigr) \;.
\end{equation}
Since $( \abs{S_N}_2)_{N \geq 0}$ is a submartingale, Doob's maximal inequality implies that 
that
\begin{equation}
  \label{maxbound}
  \P \Bigl( \max_{1 \leq N \leq \ntol} \abs{S_N}_2 \geq {1 \over 2} \delta \ell \Bigr) \leq
  {2\e( \abs{ S_\ntol}_2) \over \delta \ell} \;.
\end{equation}

It is well-known that simple symmetric random walk $S_N$ in $\Z^2$, started at the origin,  satisfies $\e \abs{S_N}_2^2 = N$,
and hence $\e \abs{S_N}_2 \leq \sqrt{N}$ by Jensen's inequality. Thus, equations
(\ref{samey}) and (\ref{maxbound}) imply that
\begin{eqnarray*}
  \P\Bigl( \Bigl\{\abs{Y_{ 2\ell^2 n}(i) - W(i)}_2 \geq \half \delta \ell \Bigr\} \cap B \Bigr)
  &\leq& {4 \sqrt{\ntol} \over \delta \ell} \\
  &=& {4 \sqrt{\kappa} \over \delta \sqrt{\ell}}
\end{eqnarray*}

Combining this with equation (\ref{cheby}) gives
\begin{eqnarray*}
  \P\Bigl( \abs{Y_{ 2\ell^2 n}(i) - W(i)}_2 \geq \half \delta \ell\Bigr)   &\leq&
   \P(B^c) +   \P\Bigl( \mbox{$\abs{Y_{ 2\ell^2 n}(i) - W(i)}_2 \geq \half \delta \ell$ and $B$ occurs} \Bigr) \\
                                                                                    &\leq&
\frac{1}{\kappa^2} + {4 \sqrt{\kappa} \over \delta \sqrt{\ell}} \;.
\end{eqnarray*}

This bound applies to tiles $j$ and $k$ as well, so

\begin{align}
  \sum_{{\tilet} \in \{i, j, k\}}\P\left (\abs{Y_{ 2\ell^2 n}({\tilet}) - W({\tilet})}_2 > \frac{1}{2}\delta\ell\right)
  \leq 3 \Bigl( \frac{1}{\kappa^2} + {4 \sqrt{\kappa} \over \delta \sqrt{\ell}} \Bigr) \;.
\end{align}
By equation (\ref{kdef}), for sufficiently large $\ell$ this is at most
${\eta^6 \over 8}$. 

Finally, we consider the last term of the lower bound in (\ref{eq:Stage 1 terms}). We shall need the following lemma.
\begin{lemma}
  \label{L}
  Let $s \in \Z^2$. Let $\Delta = (\Delta_1, \Delta_2)$ be a random variable taking values in
  $\Z^2$ and suppose $\e \Delta_i = 0$. Then
  \[
    \e \Bigl( \Abs{s + \Delta}_2^2 \Bigr)
    = \Abs{s}_2^2 + \e( \Delta_1^2) + \e( \Delta_2^2)   \;.
    \]
\end{lemma}
\begin{proof}
We have
  \begin{eqnarray}
    \Abs{ s + \Delta }_2^2 &=& (s_1 + \Delta_1)^2 + (s_2 + \Delta_2)^2.     \label{star}
  \end{eqnarray}
But since $\e(\Delta_1) = 0$ we have
  \begin{eqnarray*}
    \e(  s_1 + \Delta_1)^2 &=& \var( s_1 + \Delta_1) + ( \e(s_1 + \Delta_1))^2 \\
                          &=& \var(\Delta_1) + s_1^2 \\
                           &=& \e( \Delta_1^2) + s_1^2 \,,  
  \end{eqnarray*}
  and similarly for $\e(  s_2 + \Delta_2)^2$. Thus, taking expectations in (\ref{star}) gives
  \begin{equation}
    \e \Abs{ s + \Delta }_2^2 = s_1^2 + s_2^2 + \e \Delta_1^2 + \e \Delta_2^2,
  \end{equation}
  and note that $s_1^2 + s_2^2 = \Abs{s}_2^2$.
\end{proof}
We now return to bounding the second sum the lower bound in (\ref{eq:Stage 1 terms}), which we recall is
\begin{align}
  \label{summ}
    \sum_{{\bbT}=i, j, k}\bbP\left(\Abs{X_{ 2\ell^2 n}({\bbT}) - Y_{ 2\ell^2 n}({\bbT})}_2 > \frac{1}{2}\delta\ell\right).
\end{align}

We shall focus on the term in the sum corresponding to tile $k$.
Recall that $X_t(k)$ (respectively, $Y_t(k)$) is the position of tile $k$ after $t$ steps of the torus shuffle
(respectively, oblivious shuffle). 
Let $Z(t) = X_t(k)- Y_{t}(k)$, the vector difference between the two processes,
and suppose that $Z(t) = (Z_1(t), Z_2(t))$. 

Let $N_t^k$ be the number of times in the first $t$ steps of the torus shuffle that tile $k$ moves without being a master tile.
Let $\f_t$ be the $\sigma$-field generated by $(X_0, \dots, X_t)$ and $(Y_0, \dots, Y_t)$.
\begin{claim}
  $\Abs{ Z(t) }_2^2 - 2 N_t^k$ is a martingale with respect to $\f_t$.
  \end{claim}
To verify the claim, consider the step at time $t$, that is, the step that generates $X_{t+1}$ and $Y_{t+1}$ from
$X_{t}$ and $Y_{t}$. Note that $N_{t+1}^k - N_{t}^k$ is either $0$ or $1$.
We consider the two cases separately. \\
~\\
{\bf Case $N_{t+1}^k - N_{t}^k = 0$}:
In this case, either
\begin{itemize}
\item Tile $k$ does not move in either process, or
\item Tile $k$ is a master tile and hence moves the same way in both processes.
\end{itemize}
Either way, $Z(t+1) = Z(t)$ and hence
$\Abs{ Z(t+1) }_2^2 - 2 N_{t+1}^k = \Abs{ Z(t) }_2^2 - 2 N_t^k$. \\
~\\
{\bf Case $N_{t+1}^k - N_{t}^k = 1$}:
In this case in the step at time $t$, tile $k$ moves without being
a master tile. First, we consider the case where tile $k$ moves in a horizontal direction in the torus shuffle 
without being a master tile. 
(The analysis when tile $k$ moves vertically is similar.)
Call this event $H$. Note that on the event $H$, the direction that tile $k$ moves in the oblivious process
is independent
of whether tile $k$ moves right or left in the torus shuffle.
It follows that $(X_{t+1}(k), Y_{t+1}(k))$ is one of the following:
\begin{itemize}
\item $\Bigl( X_t  + (1,0), Y_t + (1,0) \Bigr)$
\item $\Bigl( X_t  +(1,0), Y_t  - (1,0) \Bigr)$
\item $\Bigl( X_t  - (1,0), Y_t + (1,0) \Bigr)$
\item $\Bigl( X_t - (1,0), Y_t  -(1,0) \Bigr)$
\item $\Bigl( X_t + (1,0), Y_t  + (0,1) \Bigr)$
\item $\Bigl( X_t + (1,0), Y_t - (0,1) \Bigr)$
\item $\Bigl( X_t - (1,0), Y_t + (0,1) \Bigr)$
\item $\Bigl( X_t - (1,0), Y_t - (0,1) \Bigr)$
\end{itemize}
with all $8$ possibilities equally likely. Thus 
\[
Z(t+1) - Z(t) = 
\left\{\begin{array}{lllllll}
(0,0) & 
\mbox{with probability $1/4$;}\\
(2,0)  & \mbox{with probability $1/8$;}\\
(- 2,0)  & \mbox{with probability $1/8$;}\\
(1,1)  & \mbox{with probability $1/8$;}\\
(1,-1)  & \mbox{with probability $1/8$;}\\
(-1,1)  & \mbox{with probability $1/8$;}\\
(-1,-1)  & \mbox{with probability $1/8$.}\\
\end{array}
\right.
\]
Hence, by Lemma \ref{L}, it follows that
\begin{equation}
  \e( \Abs{Z(t+1)}_2^2 \given \f_t, H) = \Abs{Z(t)}_2^2 + 2 \,.
  \end{equation}
  The calculation is similar in the event (which we denote by $V$) that tile $k$ makes a vertical move in the torus shuffle
  without being a master tile. Hence
\begin{equation}
  \e( \Abs{Z(t+1)}_2^2 \given \f_t, H \cup V) = \Abs{Z(t)}_2^2 + 2 \,,
  \end{equation}
  and so
\begin{equation}
  \e( \Abs{Z(t+1)}_2^2 - 2 N_{t+1}^k \given \f_t) = \Abs{Z(t)}_2^2 - 2 N_{t}^k \,.
\end{equation}

Hence $\Abs{Z(t)}_2^2 - 2 N_{t}^k$ is indeed a martingale. Since it initially has the value $0$,
the martingale property implies that
for all times $t$ we have
\begin{equation}
  \label{zbound}
  \e \bigl( \Abs{Z(t)}_2^2 \bigr) =  2 \e( N_{t}^k ).
\end{equation}

Let $A_0^k = 0$ and for $t \geq 1$ let
$A_t^k$ be the number of times in $\{0, 1, \dots, t-1\}$
that
tile $k$
shares a row or column with tile $j$ or tile $i$.
(Recall that $k$ is a master tile unless it is in the same row or column
as tile $j$ or tile $i$.)
Since tile $k$ has a $\frac{1}{2n}$ chance of moving at each step,
independently of the locations of the tiles,
it follows that for $t \geq 1$ we have
\begin{align}
  \label{nbound}
    \bbE(N_t^k) \leq \frac{1}{2n} \bbE(A_t^k).
\end{align}
%If $k$ is not a master tile
%at step $t$, then 
%either $k$ shares a row or column with tile $i$
%at time $t-1$
%or $k$ shares a row or column with tile $j$ at time $t-1$ (or both).
Let $A_t^{k, i}$ be the number of times in $\{0, \dots, t-1\}$ that
tile $k$ is in the same row or column as tile $i$, 
with a similar definition for
$A_t^{k, j}$.
Then for $t \geq 0$ we have
\begin{align}
  \label{as}
    A_t^k \leq A_t^{k, i} + A_t^{k, j}  \,.
\end{align}

We now wish to find uppper bounds for  
$\e (A_t^{k, i})$ and $\e (A_t^{k, j})$ when $t = 2 \ell^2 n$.
First, we bound $\e (A_t^{k, i})$; the upper bound for $\e (A_t^{k, j})$ is similar.
We will only need the bound for when tiles $k$ and $i$ start in the box $B_l$,
but we shall prove a bound that is valid for any starting points.

Fix $t_0 \geq 0$. We define a process $D(t) = ((D_1(t), D_2(t))_{t \geq t_0}$,
which we call the {\it difference process,} 
that represents the difference between the positions of tile $i$ and tile $k$
in
the torus shuffle, starting from time $t_0$. 
We define $D(t)$
inductively as follows.
Let $D(t_0)$ be such that
\begin{eqnarray*}
  D_1(t_0) &\equiv& \Bigl( X_{t_0}(i) - X_{t_0}(k) \Bigr)_1 \;\bmod \, n;   \\
  D_2(t_0) &\equiv& \Bigl( X_{t_0}(i) - X_{t_0}(k) \Bigr)_2 \;\bmod \, n,
\end{eqnarray*}
and also $|D_1(t_0)| \leq {n \over 2}$ and 
$|D_2(t_0)| \leq {n \over 2}$. For $t \geq t_0 + 1$, 
let $D(t)$ be such that
\begin{eqnarray*}
  D_1(t) &\equiv& \Bigl( X_t(i) - X_t(k) \Bigr)_1 \;\bmod \, n;   \\
  D_2(t) &\equiv& \Bigl( X_t(i) - X_t(k) \Bigr)_2 \;\bmod \, n,
\end{eqnarray*}
and also $\Abs{ D(t) - D(t-1)}_1 \leq 1$.
Note that
\[
  D(t) - D(t-1) \in \{ (0,0), (\pm 1, 0), (0, \pm 1) \} \;.
\]
So $D(t)$ is a sort of random walk on $\Z^2$ that starts somewhat close to $(0,0)$
at time $t_0$.
(It's a random walk that can only change value if tile $i$ or tile $k$ moves.)
%We shall call $D(t)$ the {\it difference random walk.}
Define the {\it wrap around time} $\tesc$ by 
\[
  \tesc := \min\{ t > t_0: \mbox{$D_1(t) \in \{-n,n\}$ or
    $D_2(t) \in \{-n,n\}$} \} \,.
\]
We can think of $\tesc$ as the time when 
$i$ and $k$ are in the same row or column because one of them ``wrapped around'' the torus.
(The possibility of wrapping around
within $2n \ell^2$ steps
is a %n irritating
technical issue
that we have to consider when $\ell$ is roughly of the same order as $n$.)

Consider first the difference process $D(t)$ starting at time $t_0 = 0$.
Note that if
$D_1(t) \neq 0$,  
$D_2(t) \neq 0$ and 
$t < \tesc$, then 
tile $i$ and tile $k$ are
not in the same row or column at time $t$. 
~\\
Let $M_t := \Abs{ D(t) }_1 - {1 \over 2n} A_t^{k,i}$.
We claim that
$M_{t \wedge \tesc}$
is
a martingale with respect to $\sigma(X_1, \dots, X_t)$. To see this,
consider the following two cases:\\
~\\
{\bf Case $D_1(t) \neq 0$ and $D_2(t) \neq 0$:} 
Suppose that $0 < |D_1(t)| < n$ and 
$0 < |D_2(t)| < n$. Note that in this case, if a move increases the value of
$\Abs{ D(t) }_1$ by one unit, then the opposite
move (which occurs with the same probability)
decreases 
$\Abs{ D(t) }_1$ by one unit. 
Also $A_{t+1}^{k,i} = A_{t}^{k,i}$. It follows that in this case we have
$\e (M_{t+1} \given X_1, \dots, X_t) = M_t$. \\
~\\
{\bf Case $D_1(t) = 0$ or $D_2(t) =  0$:} In this case tiles $i$ and $k$ are
in the same row or column and hence
$A_{t+1}^{k,i} = A_{t}^{k,i} + 1$. There are $4$ moves that
change the value of 
$\Abs{ D(t) }_1$. For example, if $i$ and $k$ are in the same row, then $i$ could move up or down
or $k$ could move up or down. Each move
increases $\Abs{ D(t) }_1$ by one unit and 
occurs with probabilty ${1 \over 8n}$.
Hence $\e (\Abs{D(t+1)}_1 \given X_1, \dots, X_t) = \Abs{D(t)}_1 + {1 \over 2n}$ and hence
$\e (M_{t+1} \given X_1, \dots, X_t) = M_t$ in this case as well. \\
~\\
It follows that $M_{t \wedge \tesc}$ is a martingale. Consequently,
\begin{eqnarray*}
 \e \Bigl( \Abs{ D(t \wedge \tesc)}_1 - {1 \over 2n} A_{t \wedge \tesc}^{k,i}  \Bigr)  &=& \e (M_{t \wedge \tesc}) \\
                                                                                       &=&  \e M_0   \\
  &=& \Abs{ D(0)}_1      \\
\end{eqnarray*}
and hence
\begin{eqnarray}
  \label{mart}
  \e \Bigl( {1 \over 2n} A_{t \wedge \tesc}^{k,i} \Bigr) &=& \e \Abs{ D(t \wedge \tesc)}_1 - \Abs{ D(0)}_1 \nonumber \\
  &\leq& \sqrt{ \e D_1(t \wedge \tesc)^2}  + \sqrt{ \e | D_2(t \wedge \tesc)^2}  - \Abs{ D(0)}_1,   \label{sixtyeight}
\end{eqnarray}
by Jensen's inequality.
Since for any real number $u$ we have $|u| = \sqrt{u^2}$,
the quantity (\ref{sixtyeight}) can be written as
\begin{equation}
 \Bigl(\sqrt{ \e D_1(t \wedge \tesc)^2}
  - \sqrt{ D_1(0)^2}  \Bigr)
+ \Bigl( \sqrt{ \e | D_2(t \wedge \tesc)^2}
  - \sqrt{ D_2(0)^2}     \Bigr) \label{twoterms}    \,.
\end{equation}
We shall now bound the first  term in parentheses in (\ref{twoterms}). 
Let $V_1(t) := D_1(t)^2 - {1 \over 2n} t$.      %, with a similar definition for $V_2(t)$.
We claim that $V_1(t)$
%and $V_2( t \wedge \tesc)$
is a  supermartingale.
%First, we verify that $V_1( t \wedge \tesc)$ is a supermartingale.  
Note that if $i$ and $k$ are in the same row at time $t$ then $D_1(t)$ cannot change in
the next step. If $i$ and $k$ are not in the same row at time $t$, then
\begin{itemize}
  \item there are two moves that decrease $D_1(t)$ by one unit;
  \item there are two moves that increase $D_1(t)$ by one unit,
  \end{itemize}
and each move that changes $D_1(t)$ occurs with probability ${1 \over 8n}$. 
Thus, 
\begin{eqnarray}
  \e ( D_1(t+1)^2 \given X_1, \dots, X_t) &=&  \Bigl( 1 - {1 \over 2n} \Bigr)   \nonumber
                                              D_1^2(t) + {1 \over 4n} (D_1(t) + 1)^2 +  {1 \over 4n} (D_1(t) - 1)^2   \\
                                          &=& D_1(t)^2 + {1 \over 2n} \,,        \label{submg}
\end{eqnarray}
and hence $V_1(t)$ is a supermartingale.
Since $\tesc$ is a stopping time, it follows that $V_1(t \wedge \tesc)$ is a supermartingale as well. 
%A similar argument shows that
%the process 
%$V_2( t \wedge \tesc)$ is a supermartingale as well.
%Since $V_1( t \wedge \tesc)$ is a supermartingale,
Hence, for any time $t$ we have
\begin{eqnarray*}
  \e (V_1(t \wedge \tesc))  ) &\leq& \e V_1(0) \\
                              &=& D_1(0)^2
\end{eqnarray*}                 
and hence
\begin{eqnarray*}
  \e ( D_1( t \wedge \tesc )^2 ) &\leq& D_1(0)^2 + {1 \over 2n} \e( t \wedge \tesc)    \\
  &\leq& D_1(0)^2 + {t \over 2n} \,,
\end{eqnarray*}
so the first term in parentheses in (\ref{twoterms}) is at most
\begin{eqnarray*}
  \sqrt{ D_1(0)^2 + {t \over 2n}} - \sqrt{ D_1(0)^2} &\leq& \sqrt{ {t \over 2n}},
\end{eqnarray*}
where the inequality holds
by the Net Change Theorem from Calculus,
because the second derivative of the function $\sqrt{x}$ is negative and hence the function
\[
  \sqrt{ x^2 + {t \over 2n}} - \sqrt{ x^2}
\]
is maximized when $x=0$. 
A similar argument shows that the second term in 
parentheses in (\ref{twoterms}) also at most
$\sqrt{ {t \over 2n}}$. It follows that
\begin{eqnarray}
  \label{marty}
  \e \Bigl( {1 \over 2n} A_{t \wedge \tesc}^{k,i} \Bigr) &\leq&
                                                                2 \sqrt{t \over 2n} \,.
\end{eqnarray}
Substituting $t = 2 \ell^2 n$ gives
\begin{equation}
  \label{escapebound}
  \e \Bigl( {1 \over 2n} A_{2 \ell^2 n \wedge \tesc}^{k,i} \Bigr) \leq
                                                                2\ell \,.
\end{equation}
We can define a sequence $\tesc^0 < \tesc^1 < \tesc^2 < \cdots$ of wrap around times inductively as follows.
Define $\tesc^0 := 0$
and for $r \geq 1$ let $\tesc^{r}$ be the wrap around time
starting from time $\tesc^{r-1}$.
%(Note that $\tesc = \tesc^1$.)
Let $N = \min\{r : \etesc{r} \geq 2 \ell^2 n\}$ be the number of times tiles $i$ and $k$ wrap around
before time $2 \ell^2 n$, plus one. The strong Markov property and an
argument similar to the argument leading to equation (\ref{escapebound}) shows that for all $r \geq 0$ we have
\[
  \e \Bigl( 
  {1 \over 2n} \Bigl(
  A_{2 \ell^2 n \wedge \etesc{r+1}}^{k,i} - A_{2 \ell^2 n \wedge \etesc{r}}^{k,i} \Bigr)
  \given N > r
\Bigr) \leq
  2\ell \,.
\]
It follows that
\begin{eqnarray}
  \e \Bigl( {1 \over 2n} A_{2 \ell^2 n}^{k,i} \Bigr) &=&
                                                         \sum_{r = 0}^\infty
                                                           \e \Bigl( 
  {1 \over 2n}\Bigl( A_{2 \ell^2 n \wedge \etesc{r+1}}^{k,i} - A_{2 \ell^2 n \wedge \etesc{r}}^{k,i}  \Bigr)       \one(N > r)
                                                         \Bigr)    \nonumber \\
                                                     &\leq& \sum_{r = 0}^\infty 2 \ell \P(N > r)    \nonumber \\
                                                     &=& 2 \ell \e N  \,.   \label{almost}
\end{eqnarray}

Our final step will be to show that $\e N$ is bounded above.
To accomplish this, we first show that $i$ and $k$ are unlikely to ``wrap around'' in fewer than
${n^3 \over 4}$ steps. To see this, note that equation (\ref{submg}) implies that
$D_1(t)^2$ is a submartingale. Thus, Doob's maximal inequality implies that
\begin{eqnarray}
\label{bd}
  \P( \max_{0 \leq t \leq {n^3 \over 4}} |D_1(t)| \geq n) &=&
                                                              \P( \max_{0 \leq t \leq {n^3 \over 4}} D_1(t)^2 \geq n^2) \\
                                                          &\leq& { \e \Bigl( D_1({n^3 \over 4})^2 \Bigr)
                                                                 \over n^2} \,.
\end{eqnarray}
But since $D_1(t)^2 - {t \over 2n}$ is a supermartingale, it follows that
  for any $t$ we have $\e D_1(t)^2 \leq D_1(0)^2 + {t \over 2n}$. Substituting $t = {n^3 \over 4}$
  and using the fact that $D_1(0) \leq {n \over 2}$ gives 
\begin{eqnarray*}
  \e D_1\Bigl( {n^3 \over 4} \Bigr)^2 &\leq& {n^2 \over 4} + {n^2 \over 8}\\
  &=& {3n^2 \over 8} \,.
\end{eqnarray*}
Combining this with equation (\ref{bd}) gives
\begin{equation}
  \label{bod}
  \P( \max_{0 \leq t \leq {n^3 \over 4}} |D_1(t)| \geq n) \leq {3 \over 8} \,.
  \end{equation}
A similar argument shows that
\begin{equation}
  \label{bodt}
  \P( \max_{0 \leq t \leq {n^3 \over 4}} |D_2(t)| \geq n) \leq {3 \over 8} \,,
  \end{equation}
  and hence $\P( \tesc \leq {n^3 \over 4} ) \leq {3 \over 4}$, which implies that
  $\P( \tesc > {n^3 \over 4} ) \geq {1 \over 4}$.
  By the strong Markov property, this means that for all $r \geq 0$ we have
  \[
    \P( \etesc{r+1} - \etesc{r} > {n^3 \over 4} \given \etesc{0}, \dots, \etesc{r} ) \geq {1 \over 4}\;.
  \]
Since
  $\ell \leq n$, we have
  \[
    {2n \ell^2 \over n^3/4} \leq 8 \,,
  \]
and hence $[0, 2n \ell^2]$ cannot be partitioned into more than $8$ subintervals 
of length greater than ${n^3 \over 4}$. 
It follows that the random variable $N$ is stochastically dominated by the number of flips required
  to get $8$ heads
when flipping a coin of bias ${1 \over 4}$
  (that is, the negative binomial distribution with parameters $8$ and ${1 \over 4}$). Hence $\e N \leq 32$. Combining this
  with equation (\ref{almost}) gives
\begin{equation}
  \e \Bigl( {1 \over 2n} A_{2 \ell^2 n}^{k,i} \Bigr) \leq 64 \ell \;.
\end{equation}
A similar argument shows that we also have
\begin{equation}
  \e \Bigl( {1 \over 2n} A_{2 \ell^2 n}^{k,j} \Bigr) \leq 64 \ell \;. 
\end{equation}
Combining this with equation (\ref{as}) gives
\begin{align}
  \e \Bigl( {1 \over 2n} A_{2 \ell^2 n}^k \Bigr)      \leq 128 \ell \,.
\end{align}
Combining this with (\ref{zbound}) and (\ref{nbound}) gives
\begin{eqnarray*}
  \e( \abs{ Z( 2 \ell^2 n) }_2  )   &\leq& \sqrt{256 \ell } \\
                           &=& 16 \sqrt{\ell} \,.
\end{eqnarray*}
Thus, Markov's inequality implies that the term in the sum in (\ref{summ}) corresponding to $\tilet = k$, 
\begin{eqnarray*}
  \P( \abs{ Z(2 \ell^2 n)}_2 \geq {1 \over 2} \delta \ell) &\leq& {16 \sqrt{\ell} \over \delta \ell/2}    \\
                                                           &=& {32 \over \delta \sqrt{\ell}} \;.
\end{eqnarray*}
This is at most ${\eta^6 \over 16}$ when $\ell$ is sufficiently large.
(Recall that $\eta$ is the constant defined after equation (\ref{etadef}).)
We can get a similar bound for the term in the sum in (\ref{summ}) corresponding to $\tilet = j$
(although, we could do a little better because tile $j$ has less interference). Finally, note that
the term in the sum in (\ref{summ}) corresponding to $\tilet = i$ is  $0$ because
tile $i$ behaves the same way in the oblivious process as in the torus shuffle. It follows that
if $\ell$ is sufficiently large then the second sum
in (\ref{eq:Stage 1 terms}) is at most ${\eta^6 \over 8}$.

In summary, we have shown that when $\ell$ is sufficiently large,
the first term in (\ref{eq:Stage 1 terms}) is at least ${\eta^6 \over 2}$,
and each sum in (\ref{eq:Stage 1 terms}) is at most ${\eta^6 \over 8}$. It follows that
\begin{equation}
  \label{sone}
\begin{aligned}
    \P(\mbox{$X_{ 2\ell^2 n}({\tilet}) \in B_{\tilet}$ for ${\tilet} \in \{i, j, k\}$} ) %\\
    & \geq {\eta^6 \over 4}     \,.
\end{aligned}
\end{equation}
~\\
{\bf Stage 2 Analysis:} \\
~\\
Recall that $c \ell$ is the side length of boxes $\boxi$, $\boxj$ and $\boxk$. 
Let $\ellp$ be an even number
such that
\begin{equation}
  \label{ellpdef}
  c \ell \leq \ellp \leq 2c \ell  \;.
\end{equation}
Such a number exists if $\ell$ is sufficiently large. 
In Stage 2, we
will run the torus shuffle ${n  \ellp^2}$ steps.
%for some constant $\beta$.
We  
need a $\bigomega( \ell^{-6})$ lower bound 
for the probability that tiles $i$, $j$ and $k$,
all starting in their respective boxes, move to specified
target positions in their respective boxes after the $n \ellp^2$
steps. Our main tool will be the local central limit theorem.  \\
~\\
As in our analysis of Stage 1, we will construct the
torus shuffle using  master tiles.
We will define random walks associated to 
tiles $i$, $j$ and $k$, respectively,
and each tile will
move according to its random walk when it is a  master tile.
This time, tiles $i$, $j$ and $k$ will be ``far away'' from each other
so we can ensure that they remain master tiles, with high probability,
throughout the duration of Stage 2.

To avoid parity issues in our application of the local limit
theorem it will be helpful if the random walks driving the shuffle are
lazy. So we modify our construction from Stage 1 as follows.
Instead of flipping a coin initially to determine whether each step is lazy,
this time there will always be a tile that does
a step of its random walk, only this time the tiles' random walks themselves will be lazy. 
In addition, to simplify calculations, for each tile
we consider the horizontal walk and the vertical walk separately.

More precisely,
for each tile ${\bbT}$ %\textit{two}
we assign two sequences of
one-dimensional lazy random walk increments, a horizontal walk $\{h_q^{\bbT}\}_{q=1}^{\infty}$ and a vertical
walk $\{v_q^{\bbT}\}_{q=1}^{\infty}$.
The random walk increments are independent and
have the distribution corresponding to lazy simple random walk on $\Z$, that is, 
$1$ or $-1$ with probability $\quarter$ each, and $0$ with probabilty $\half$.
%\begin{table}[H]
%\begin{tabular}{r|c}
%$m$ & $p_m$ \\
%\hline
%$-1$ & $1/4$ ; \\
%$0$ & $1/2$; \\
%$1$ & $1/4$ .
%\end{tabular}
%\caption{Results of the weighted spring experiment}\label{table:spring}
%\end{table}
%$-1$ or $1$ with probability $\frac{1}{4}$ each, and $0$ with probability $\frac{1}{2}$.
%(That is, ``lazy'' simple random walk on $\Z$.)
We shall consider an increment taken from the sequence $\{h_q^\tilet\}$ to be a ``horizontal move,''
even when $h_q^{\tilet} = 0$, and similarly for the vertical walk. 

To generate the torus shuffle, at each step:
\begin{enumerate}
\item Choose $n$ tiles that don't share any rows 
  or columns to be master tiles. %, including $i$, $j$ and $k$, if possible.
  If $i$, $j$ and $k$ don't share any rows or columns then
  include them among the master tiles. 
    \item Choose a  master tile uniformly at random; call it $\widehat{{\bbT}}$.
    \item Flip a coin. (Heads will indicate a horizontal move and tails will indicate a vertical move.) 
      %Randomly decide whether to try to move $\widehat{{\bbT}}$ horizontally or vertically.
    \begin{itemize}
    \item If the coin lands heads:
      suppose $\widehat{{\bbT}}$ has already made $(q-1)$
      horizontal moves. Then  move  $\widehat{{\bbT}}$ according to $h_q^{\widehat{{\bbT}}}$, and cyclically rotate 
      the other tiles accordingly.
 \item If the coin lands tails:
      suppose $\widehat{{\bbT}}$ has already made $(q-1)$
      vertical moves. Then  move  $\widehat{{\bbT}}$ according to $v_q^{\widehat{{\bbT}}}$, and cyclically rotate 
      the other tiles accordingly.
          \end{itemize}
        \end{enumerate}
        Let $r := \ellp/2$.  Then equation 
(\ref{ellpdef}) implies that
\begin{equation}
  \label{rbounds}
  {cl \over 2} \leq  r \leq cl \;,
\end{equation}
and hence
\begin{eqnarray}
  Kr &\leq& Kcl     \nonumber  \\
     &<& {l \over 3}   \label{smallbox}  \;.
\end{eqnarray}        
(Recall that the constant $c$ is such that $cK < {1 \over 3}$.)
Assume that tiles $i$, $j$ and $k$ all start in their respective boxes. 
We wish to use the local central limit theorem to find a lower bound for the probability
that the tiles go to their target positions in their boxes during Stage 2.
We want to ensure that tiles $i$, $j$ and $k$ are always in different rows and columns throughout
Stage 2, so that they remain master tiles. We accomplish this by defining ``buffer zones''
$\effi, \effj$ and $\effk$ around the boxes $\boxi$, $\boxj$ and $\boxk$, respectively.
(See Figure \ref{fig:Stage 2 diagram} below.) The buffer zones
will be squares whose side lengths are $Kr$. 
Since $Kr < {l \over 3}$ (see equation (\ref{smallbox})), tiles in different
buffer zones don't share any rows or columns with each other. 
Say that a tile goes to its target position {\it safely}
if it does so without ever leaving its buffer zone.
We will actually find a lower bound for the probability that
each tile's random walk takes it to its target position safely. 

\begin{figure}[!h]
  \centering
\begin{tikzpicture}
%left box
%big box
\draw (0,0) rectangle (4,4);
%small boxes
\draw (4/9, 28/9) rectangle (8/9, 32/9);
\draw (16/9, 16/9) rectangle (20/9, 20/9);
\draw (28/9, 4/9) rectangle (32/9, 8/9);
\node at (9.5/9, 33/9){\small$\mathcal B_i$};
\node at (21.5/9, 21/9){\small$\mathcal B_j$};
\node at (33.5/9, 9/9){\small$\mathcal B_k$};
\node at (2/9, 23/9) {$\mathcal F_i$};
\node at (14/9, 11/9) {$\mathcal F_j$};
\node at (23/9, 2/9) {$\mathcal F_k$};
%buffer zones
\draw [dashed] (1/9, 25/9) rectangle (11/9, 35/9);
\draw [dashed] (13/9, 13/9) rectangle (23/9, 23/9);
\draw [dashed] (25/9, 1/9) rectangle (35/9, 11/9);
%draw braces for length measurements
%horizontal braces
\draw[decoration={brace,raise=2pt},decorate] 
    (1/9,4) -- node[above=4pt] {$Kr$} (11/9,4);
\draw[decoration={brace,raise=2pt},decorate] 
    (13/9,4) -- node[above=4pt] {$Kr$} (23/9,4);
\draw[decoration={brace,raise=2pt},decorate] 
    (25/9,4) -- node[above=4pt] {$Kr$} (35/9,4);
%vertical braces
\draw[decoration={brace,raise=2pt},decorate] 
    (0,1/9) -- node[left=4pt] {$Kr$} (0,11/9);
\draw[decoration={brace,raise=2pt},decorate] 
    (0,13/9) -- node[left=4pt] {$Kr$} (0,23/9);
\draw[decoration={brace,raise=2pt},decorate] 
    (0,25/9) -- node[left=4pt] {$Kr$} (0,35/9);
%draw squiggles
\draw[->] (5/9, 30/9) .. controls (6/9, 34/9) .. (7/9, 31/9);
\draw[->] (17/9, 17/9) .. controls (12/9, 12/9) and (12/9, 24/9) .. (19/9, 18/9);
\draw[->] (29/9, 6/9) .. controls (4, 6/9) and (4, 0) .. (30/9, 5/9);
\end{tikzpicture}
  \caption{As long as $i$, $j$, and $k$ stay close to $\mathcal B_i$, $\mathcal B_j$, and $\mathcal B_k$ by remaining inside $\mathcal F_i$, $\mathcal F_j$, and $\mathcal F_k$ respectively, there is no potential for interference.}
  \label{fig:Stage 2 diagram}
\end{figure}

It will be helpful to condition on the number of horizontal and vertical random walk
steps taken by each of the tiles $i$, $j$ and $k$. 
Let $H_i$, $H_j$, and $H_k$ be the
number of horizonal 
steps taken by tiles $i$, $j$, and $k$, respectively,
in the $n \ellp^2$ steps, 
with a similar definition for  $V_i$, $V_j$, and $V_k$.
%Recall that in Stage 2, we
%run the torus shuffle ${n \ell^2}$ steps.
Note that each step moves tile $i$ horizontally with probability ${1 \over 2n}$.
It follows that $H_i$ is likely to be about ${\ellp^2 \over 2}$. More precisely,
since $H_i$ is a Binomial($n\ellp^2, {1 \over 2n}$) random variable, its
mean is ${\ellp^2 \over 2} = 2r^2$ and hence
its variance is
at most $2r^2$. Thus
\begin{align}
  \label{cheb}
  \P \Bigl( \abs{H_i - 2r^2}
  \geq r^2 \Bigr) \leq \frac{\var(H_i)}{ \left( {r^2} \right)^2} \leq {2 \over r^2},
\end{align}
by Chebyshev's inequality.
A similar bound holds 
for the other
$H$ and $V$ random variables. Since $r$ goes to infinity with $\ell$, it
follows that when $\ell$ is sufficiently large we have 
$\P(A) \geq \half$, where $A$ is the event that
all the $H$ and $V$ random variables are in the interval $[r , 3r^2]$.

Recall that the
buffer zones have a side length $Kr$,
where 
$K$ is an even integer large enough so that $e^{-{1 \over 3}(K-2)^2} \leq {e^{-4} \over 4 \sqrt{3}}$.
%(The side length of the buffer boxes $K$ times larger than the
%length $cl$ of the
%boxes
%$\boxi$, $\boxj$ and $\boxk$.) 
We claim that the probability that
each tile's random walk takes it to its target position safely
is $\bigomega( \ell^{-6})$.
Since  conditional on the $H$ and $V$ random variables, the
tiles' random walks are independent, this will follow from six applications
of the following lemma about lazy  simple symmetric random walk on $\Z$.  

\begin{lemma}
  \label{loclem}
  Let $K$ be large enough so that
  $e^{-{1 \over 3}(K-2)^2} \leq {e^{-4} \over 4  \sqrt{3}}$.
  Let $N$ and $r$ be positive integers such that $r^2 \leq N \leq 3 r^2$.
Suppose that  $x \in \{-r, \dots, r\}$ and 
let $Y_t$ be a lazy simple random walk on $\Z$
started at $x$.
% Suppose that $K$ is large enough so that $e^{-{1 \over 3}(2K-1)^2} \leq {1 \over 2 e \sqrt{3}}$.
Then 
there is a universal constant $\Crw >0$ such that
if $r$ is sufficiently large and 
  $y   \in \{-r, \dots, r\}$, then 
  \[
    \P( \mbox{$Y_N = y$ and $|Y_t| \leq {Kr \over 2}$ for all $t$ with $0 \leq t \leq N$}) \geq {\Crw \over r} \,.
    \]
  \end{lemma}
  \begin{remark}
    Before proving the lemma, we discuss why it implies the claimed lower bound
    on the probability that each tile goes to its target position safely. 
The first inequality in (\ref{rbounds}) implies that 
the coordinates
of each tile's random walk
start within distance
$r$ of the coordinates of the center 
of its box. 
Furthermore,     with probability at least $\half$
each random walk takes between $r^2$ and $3r^2$ steps.
Finally, we have $r \leq cl$ by (\ref{ellpdef}). 
Hence 
    Lemma \ref{loclem}
  implies that
  \[
    \P(\mbox{tiles $i$, $j$ and $k$ reach their targets safely}) \geq
    {1 \over 2}  \Bigl( {\Crw \over c \ell} \Bigr)^6 = \bigomega( \ell^{-6}),
  \]
  as claimed.
  \end{remark}
\begin{proof}[Proof of Lemma \ref{loclem}.]
%  We prove the lemma assuming that $r$ and $Kr$ are integers. The proof in
%  the general case is similar.
  For $m \in \Z$, define $p_N(m) := \P(Y_N - x = m)$. By
  the local central limit theorem (see Theorem 2.1.1 in \cite{ll}),
  if $\po_N(m) := \frac{1}{\sqrt{\pi N}}e^{-\frac{m^2}{N}}$, then  
there is a universal constant $\xi$ such that
\begin{align}
  \label{lclt}
    \abs{p_N(m) - \po_N(m)} \leq \frac{\xi}{N^{3/2}} \,.
%    \label{loc}
\end{align}
Define $\K := {K \over 2}$.
(Recall that $K$ is an even integer, so $\K$ is an integer.) 
For $m \in \Z$, let $\tau_m := \min\{ t: Y_t = m \}$ be the hitting time of $m$,
%, with a similar definition for $\tau_{-Kr}$,
and define $\tau := \min(\tau_{\K r}, \tau_{-\K r})$. 
By the reflection principle (see Lemma 2.18 in \cite{lpw}),
\begin{eqnarray}
  \P( Y_N = y, \tau_{\K r} \leq N) &=& p_N(\K r + (\K r - y))     \nonumber \\
                              &=& p_N(2 \K r - y),    \label{outlier1}
\end{eqnarray}
and
\begin{eqnarray}
  \P( Y_N = y, \tau_{-\K r}  \leq N) &=& p_N(-\K r - (\K r + y))    \nonumber   \\
                              &=& p_N(-2\K r - y).     \label{outlier2}
\end{eqnarray}
Note also that $p_N(m)$ is a decreasing function of $|m-x|$. Furthermore,
\begin{eqnarray*}
  |2\K r - y - x| &\geq& |2\K r| - |x+y| \\
                &\geq& 2\K r - 2r \\
                &=& Lr,
\end{eqnarray*}
where $L := 2\K-2$; the first  line follows from the triangle inequality
and the second line follows from the triangle inequality
and the fact that $|x| \leq r$ and $|y| < r$. It follows that 
\[
  p_N(2\K r - y) \leq p_N(Lr) \;.
\]
A similar argument shows that
\[
  p_N(-2\K r - y) \leq p_N(Lr) \;.
\]
Combining this with (\ref{outlier1}) and (\ref{outlier2}) gives
\begin{eqnarray}
  \P(Y_N = y, \tau > N) &\geq& p_N(y) - p_N(2 \K r - y) - p_N(-2\K r - y)  \nonumber    \\
                        &\geq& p_N(y) - 2 p_N( L r) \,.   \label{llast}
\end{eqnarray}
Since  $L  = K - 2$ 
the  assumption on the constant $K$ implies that $e^{-L^2/3} \leq {e^{-4} \over 4\sqrt{3}}$.
Thus, by equations (\ref{lclt}), 
(\ref{llast}) and the definition of $\po_N$
we have
\begin{eqnarray}
  \P(Y_N = y, \tau > N) &\geq& {1 \over \sqrt{ \pi N}} \exp\Bigl(-{ (y-x)^2 \over N}\Bigr) - {2 \over \sqrt{\pi N}}
                              \exp \Bigl(-{L^2 r^2 \over N} \Bigr) - {3 \xi \over N^{3/2}}    \nonumber  \\
  &\geq& {1 \over \sqrt{ \pi N}} \exp \Bigl(-{ 4r^2 \over N} \Bigr) - {2 \over \sqrt{\pi N}}
                             \exp \Bigl(-{L^2 r^2 \over N} \Bigr) - {3 \xi \over N^{3/2}}   \label{low}       \,.  
\end{eqnarray}
Replacing each $N$ in (\ref{low}) with $r^2$ or $3r^2$, as appropriate
for a lower bound, 
gives
\begin{eqnarray}
  \P(Y_N = y, \tau > N) &\geq& {1 \over \sqrt{3 \pi}r} e^{-4} - {2 \over \sqrt{\pi} r} e^{-L^2/3} -
                               {3 \xi \over r^3}   \nonumber     \\
                        &=& {1 \over \sqrt{\pi} r} \Bigl[ {e^{-4} \over \sqrt{3}} - 2e^{-L^2/3}  \Bigr] -
 {3 \xi \over r^3}      \,.   \label{square}
                        \end{eqnarray} 
                        Since $e^{-L^2/3} \leq { e^{-4} \over 4 \sqrt{3}}$,
                        the quantity in square brackets in (\ref{square}) is at least ${1 \over 2 e \sqrt{3}}$.
                        Hence for a universal constant $\Crw$,
                        the quantity (\ref{square}) is at least $\Crw/r$ when $r$ is sufficiently large.
%                        This completes the proof of Lemma \ref{loclem}. 

                      \end{proof}
\noindent {\bf Stage 3 Analysis and Final Considerations:}
The analysis for Stage 3 is similar to that for Stage 1.
Since the Markov chain describing the  motion of tiles $i$, $j$ and $k$ is reversible,
the probability that the tiles that end up in positions $i'$, $j'$ and $k'$
at the end of Stage 3 began in boxes $\boxi$, $\boxj$ and $\boxk$, respectively,
at the start of Stage 3,    
is at least ${\eta^6 \over 4}$, by equation (\ref{sone}). 
Combining this with the analysis of Stage 2 shows that the probability
that all three stages are successful (and hence tiles $i$, $j$ and $k$ all
go to their target positions at the end of the $(4\ell^2 + \ellp^2) n$ steps) is at least
${D \over \ell^6}$ for some universal constant $D > 0$.
Since $\ellp < \ell$ the total number of steps is $\C \ell^2 n$ for some
$\C \in [4,5]$, and since $\ellp$ is even the total number of steps is even.

We have proved that there is a constant $l_0$ such that the lemma holds when $l > l_0$. 
We now consider the case where $l \leq l_0$.
Let $\C$ be any real number in $[4,5]$ such that 
$\C\ell^2 n$ is even, %(such a number exists because $n \geq 4$),
and define $T := \C\ell^2 n$.
It is enough to show that the probability that %after $T'$ steps
tiles $i$, $j$ and $k$ go to
their target positions $i'$, $j'$ and $k'$,
respectively, in $T'$ steps is bounded away from zero.

When a step of the torus shuffle moves one of tiles $i, j$ or $k$ we
shall call it an {\it interesting move}.
It is a straightforward exercise to show that any three tiles
in an $l \times l$ box can be moved to
any three positions in the box using at most $14l$ interesting moves.
(Here is a rough sketch of the argument that gives this crude upper bound:
Tile $i$ can be moved to its target position in at most $2l$ interesting moves;
these moves displaced tile $j$ a distance at most $2l$ from its original position. Now
tile $j$ can be moved to its target position at at most $4l$ intereresting moves.
Tile $k$ is now  a distance at most $6l$ from its original position and can be moved
to its target position in at most $8l$ interesting moves.) 

Consider a sequence $S$ of $m$ interesting moves that move
tiles $i$, $j$ and $k$ to
their target positions, %$i'$, $j'$ and $k'$,
where $m \in \{0, 1, \dots, 14l\}$.
Note that, depending on how many rows and columns tiles
$i, j$ and $k$ have in common, there are between $8$ and $12$
possible interesting moves at each step of the torus shuffle. Since the probability of any move
is $\frac{1}{8n}$, each move is interesting with probability between
$\frac{1}{n}$ and $\frac{3}{2n}$.
For $n \geq 4$ and $l \leq l_0$, let  $B_{n,l}$ denote the event that there are at least $m$ interesting moves. 
The total number of moves $T' \geq 4nl^2 \geq 16l^2 \geq m$ (recall $n \geq 4$),
so for any values of $l$ and $n$ we have $\P(B_{n,l}) > 0$. 
Furthermore, the number of interesting moves stochastically dominates a Binomial($4n l^2, {1 \over n}$)
random variable.
Since a Binomial($4n l^2, {1 \over n}$)
random variable converges in distribution to Poisson($4l^2$) as $n \to \infty$ it follows that
for any value of $l$ we have $\liminf_{n \to \infty} \P(B_{n,l}) > 0$. 
Since there are only finitely many $l$ with $l \leq l_0$ it follows that
$\P(B_{n,l})$ is bounded away from $0$ uniformly in $n$ and $l$. 

Therefore, since there are only a bounded number of possible interesting moves at each step, it
follows that $\P(E)$ is bounded away from zero,
where $E$ is the event that:
\begin{itemize}
\item  there are at least $m$ interesting moves in the $T'$ steps;
\item  the first $m$ interesting moves follow the sequence $S$.
\end{itemize}

Finally, the strong Markov property implies that  given
$E$ the conditional probability that 
there are no more
interesting moves after the first $m$ is at least
\begin{equation}
  \label{finaleq}
  \Bigl( 1 - \frac{3}{2n} \Bigr)^{T'} \geq \exp \Bigl( -{15l^2 \over 2} \Bigr)
\geq \exp \Bigl( -{15(l_0)^2 \over 2} \Bigr), 
\end{equation}
where the inequality holds because  $T' \leq 5 n l^2$.
%The quantity (\ref{finaleq}) 
%is
%bounded away from $0$ since $l \leq l_0$. 
It follows that the probability that tiles $i$, $j$ and $k$ go to their target positions
in the $T'$ steps is bounded away from zero, proving the lemma in the case where $l \leq l_0$.

This completes 
the proof of Lemma \ref{mainlemma}.
\end{proof}

\section{Applying Theorem \ref{maintheorem} to bound the mixing time}
\label{payoff}
We saw in Section \ref{fitsin} that two steps of the torus shuffle can be described
in terms of a 3-collision. Therefore, we will actually apply Theorem \ref{maintheorem}
to the shuffle that does two steps of the torus shuffle at each step.
Call this the {\it two-step torus shuffle.}  \\
~\\
We
write
$\pt$ for a product of $t$ i.i.d.~copies of the
two-step torus shuffle. 
The following lemma relates to the decay of
relative entropy after many steps of the two-step torus shuffle.
\begin{lemma}
\label{tl}
There is a universal constant $\fK$ such that
for any random permutation $\mu$ that is equally likely
to be odd or even, 
there is a value of  
$t \in \{1, \dots, 
{\fK n^3}\}$
such that 
\[
\ent(\mu \pt) \leq (1-f(t)) \ent( \mu),
\]
where $f(t) = {\gamma \over \log^2 n} \Bigl( {t \over n^2} \wedge 1\Bigr)$,
for a universal constant $\gamma$. 
\end{lemma}
Before proving Lemma \ref{tl}, we first show how 
it gives the claimed mixing time
bound.
\begin{lemma}
\label{mixingtime}
The mixing time for the torus shuffle is 
$\bigo\Bigl( 
n^3 \log^3n  \Bigr)$.
\end{lemma}
\begin{proof}
It suffices to prove the bound for the two-step torus shuffle.   
We can assume that the chain starts in a state that is equally likely
to be odd or even. (This is already an assumption in the case where $n$ is odd; see Remark \ref{even}.
When $n$ is even, it is true after one step because the chain is lazy and the
row and column rotations are odd permutations.) 
Let $t$ and $f$ be as defined in Lemma \ref{tl}. Then 
\begin{equation}
\lab{rell}
{t \over f(t)} = \gamma^{-1} (\log^2 n)  t \Bigl({n^2 \over t} \vee 1 \Bigr)
= \gamma^{-1} (\log^2 n) ({n^2 \vee t}) \leq T_n,
\end{equation}
where $T_n := \gamma^{-1} (\log^2 n) [ n^2 \vee {\fK n^3}]$.
We can think of  $1/T_n$ as a bound on the long run rate of relative entropy loss
per unit of time.
Lemma \ref{tl} implies that
there is a 
$t_1 \in \{1, \dots, 
{\fK n^3}\}$
such that 
\[
\ent(\pi_1 \cdots \pi_{t_1}) \leq (1-f(t_1)) \ent( \id),
\]
and a 
$t_2 \in \{1, \dots, 
{\fK n^3}\}$
such that 
\[
\ent(\pi_1 \cdots \pi_{t_1 + t_2}) \leq (1-f(t_2)) 
\ent( \pi_1 \cdots \pi_{t_1}),
\]
etc. Continue this way to define $t_3, t_4$, and so on. 
For $j \geq 1$ let $\tau_j = \sum_{i=1}^j t_i$. 
Then
\begin{eqnarray}
\ent( \pi_{(\tau_j)}) &\leq& \Bigl[\prod_{i=1}^j (1 - f(t_j)) \Bigr] 
\ent(\id)\\
&\leq& \exp\Bigl( - \sum_{i=1}^{j} f(t_j) \Bigr) \ent(\id).
\end{eqnarray}
Note that since $t_j \leq T_n f(t_j)$ by equation (\ref{rell}), we have
\[
\tau_j = \sum_{i=1}^{j} t_j
\leq T_n \sum_{i=0}^{j} f(t_j).
\]
It follows that
\begin{eqnarray}
\ent( \pi_{(\tau_j)}) &\leq& 
\exp\Bigl({-\tau_j \over T_n} \Bigr) \ent(\id).
\end{eqnarray}
Since $\ent(\id) = \log (n^2)! \leq  2n^2\log n$, it follows that
if $\tau_j \geq 
T_n \log(16n^2 \log n)$ 
we have $\ent(\pi_{(\tau_j)}) \leq \eighth$
and hence $|| \pi_{(\tau_j)} - \u || \leq \quarter$
by equation (\ref{totent}). It follows that 
the mixing time is $\bigo(T_n \log(16n^2 \log n)) = 
\bigo\Bigl( 
n^3 \log^3n  \Bigr)$.
\end{proof}

We shall now prove Lemma \ref{tl}.
\begin{proof}[Proof of Lemma \ref{tl}]
  Let $m = \lceil \log_2 n \rceil$.
  Let $l_0 = 0$ and for integers $k$ with $1 \leq k \leq m+1$, let $l_k = 2^{k-1} \wedge n$.
Then we can partition  the 
set of locations $\{1, \dots, n^2\}$ into $m+1$ ``upside-down-L-shaped regions'' as 
follows. For $1 \leq k \leq m+1$ define 
$I_k = \{ l_{k-1}^2 + 1 , \dots, l_k^2\}$.
Note that the convention for labeling the positions in the deck described at the beginning of Section
\ref{fitsin} ensures that $I_1 \cup \cdots \cup I_k$ is a box of side length $l_k$. 
%\cap \{0, \dots, n-1\}$.  
Since $\mu$ is equally likely to be odd or even, we can decompose the relative entropy of $\mu$
as
\begin{equation}
  \ENT(\mu)  = \sum_{k=3}^{n} \Et_i, \;        \label{dec} 
\end{equation}
where
$\Et_k = \ENT(\mu^{-1}(k) \given \tailt_{k+1}^{\mu})$.
%By further decomposing the relative entropy into contributions from
%each of the $I_k$, we get
The next step is to find the region $I_k$ that contributes most
to the relative entropy. 
After rewriting the sum in (\ref{dec}), we get
\begin{equation}
\lab{eee1}
\ent(\mu) = \sum_{k=1}^m \sum_{j \in I_k} \Et_j\, .
\end{equation}
Thus, if $\kss$ maximizes $\sum_{j \in I_k} \Et_j$, then
\begin{equation}
  \label{logfun}
\sum_{j \in I_\kss}  \Et_j \geq {1 \over m} \ent(\mu).
\end{equation}
Next, we want to choose the values of $t$ and $T$ to be used in Theorem \ref{maintheorem}
based on the value of $\kss$. It is best to choose values of $t$ and $T$ that are roughly equal to the number
of steps required to "mix up" the tiles in a box of size $l_{\kss}$.

\begin{claim}
\label{bigclaim}
Fix $k$ with $2 \leq k \leq m+1$.
Let $\C \in [4,5]$ be the constant from Lemma \ref{mainlemma}
that corresponds to the case where $l = 3 \ell_k$
(and hence $\C n (3 \ell_k)^2  = 9\C n \ell_k^2$ is an even number). 
Let $T := {9 \over 2} \C n\ell_k^2$ and let
$t := \lceil T + {1 \over 12} n \ell_k^2 \rceil$. 
There is a universal constant $c > 0$ such that if
\begin{align}
    \Boy_x^k := \begin{cases}
        c \left(\frac{t}{n^2} \wedge 1 \right) & \text{if } x \in I_k;\\
        0 & \text{otherwise},
    \end{cases}
\end{align}
then the two-step torus shuffle and the above values of
$t$, $T$, and $\Boy_x^k$
satisfy the assumptions of Theorem \ref{maintheorem}.
\end{claim} 

Before proving Claim \ref{bigclaim}, we show how it implies Lemma \ref{tl}. 
Claim \ref{bigclaim} (with $k = \kss$) and
Theorem \ref{maintheorem} (in the case of $n^2$ cards)
imply  
that there is a universal constant $C$ such that
\begin{eqnarray*}
  \ENT(\mu\pi_{(t)}) - \ENT(\mu) &\leq& \frac{-C}{\log n^2} \sum_{x=1}^{n^2} \Boy_x^\kss E_x    \\
                                 &=& \frac{-Cc}{2\log n} \Bigl(\frac{t}{n^2} \wedge 1 \Bigr)
                                     \sum_{x\in I_{\kss}} E_x \;.  
\end{eqnarray*}
Thus, by equation (\ref{logfun}) and the fact that $m$ is $\bigo(\log n)$, we have
\begin{eqnarray*}
  \ENT(\mu\pi_{(t)}) - \ENT(\mu) &\leq& \frac{-\gamma}{\log^2 n} \Bigl(\frac{t}{n^2}
                                        \wedge 1\Bigr) \ENT(\mu),
\end{eqnarray*}
for a universal constant $\gamma$, which proves Lemma \ref{tl}, assuming Claim \ref{bigclaim}.
It remains to prove Claim \ref{bigclaim}. We will actually prove a slightly stronger
lemma. The lemma gives  a lower bound,
for any triple $(x,y,z)$ of tiles in
a box of side length $l_k$, on the probability that $(x,y,z)$ 
is $(T,t)$ intrined (see Definition \ref{intrined}). 
\begin{lemma}
  \label{lastlemma}
Fix $k$ with $2 \leq k \leq m+1$ and let $T$ and $t$ be as defined in Claim \ref{bigclaim}. 
Let $B_{\ell_k}$ be a box of side length $l_k$ in the $n \times n$ grid. 
Let $x, y$ and $z$ be distinct tiles in $B_{\ell_k}$.
%and 
%let $T_{xyz}$ be the time of the first
%3-collision in the time interval $\{T, \dots, t\}$
%that involves either $x, y$ or $z$.
There is a universal constant $d >0$ such that
\[
  \P( \mbox{$(x , y , z)$ is $(T,t)$ intrined}) \geq
  \left(\frac{t}{n^2} \wedge 1 \right) \frac{d}{\ell_k^4} \,. 
\]
\end{lemma}
Before we prove Lemma \ref{lastlemma} we show why it implies Claim \ref{bigclaim}. 
Suppose
  \[
    x \in I_k, \;\;\;\;\;x \geq 3, \;\;\;\;\;z < x \;.
    \]
  Define
  \[
    S_{xz} := \{1, 2, \dots, x-1\} - \{z\} \;.
    \]
    Since $x \geq 3$ we have $|S_{xz}| \geq 1$ and since $x \in I_k$ we have $x > l_{k-1}^2$ and hence 
    $|S_{xz}| \geq l_{k-1}^2 - 1$. It follows that 
    \begin{eqnarray*}
      |S_{xz}| &\geq& \max(l_{k-1}^2 - 1, 1)  \\
               &\geq& \max\Bigl(\frac{l_k^2}{4} - 1, 1 \Bigr)\;,
    \end{eqnarray*}
    where the second inequality holds because $l_{k-1} \geq l_k/2$. 
Hence 
\begin{eqnarray*}
  \P( \M_2(x) = z, \M_1(x) < x) &=&   \sum_{y \in S_{xz}} \P(\mbox{$(x,y,z)$ are $(T,t)$ intrined}) \\
                                &\geq&           \max \left( \frac{l_k^2}{4} - 1, 1 \right) \cdot \left(\frac{t}{n^2} \wedge 1 \right) \frac{d}{\ell_k^4}
  \\
  &\geq& \frac{c}{x} \left(\frac{t}{n^2} \wedge 1 \right) \;,
\end{eqnarray*}
for a universal constant $c$,
where the second line is by
Lemma \ref{lastlemma}
and the third line holds because $x \in I_k$ and hence $x \leq l_k^2$. This
verifies Claim \ref{bigclaim}.\\
~\\
The last remaining item is to prove Lemma \ref{lastlemma}. 
\begin{proof}[Proof of Lemma \ref{lastlemma}.]
Recall from Section \ref{fitsin} that when three tiles collide
they are in an gamma-shape, as shown in Figure \ref{fig:3-collision}.
Say that  $(x, y, z)$ is {\it intrined nicely}
if $(x, y, z)$ is $(T, t)$ intrined and  
and at time $T_{xyz}$, tile $x$ is in the middle
of the gamma shape. We will actually find a lower bound for the probability
that $(x,y,z)$ is intrined nicely. \\
~\\
Let $P^m\Bigl( (a,b,c), (a', b', c') \Bigr)$ denote the
probability that in $m$ steps of the two-step torus shuffle,
tiles $a$, $b$ and $c$ move to positions $a'$, $b'$ and $c'$,
respectively. 
Let $\ptt^m\Bigl( (a,b,c), (a', b', c') \Bigr)$ denote the
probability that in $m$ steps of the two-step torus shuffle,
tiles $a$, $b$ and $c$ move to positions $a'$, $b'$ and $c'$,
respectively, without being involved in a collision at any point over those $m$ steps.
Note that $\ptt$ is substochastic. It is also symmetric
since for any move where none of three given tiles are involved in a collision
the reverse move happens with the same probability,
Hence
\[
  \ptt^m\Bigl( (a,b,c), (a', b', c') \Bigr) =
  \ptt^m\Bigl( (a',b',c'), (a, b, c) \Bigr) \,.
\]
We have
\begin{eqnarray*}
  & & \P((x,y,z) \text{ is intrined nicely}) \\
  &=& \sum_{x',y',z'}    \sum_{s=0}^{\frac{1}{12}n l_k^2-1}
  \sum_{\Gamma}
  P^T\Bigl((x, y, z) ,(x', y', z')\Bigr) \cdot \ptt^s\Bigl((x',y',z'), \Gamma\Bigr)\cdot \frac{a}{n^2} \,,
\end{eqnarray*}
where the sum is over the possible locations $x', y'$ and $z'$ of tiles $x, y$ and $z$, respectively, after $T$ steps,
the possible values of number of steps $s$ after time $T$ required to get into an gamma-shaped configuration,
and all gamma-shaped configurations $\Gamma$ in the grid; the constant $a = \frac{1}{32}$
because three tiles in an gamma-shaped configuration will collide in the next step with probability
$\frac{1}{32n^2}$.

Let
$\btt_{l_k}$ be the box of side length $3{l_k}$ with the same center as $B_{l_k}$
(or $\btt_{l_k} =$ the whole $n \times n$ torus if $3 l_k > n$).
By considering only positions $x', y'$ and $z'$ in the box $\btt_{l_k}$ and only
gamma-shaped configurations inside $B_{l_k}$ we get the following lower bound:
\begin{eqnarray*}
  & & \P((x,y,z) \text{ is intrined nicely}) \\
  &\geq& \frac{a}{n^2}
  \sum_{x',y',z'\in \btt_{\ell_k}}\sum_{s=0}^{\frac{1}{12}n l_k^2-1}
  \sum_{\Gamma\in B_{\ell_k}^3}
         P^T\Bigl((x, y, z) ,(x', y', z')\Bigr) \cdot \ptt^s\Bigl((x',y',z'), \Gamma\
         \Bigr).
\end{eqnarray*}
Note that $T$ steps of the two-step torus shuffle is like $2T = \C n (3\ell_k)^2$ steps of the
regular torus shuffle (recall that
$T := {1 \over 2} \C n (3\ell_k)^2$).
Since for every term in
the sum above, the positions $x, y,z, x', y'$ and $z'$ are all in the box $\btt_{l_k}$
(a box of side length $3 \ell_k$), Lemma \ref{mainlemma}
implies that
\[
  P^T((x, y, z) ,(x', y', z')) \geq \frac{D}{ (3\ell_k)^6}  \geq \frac{D'}{\ell_k^6},
\]
for a universal constant $D'$. Hence
\begin{eqnarray}
  & & \P((x,y,z) \text{ is intrined nicely})    \\
  &\geq& \frac{a D'}{n^2  l_k^6}
  \sum_{x',y',z'\in \btt_{\ell_k}}\sum_{s=0}^{\frac{1}{12}n l_k^2-1}
  \sum_{\Gamma\in B_{\ell_k}^3}
         \ptt^s((x',y',z'), \Gamma)    \\
   &=& \frac{a D'}{n^2  l_k^6}
              \sum_{s=0}^{\frac{1}{12}n l_k^2-1}
  \sum_{\Gamma\in B_{\ell_k}^3}
       \sum_{x',y',z'\in \btt_{\ell_k}}
  \ptt^s((x',y',z'), \Gamma)     \label{ninetyfive}   \;.
\end{eqnarray}
Consider the inner sum in the righthand side of (\ref{ninetyfive}). By the symmetry
of $\ptt$, we can write it as
\begin{equation}
\label{newsum}
       \sum_{x',y',z'\in \btt_{\ell_k}}
  \ptt^s(\Gamma, (x',y',z')) \;.
\end{equation}
This is the probability that in $s$ steps of the two-step shuffle, the tiles
initially in $\Gamma$
move to positions in $\btt_{l_k}$  without ever being involved
in a 3-collision.
Let $F_s$ be the event that the tiles initially in $\Gamma$
are in $\btt_{l_k}$ after $s$ steps of the two-step torus shuffle. Let
$\tau$ be the time of the first collision involving a tile from $\Gamma$
and let $G_s$ be the event  $\{\tau > s\}$. 
The quantity (\ref{newsum}) is
\begin{equation}
  \label{diffbound}
  \P( F_s \cap G_s) \geq \P(F_s) - \P(G_s^c) \,.
\end{equation}
To bound $\P(F_s)$ we shall need the following lemma.
\begin{lemma}
  \label{dis}
  Fix a tile $\tilet$ and let $Y^s = (Y_1^s, Y_2^s)$ be the displacement of tile $\tilet$
  after $s$ steps of the two-step torus shuffle. If $s \leq \frac{1}{12}nm^2$ then
  \[
    \P( \max( |Y_1^s|, |Y_2^s|) > m) \leq \frac{1}{12} \,.
  \]
\end{lemma}
\begin{proof}
  Since $Y_1^s$ and $Y_2^s$ have the same distribution it suffices to prove that we have
  $\P(|Y_1^s| > m) \leq \frac{1}{24}$ and the lemma then follows from a union bound.
  Note that $|Y_1^s|$ is stochastically dominated by $|W_{2s}|$, where $W_s$ is
  a random walk starting at zero with increment distribution
\[
\left\{\begin{array}{lll}
1 &  \mbox{with probability $\frac{1}{8n}$;}\\
-1 & \mbox{with probability $\frac{1}{8n}$;}\\
         0 &  \mbox{with probability $1 - \frac{1}{4n}$.}
             \end{array}
\right.
\]
         It follows that
         \begin{eqnarray}
           \P( |Y_1^s| > m) &\leq& \P( |W_{2s}| > m)  \nonumber      \\
                            &\leq& \frac{\var(W_{2s})}{m^2}   \nonumber    \\
                            &=& \frac{s}{2nm^2} \;,    \label{cub}
         \end{eqnarray}
                                where the second line is by Chebyshev's inequality. If $s \leq \frac{1}{12}nm^2$ then
                                the quantity (\ref{cub}) is at most $\frac{1}{24}$. 
                                \end{proof}
Since $\Gamma$ is inside $B_{l_k}$, each tile from $\Gamma$ will end up inside $\btt_{l_k}$
(a box of width $3 l_k$)
after $s$
steps provided that its displacement in those $s$ steps
is at most $l_k$. Thus, three applications of Lemma \ref{dis}
imply that $\P(F_s^c) \leq \quarter$ and hence
\begin{equation}
  \label{fbound}
  \P(F_s) \geq {3 \over 4} \,.
  \end{equation}
Next, we shall find an upper bound for $\P(G_s^c)$. Note that in any step of the two-step torus
shuffle, a 3-collision occurs with probability $\frac{1}{8}$.
(The torus shuffle has to be nonlazy in both steps and have both a column move and a row move.)
Since there are 3 tiles
involved in any 3-collision and a total of $n^2$ tiles,
the probability that a particular tile is in a 3-collision in
a given step is $\frac{3}{8n^2}$. Thus, by a union bound the probability that any tile from $\Gamma$
is involved in a collision over $s$ steps is at most 
$\frac{9s}{8n^2}$. It follows that if $s \leq \frac{2n^2}{9}$ then
$\P(G_s^c) \leq \quarter$. Combining this with equations
(\ref{diffbound}) and (\ref{fbound}) gives
$\P(F_s \cap G_s) \geq \half$.

We have shown that when $s \leq \frac{2n^2}{9}$ the inner sum in (\ref{ninetyfive})
is at least $\half$.
Hence there are at least
$\Bigl(\frac{1}{6}n l_k^2 \wedge \frac{2n^2}{9} \Bigr)$
values of 
$s$ such that the inner sum  in (\ref{ninetyfive}) is at least $\half$.
It follows that
\begin{eqnarray}
  & & \P((x,y,z) \text{ is intrined nicely})  \nonumber  \\
  &\geq&
         \frac{a D'}{n^2  l_k^6}
              \Bigl(\frac{1}{6}n l_k^2 \wedge \frac{2n^2}{9} \Bigr) |
         B_{\ell_k}| \cdot \frac{1}{2}    \nonumber  \\
  &\geq&
         \frac{D''}{  l_k^4}
              \Bigl( \frac{l_k^2}{n} \wedge 1 \Bigr),   \label{hun}
\end{eqnarray}
for a universal constant $D'' > 0$. The first inequality holds because
the number of gamma-shaped configurations  inside 
$B_{l_k}$ is at least $|B_{l_k}|$
(since each tile in $B_{l_k}$ is in the middle of at least one gamma-shaped configuration inside 
$B_{l_k}$). The second inequality holds because $|B_{l_k}| = \ell_k^2$.
Finally, note from the definition of $t$
in Claim \ref{bigclaim}
that $t \leq \fK n l_k^2$ for a
universal constant $\fK$. Hence $\frac{t}{n^2} \leq \frac{\fK l_k^2}{n}$. Combining this with
equation (\ref{hun}) gives
\[
  \P((x,y,z) \text{ is intrined nicely})  \geq 
         \frac{d}{  l_k^4}
         \Bigl( \frac{t}{n^2} \wedge 1 \Bigr),
       \]
       for a universal constant $d > 0$, proving the lemma.
\end{proof}
This completes the proof of Lemma \ref{tl}.
\end{proof}
\section{Acknowledgments}
We thank Sukhada Fadnavis for many helpful conversations.

\end{document}